\documentclass[a4paper,12pt]{article}
\title{{\bf Multiple cover formula of generalized DT 
invariants II: Jacobian localizations}}
\date{}
\author{Yukinobu Toda}

\usepackage{makeidx}

\usepackage{latexsym}
\usepackage{amscd}
\usepackage{amsmath}
\usepackage{amssymb}
\usepackage{amsthm}
\usepackage{float}
\usepackage[dvips]{graphicx}

\usepackage[all,ps,dvips]{xy}

\usepackage{array}
\usepackage{amscd}
\usepackage[all]{xy}
\usepackage{makeidx}
\usepackage{latexsym}
\DeclareFontFamily{U}{rsfs}{%
\skewchar\font127}
\DeclareFontShape{U}{rsfs}{m}{n}{%
<-6>rsfs5<6-8.5>rsfs7<8.5->rsfs10}{}
\DeclareSymbolFont{rsfs}{U}{rsfs}{m}{n}
\DeclareSymbolFontAlphabet
{\mathrsfs}{rsfs}
\DeclareRobustCommand*\rsfs{%
\@fontswitch\relax\mathrsfs}
\theoremstyle{plain}
\newtheorem{thm}{Theorem}[section]
\newtheorem{prop}[thm]{Proposition}
\newtheorem{lem}[thm]{Lemma}

\newtheorem{defi}[thm]{Definition}
\newtheorem{rmk}[thm]{Remark}
\newtheorem{cor}[thm]{Corollary}

\newtheorem{prop-defi}[thm]{Proposition-Definition}
\newtheorem{thm-defi}[thm]{Theorem-Definition}
\newtheorem{lem-defi}[thm]{Lemma-Definition}

\newtheorem{conj}[thm]{Conjecture}
\newtheorem{exam}[thm]{Example}

\newdimen\argwidth
\def\db[#1\db]{
 \setbox0=\hbox{$#1$}\argwidth=\wd0
 \setbox0=\hbox{$\left[\box0\right]$}
  \advance\argwidth by -\wd0
 \left[\kern.3\argwidth\box0 \kern.3\argwidth\right]}

\newcommand{\cC}{\mathcal{C}}

\newcommand{\eE}{\mathcal{E}}
\newcommand{\fF}{\mathcal{F}}

\newcommand{\hH}{\mathcal{H}}

\newcommand{\lL}{\mathcal{L}}
\newcommand{\mM}{\mathcal{M}}
\newcommand{\nN}{\mathcal{N}}
\newcommand{\oO}{\mathcal{O}}

\newcommand{\qQ}{\mathcal{Q}}

\newcommand{\xX}{\mathcal{X}}

\newcommand{\Supp}{\mathop{\rm Supp}\nolimits}
\newcommand{\Hom}{\mathop{\rm Hom}\nolimits}

\newcommand{\dL}{\mathbf{L}}

\newcommand{\Hilb}{\mathop{\rm Hilb}\nolimits}

\newcommand{\Pic}{\mathop{\rm Pic}\nolimits}

\newcommand{\id}{\textrm{id}}

\newcommand{\ch}{\mathop{\rm ch}\nolimits}

\newcommand{\td}{\mathop{\rm td}\nolimits}
\newcommand{\Ext}{\mathop{\rm Ext}\nolimits}

\newcommand{\Coh}{\mathop{\rm Coh}\nolimits}

\newcommand{\cneq}{\mathrel{\raise.095ex\hbox{:}\mkern-4.2mu=}}
\newcommand{\eqcn}{\mathrel{=\mkern-4.5mu\raise.095ex\hbox{:}}}

\newcommand{\Aut}{\mathop{\rm Aut}\nolimits}

\newcommand{\SL}{\mathop{\rm SL}\nolimits}

\newcommand{\DT}{\mathop{\rm DT}\nolimits}

\newcommand{\End}{\mathop{\rm End}\nolimits}

\newcommand{\GL}{\mathop{\rm GL}\nolimits}

\begin{document}
\maketitle

\begin{abstract}
The generalized Donaldson-Thomas invariants
counting one dimensional semistable sheaves on 
Calabi-Yau 3-folds are 
conjectured to satisfy a certain multiple cover formula.  
This conjecture is equivalent to Pandharipande-Thomas's
strong rationality conjecture on the generating series of 
stable pair invariants, 
 and its local version
is enough to prove. In this paper, using 
Jacobian localizations and 
 parabolic stable pair invariants introduced
in the previous paper, 
we reduce the conjectural multiple cover formula 
for local curves with at worst nodal singularities
 to the case of local trees of smooth rational curves. 
\end{abstract}
\section{Introduction}
This paper is a sequel of the author's previous paper~\cite{Todpara}, 
and we study the conjectural multiple cover formula of 
generalized Donaldson-Thomas (DT) invariants 
counting one dimensional semistable sheaves on 
Calabi-Yau 3-folds. Our main result is to 
reduce the multiple cover formula 
for local curves with at worst nodal singularities
to that for local trees of $\mathbb{P}^1$. The latter case is easier to study, 
and we actually prove the multiple cover formula 
in some cases using our main result. 
The idea consists of 
twofold: using the notion of parabolic stable pairs
introduced in~\cite{Todpara}, and the localizations with 
respect to the actions of  
Jacobian groups on the moduli spaces of parabolic stable pairs. 

\subsection{Conjectural multiple cover formula}
Let $X$ be a smooth 
projective Calabi-Yau 3-fold over $\mathbb{C}$, i.e. 
\begin{align*}
\bigwedge^3 T_X^{\vee} \cong \oO_X, \quad 
H^1(X, \oO_X)=0. 
\end{align*}
Given data, 
\begin{align*}
n\in \mathbb{Z}, \quad \beta \in H_2(X, \mathbb{Z}),
\end{align*} 
the \textit{generalized DT invariant} is 
introduced by Joyce-Song~\cite{JS}, 
Kontsevich-Soibelman~\cite{K-S}, 
\begin{align}\label{intro:Nnb}
N_{n, \beta} \in \mathbb{Q}. 
\end{align}
The invariant (\ref{intro:Nnb}) 
counts one 
dimensional semistable sheaves $F$ on $X$
satisfying 
\begin{align*}
\chi(F)=n, \quad [F]=\beta. 
\end{align*}
(cf.~Subsection~\ref{subsec:Genera}.)
The above invariant is 
expected to satisfy the following 
multiple cover conjecture: 
\begin{conj}
{\bf\cite[Conjecture~6.20]{JS}, 
\cite[Conjecture~6.3]{Tsurvey}}\label{conj:mult}
We have the following formula, 
\begin{align}\notag
N_{n, \beta}=\sum_{k\ge 1, k|(n, \beta)}
\frac{1}{k^2}N_{1, \beta/k}. 
\end{align}
\end{conj}
The motivation of the above 
conjecture is that it is equivalent to 
Pandharipande-Thomas's (PT) strong rationality 
conjecture~\cite{PT}. 
(See~\cite[Theorem~6.4]{Tsurvey}.)
The PT strong rationality conjecture 
claims the product expansion formula
(called \textit{Gopakumar-Vafa form}) of 
the generating series of
rank one DT type invariants, which should be true if we believe
GW/DT correspondence~\cite{MNOP}. 

There is also a local version of the invariant (\ref{intro:Nnb}) 
and its conjectural multiple cover formula. 
Namely for a one cycle
$\gamma$ on $X$, 
we can associate the invariant, 
\begin{align*}
N_{n, \gamma} \in \mathbb{Q}, 
\end{align*}
which counts one dimensional semistable sheaves 
$F$ on $X$ satisfying 
\begin{align*}
\chi(F)=n, \quad [F]=\gamma, 
\end{align*}
where the second equality is an equality 
as a one cycle. 
The above local invariant is also 
expected to satisfy the multiple cover formula, 
\begin{align}\label{N:mult}
N_{n, \gamma}=\sum_{k\ge 1, k|(n, \gamma)} \frac{1}{k^2}
N_{1, \gamma/k}. 
\end{align}
The local version (\ref{N:mult}) is enough to prove 
Conjecture~\ref{conj:mult}. 
(cf.~\cite[Proposition~4.17]{Todpara}.)
The purpose of this paper is to study 
the conjectural formula (\ref{N:mult})
via Jacobian localization technique. 

\subsection{Main result}
Let $X$ be as before, $\gamma$ a one cycle 
on $X$ and $C \subset X$ the support of $\gamma$. 
The invariant $N_{n, \gamma}$ can be shown to 
be zero if there is an irreducible component 
of $C$ whose geometric genus is 
bigger than or equal to one. 
(cf.~Lemma~\ref{lem:higher}.)
Therefore
in discussing the formula (\ref{N:mult}), 
 we may assume that
$C$ is a rational curve, i.e.  
the normalization of $C$
is a disjoint union of $\mathbb{P}^1$. 
The simple cases are $C=\mathbb{P}^1$, or 
$C$ is a tree of $\mathbb{P}^1$. 
The main result of this paper is to show that, 
when $C$ has at worst  
nodal singularities, 
then
the formula (\ref{N:mult}) 
follows from the same formula for local 
trees of $\mathbb{P}^1$. 
More precisely, 
suppose that $C$ is a rational curve
with at worst nodal singularities, 
and
\begin{align}\label{intro:CUX}
C\subset U \subset X
\end{align} 
a sufficiently small analytic neighborhood 
of $C$ in $X$. 
We consider data, 
\begin{align*}
(C' \subset U') \stackrel{\sigma'}{\to} (C\subset X),
\end{align*}
where $C'$ is a reduced curve, 
$U'$ is a three dimensional complex manifold
and $\sigma'$ is a local immersion. 
The above data is called a \textit{cyclic neighborhood}
if it is given as a composition of cyclic coverings
of $U$. (See Definition~\ref{def:cyclic} for more 
precise definition.)
For any one cycle $\gamma'$ on $U'$
supported on $C'$, we
can similarly construct
the invariant 
\begin{align*}
N_{n, \gamma'}(U') \in \mathbb{Q}.
\end{align*}
(cf.~Subsection~\ref{moduli:cyclic}.)
Our main result is as follows: 
\begin{thm}{\bf [Theorem~\ref{thm:main:cov}]}
\label{thm:main}
Let $X$ be a smooth projective Calabi-Yau 3-fold 
over $\mathbb{C}$, $C \subset X$ 
a reduced rational curve with at worst nodal singularities, 
and $\gamma$ a one cycle on $X$ supported on $C$. 
Suppose that for any cyclic neighborhood $(C' \subset U')
\stackrel{\sigma'}{\to} (C\subset X)$
with $C'$ a tree of $\mathbb{P}^1$, 
the following conditions hold: 
\begin{itemize}
\item The moduli stack of one dimensional semistable 
sheaves on $U'$ is locally written as a
critical locus of some holomorphic function on a complex 
manifold up to some group action. (cf.~Conjecture~\ref{conj:crit}.) 
\item For any one cycle $\gamma'$ on $U'$ with 
$\sigma'_{\ast}\gamma'=\gamma$, 
the invariant $N_{n, \gamma'}(U')$ satisfies the formula 
\begin{align*}
N_{n, \gamma'}(U')=\sum_{k\ge 1, k|(n, \gamma')}
\frac{1}{k^2} N_{1, \gamma'/k}(U'). 
\end{align*}
\end{itemize}
Then the invariant $N_{n, \gamma}$ satisfies the formula (\ref{N:mult}). 
\end{thm}
There are several situations 
in which 
the cyclic neighborhood $C'\subset U'$ satisfies
the assumptions in 
Theorem~\ref{thm:main}, 
e.g. $C'$ is a chain of super rigid 
rational curves in $U'$. 
Roughly speaking, 
we will give the following 
applications in Section~\ref{sec:apply}:
\begin{itemize}
\item If $\gamma$ is supported on 
an irreducible rational curve with one node, 
or a circle of $\mathbb{P}^1$, 
we explicitly compute the invariant $N_{n, \gamma}$. 
(cf.~Theorem~\ref{thm:typeI}.)
\item We prove the local multiple cover formula 
of $N_{n, \gamma}$ 
if $\gamma=p[C]$ for an irreducible 
rational curve $C$ with at worst nodal 
singularities, and $p$ is a prime number. 
(cf.~Theorem~\ref{prop:prime}.)
\item We give some evidence of the 
conjecture in~\cite[Conjecture~1.3]{TodK3}
on the Euler characteristic invariants of 
local K3 surfaces. 
(cf.~Theorem~\ref{thm:K3}.)
\end{itemize}
The first and the second 
applications will be given under
a certain assumption 
on an analytic neighborhood of a 
one cycle $\gamma$. 
(cf.~Definition~\ref{def:rigid:surface}.)

\subsection{Idea for a local curve with one node}
Here we explain the idea of the proof of Theorem~\ref{thm:main}
in a simple example. 
Let 
\begin{align*}
C\subset X
\end{align*} be an irreducible 
rational curve with one node $x\in C$. 
Suppose that a one cycle $\gamma$ on $X$ is supported on $C$. 
Then for any analytic neighborhood $U$
as in (\ref{intro:CUX}), 
the Jacobian group 
$\Pic^0(U)$ acts on the moduli space which 
defines $N_{n, \gamma}$. 
If we take $U$ to be homotopically equivalent to $C$, 
then $\Pic^0(C) \cong \mathbb{C}^{\ast}$
is considered to be a subgroup of 
$\Pic^0(U)$. 
So we would like to apply $\Pic^0(C)$-localization 
on the invariant $N_{n, \gamma}$. 
In order to 
see this, we need to find $\Pic^0(C)$-fixed 
semistable sheaves on $U$ supported on $C$.  

If we take $U$ as above, then we have 
\begin{align*}
\pi_1(C) \cong \pi_1(U) \cong \mathbb{Z}. 
\end{align*}
Hence if we take the universal covering space of $U$, 
\begin{align}\label{univ:U}
f_U \colon 
\widetilde{U} \to U, 
\end{align}
then $\widetilde{U}$ admits a $\mathbb{Z}$-action, 
and it contains the universal cover 
of $C$ denoted by $\widetilde{C}$. 
A key observation is that 
a stable sheaf on $U$ supported on $C$ is 
$\Pic^0(C)$-fixed if and only if 
it is 
a push-forward of some sheaf on $\widetilde{U}$
supported on $\widetilde{C}$, 
which is unique up to $\mathbb{Z}$-action on $\widetilde{U}$. 

The universal
cover $\widetilde{C} \to C$ is described in the 
following way. Let 
\begin{align*}
\mathbb{P}^1 \cong 
C^{\dag} \to C
\end{align*}
be the normalization 
and $x_1, x_2 \in C^{\dag}$
the preimage at the node $x\in C$. 
We take an infinite number of copies of 
$\{C^{\dag}, x_1, x_2\}$, denoted by 
\begin{align*}
\{C_i, x_{1, i}, x_{2, i}\}, \quad i\in \mathbb{Z}. 
\end{align*}
Then $\widetilde{C}$ is an infinite chain of smooth rational curves, 
\begin{align*}
\widetilde{C} = \cdots \cup C_{-1} \cup C_{0} \cup C_1\cdots
\cup C_{i}^{} \cup C_{i+1}^{} \cup \cdots, 
\end{align*}
where $C_{i}$ and $C_{i+1}$
are attached along $x_{2, i}$ and $x_{1, i+1}$. 
(See Figure~\ref{fig:one}.)

For instance, let us look at the invariant 
$N_{0, 2C}$. 
By the above argument, we may expect the formula, 
\begin{align}\label{expect}
N_{n, 2C} =\sum_{i\ge 0} N_{n, C_{0} +C_i}(\widetilde{U}).  
\end{align}
Now by the assumptions in Theorem~\ref{thm:main}, we obtain 
\begin{align*}
N_{0, 2C_0}(\widetilde{U}) &=N_{1, 2C_0}(\widetilde{U})+
\frac{1}{4}N_{1, C_0}(\widetilde{U}), \\
N_{0, C_0 +C_1}(\widetilde{U}) &=N_{1, C_0+C_1}(\widetilde{U}). 
\end{align*}
The above localization argument also implies 
$N_{1, C_0}(\widetilde{U})=N_{1, C}$, and 
it is also easy to see $N_{n, C_0+C_i}(\widetilde{U})=0$ for $i\ge 2$. 
Thus we obtain 
\begin{align*}
N_{0, 2C}=N_{1, 2C}+ \frac{1}{4}N_{1, C}, 
\end{align*}
which is nothing but the desired formula (\ref{N:mult})
for $\gamma=2C$. 
 This picture is quite similar to 
the multiple cover formula for genus zero 
Gromov-Witten invariants of a local nodal curve with one 
node~\cite{BKL}. 

\begin{figure}[htbp]
 \begin{center}
  \includegraphics[width=140mm]{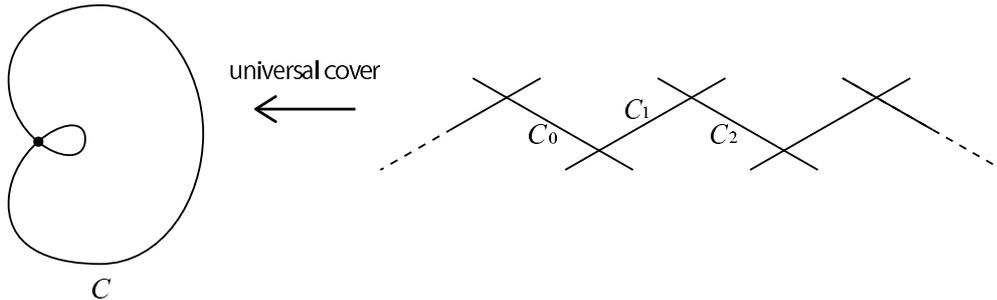}
 \end{center}
 \caption{Universal cover $C \leftarrow \widetilde{C}$}
 \label{fig:one}
\end{figure}

\subsection{Parabolic stable pairs}
In the previous subsection, we 
explained the idea of the multiple cover formula 
in a simple example.  
However it is not
obvious to realize the story 
there directly,  
especially the formula (\ref{expect}) 
seems to be hard to deduce.   
The issue is that, since the definition of $N_{n, \gamma}$
involves Joyce's log stack function~\cite{Joy4},
denoted by $\epsilon_{n, \gamma}$ 
in Subsection~\ref{subsec:Genera}, 
the above localization argument seems to be very hard to apply. 
Namely, we 
have to compare the contribution of 
$\epsilon_{n, \gamma}$ on 
the $\mathbb{C}^{\ast}$-fixed points
with that on the universal cover.
But
to do this, we also have to `localize'
the product structure on the Hall algebra, 
which seems to require a new technique. 
In order to overcome this technical difficulty, 
we use the idea of \textit{parabolic stable pairs}, 
introduced in the previous paper~\cite{Todpara}. 
By definition, a parabolic stable pair 
consists of a pair, 
\begin{align*}
(F, s), \quad s \in F \otimes \oO_H, 
\end{align*}
where 
$F$ is a one dimensional semistable sheaf on $X$, 
$H$ is a fixed divisor in $X$, satisfying 
a certain stability condition. 
(cf.~Definition~\ref{defi:para}.)
In~\cite{Todpara}, we constructed invariants counting 
parabolic stable pairs, and showed that Conjecture~\ref{conj:mult}
is equivalent to a certain product expansion formula 
of the generating series of parabolic stable pair invariants. 
The moduli space
of parabolic stable pairs is a scheme, (not a stack,)
 and 
$\Pic^0(C)$ also acts on the moduli space 
of (local) parabolic stable pairs. 
There is no technical difficulty in 
applying 
$\Pic^0(C)$-localizations
to parabolic stable pair invariants, and 
the arguments similar to the previous
subsection work for parabolic stable pairs. 
 
When the one cycle $\gamma$ on $X$ is 
supported on a nodal curve which has 
more than one nodes, then 
its universal covering space
is much more complicated. Instead of taking 
the universal cover, we take 
cyclic neighborhoods and proceed the 
induction argument. 
Combining the above ideas, 
(Jacobian localizations, parabolic stable pairs, 
induction via cyclic neighborhoods,)
we are able to prove Theorem~\ref{thm:main}.

\subsection{Acknowledgement}
The author is grateful to Richard Thomas, 
Jacopo Stoppa 
 for valuable discussions on 
the subject of this paper, 
and Kentaro Nagao for pointing out the 
reference~\cite{Bergh}. 
The author would like to thank the 
Isaac Newton Institute and its program 
`Moduli Spaces', 
during which a part of this work was done. 
This work is supported by World Premier 
International Research Center Initiative
(WPI initiative), MEXT, Japan. This work is also supported by Grant-in Aid
for Scientific Research grant (22684002), 
and partly (S-19104002),
from the Ministry of Education, Culture,
Sports, Science and Technology, Japan.

\section{Multiple cover formula of generalized DT invariants}
In this section, we recall
(generalized) DT invariants on Calabi-Yau 3-folds 
and 
 the
conjectural multiple cover formula. 
In what follows, $X$ is a smooth projective Calabi-Yau 
3-fold over $\mathbb{C}$, i.e. 
\begin{align*}
\bigwedge^{3}T_{X}^{\vee} \cong \oO_X, \quad
H^1(X, \oO_X)=0. 
\end{align*}
We fix an ample line bundle $\oO_X(1)$
and set $\omega=c_1(\oO_X(1))$. 
Below we say a coherent sheaf $F$ on $X$
\textit{$d$-dimensional}
if the support of $F$ is $d$-dimensional. 
\subsection{Semistable sheaves}
Let us recall the notion of one dimensional
 $\omega$-semistable 
sheaves on $X$. 
They are defined by the notion of slope:
for a one dimensional coherent sheaf $F$, its slope is defined by 
\begin{align*}
\mu_{\omega}(F) \cneq \frac{\chi(F)}{[F] \cdot \omega}. 
\end{align*}
Here $\chi(F)$ is the holomorphic Euler characteristic 
of $F$ and $[F]$ is the fundamental one cycle associated to 
$F$, defined by 
\begin{align}\label{onecycle}
[F] = \sum_{\eta} (\mathrm{length}_{\oO_{X, \eta}}
F) \overline{\{\eta \}}. 
\end{align}
In the above sum, $\eta$ runs all the 
 codimension 
two points in $X$.
\begin{defi}\label{defi:semi}
A one dimensional coherent sheaf $F$ on $X$ is 
$\omega$-(semi)stable if for any 
subsheaf $0\neq F' \subsetneq F$, we have the inequality, 
\begin{align*}
\mu_{\omega}(F') <(\le) \mu_{\omega}(F). 
\end{align*}
\end{defi}
Note that any 
one dimensional $\omega$-semistable sheaf $F$ is 
pure, i.e. there is no zero dimensional 
subsheaf in $F$. 
Also we say that $F$ is \textit{strictly $\omega$-semistable} 
if $F$ is $\omega$-semistable but not $\omega$-stable. 
For the detail of (semi)stable sheaves, see~\cite{Hu}.

\subsection{DT invariants}
Let us take data, 
\begin{align}\label{data:bn}
n \in \mathbb{Z}, \quad 
\beta \in H_2(X, \mathbb{Z}).
\end{align}
The (generalized) DT invariant is the $\mathbb{Q}$-valued
invariant, 
\begin{align}\label{Nnb}
N_{n, \beta} \in \mathbb{Q}, 
\end{align}
counting one dimensional $\omega$-semistable sheaves $F$ on $X$
satisfying
\begin{align}\label{Fbn}
[F]=\beta, \quad \chi(F)=n. 
\end{align}
Here by an abuse of notation, we denote by 
$[F]$ the homology class of the one cycle (\ref{onecycle}).

The invariant (\ref{Nnb}) is defined in the following way. 
Let 
\begin{align}\label{moduli:M}
M_n(X, \beta)
\end{align}
be the coarse moduli space of 
one dimensional $\omega$-semistable 
sheaves $F$ on $X$ satisfying (\ref{Fbn}). 
There are some criterions for the moduli 
space (\ref{moduli:M}) to be fine. 
For instance suppose that the following condition holds:
\begin{align}\label{gcd}
\mathrm{g.c.d.}(\omega \cdot \beta, n)=1, 
\end{align}  
e.g. $n=1$. Then  
there is no strictly $\omega$-semistable 
sheaf on $X$ satisfying (\ref{Fbn}),
and (\ref{moduli:M}) is a fine 
projective scheme over $\mathbb{C}$.
In this case, the moduli space (\ref{moduli:M})
carries a symmetric perfect obstruction theory, 
hence the zero dimensional virtual cycle~\cite{Thom}. 
\begin{defi}
If the condition (\ref{gcd}) holds, then 
we define $N_{n, \beta}$ to be 
\begin{align*}
N_{n, \beta}=\int_{[M_n(X, \beta)]^{\rm{vir}}} 1 \in \mathbb{Z}. 
\end{align*}
\end{defi}
Another way to define $N_{n, \beta}$ is to use 
Behrend's constructible function~\cite{Beh}. 
Recall that for any $\mathbb{C}$-scheme $M$, 
Behrend constructs a canonical constructible function, 
\begin{align*}
\nu \colon M \to \mathbb{Z}, 
\end{align*}
such that if $M$ carries a symmetric perfect obstruction theory, 
then we have 
\begin{align*}
\int_{[M]^{\rm{vir}}} 1 &= \int_{M} \nu d\chi, \\
&=\sum_{m \in \mathbb{Z}} m \cdot \chi(\nu^{-1}(m)). 
\end{align*}
Hence by using the Behrend function 
$\nu$ on $M_n(X, \beta)$, the 
invariant (\ref{Nnb}) can also be 
also expressed as 
\begin{align}\label{intM}
N_{n, \beta}=\int_{M_n(X, \beta)} \nu d\chi. 
\end{align}
\subsection{Generalized DT invariants}\label{subsec:Genera}
In a general choice of (\ref{data:bn}), 
 the condition (\ref{gcd}) 
may not hold, and there may be strictly $\omega$-semistable
sheaves $F$ satisfying (\ref{Fbn}). 
 In this case, the invariant (\ref{Nnb}) 
is one of
\textit{generalized DT invariants}
introduced by Joyce-Song~\cite{JS}
and Kontsevich-Soibelman~\cite{K-S}. 
It requires sophisticated techniques on Hall algebras
of coherent sheaves
to define them, and we need
some more preparations for this.  
Since we will not need the detail of the 
definition of (\ref{Nnb}) in a general case, we 
just give a rough explanation. 

A strictly $\omega$-semistable sheaf has non-trivial 
automorphisms, and we need to involve the contributions of 
the automorphism groups with the invariant (\ref{Nnb}). 
For this purpose, we need to work with the moduli stack, 
\begin{align}\label{stack:M}
\mM_n(X, \beta), 
\end{align}
which parameterizes $\omega$-semistable one dimensional 
sheaves $F$ satisfying (\ref{Fbn}). 
The stack (\ref{stack:M}) is known to be an Artin 
stack of finite type over $\mathbb{C}$. 

The Behrend functions on $\mathbb{C}$-schemes 
naturally extend to constructible 
functions on Artin stacks of finite type
over $\mathbb{C}$. (cf.~\cite[Proposition~4.4]{JS}.)
However the stack (\ref{stack:M})
may have stabilizer groups whose Euler characteristic 
are zero, e.g. $\GL(2, \mathbb{C})$. 
Hence the integration of the Behrend function 
(\ref{intM}), replacing $M_n(X, \beta)$
by $\mM_n(X, \beta)$, does not make sense. 
The idea of the definition of generalized
DT invariant is that, instead of working with
the stack (\ref{stack:M}), we should work 
with the `logarithm' of (\ref{stack:M}) 
in the Hall algebra of coherent sheaves, 
denoted by $H(X)$. 

The algebra $H(X)$ is, as a $\mathbb{Q}$-vector 
space, spanned by the isomorphism classes of 
symbols, 
\begin{align*}
[ \rho \colon \xX \to \cC oh(X)]. 
\end{align*}
Here $\xX$ is an Artin stack of finite type 
with affine geometric stabilizers, and 
$\cC oh(X)$ is the stack of all 
the coherent sheaves on $X$. 
There is an associative $\ast$-product 
on $H(X)$ based on Ringel Hall algebras. 
For the detail, see~\cite[Theorem~5.2]{Joy2}. 

The stack (\ref{stack:M}) is considered to be 
an element of $H(X)$, by 
regarding it as an open substack of $\cC oh(X)$, 
\begin{align*}
\delta_{n, \beta} \cneq \left[\mM_n(X, \beta) \hookrightarrow
 \cC oh(X)  \right]
 \in H(X).  
\end{align*}
The `logarithm' of $\delta_{n, \beta}$, 
denoted by $\epsilon_{n, \beta} \in H(X)$, is defined by the rule, 
\begin{align*}
\sum_{n/\omega \cdot \beta =\mu}
\epsilon_{n, \beta}
= \log \left( 1+ \sum_{n/\omega \cdot \beta=\mu}
\delta_{n, \beta}   \right),
\end{align*}
for any $\mu \in \mathbb{Q}$
in a certain completion of the algebra $(H(X), \ast)$.
In other words, $\epsilon_{n, \beta}$ is given by 
\begin{align*}
\epsilon_{n, \beta} =
\sum_{l\ge 1}\frac{(-1)^{l-1}}{l} 
\sum_{\begin{subarray}{c}
\beta_1, \cdots, \beta_l \in H_2(X, \mathbb{Z}), \\
n_1, \cdots, n_l \in \mathbb{Z}, \\
n_i/\omega \cdot \beta_i=n/\omega \cdot \beta
\end{subarray}}
\delta_{n_1, \beta_1} \ast \cdots \ast \delta_{n_l, \beta_l}. 
\end{align*}
The above sum is easily shown to be a finite sum. 

The important fact is that $\epsilon_{n, \beta}$ is 
supported on `virtual indecomposable sheaves'. 
Roughly speaking this implies that,
modulo some 
relations in $H(X)$,
the element $\epsilon_{n, \beta}$ is written as 
\begin{align*}
\epsilon_{n, \beta}=\sum_{i} a_i
[\rho_i \colon [M_i/\mathbb{C}^{\ast}] \to \cC oh(X)],
\end{align*}
where $a_i \in \mathbb{Q}$, 
 $M_i$ are quasi-projective varieties on which 
$\mathbb{C}^{\ast}$ act trivially. 
The invariant (\ref{Nnb}) is then defined by 
the weighted Euler characteristics of $M_i$, 
weighted by the Behrend function $\nu$ on 
$\cC oh(X)$ pulled back by $\rho_i$. Namely, 
$N_{n, \beta}$ is defined by 
\begin{align}\label{skip}
N_{n, \beta} \cneq 
-\sum_{i} a_i \int_{M_i} \rho_i^{\ast} \nu d\chi. 
\end{align}
Here we need to change the sign due to the appearance
of the trivial $\mathbb{C}^{\ast}$-action.

We have skipped lots of details in the above
definition of (\ref{Nnb}). For more detail, we refer~\cite{JS}. 
Also see~\cite[Section~4]{Tsurvey} for a more direct explanation.   
\begin{rmk}\label{rmk:omega}
In priori, we need to choose an ample divisor $\omega$
to define $N_{n, \beta}$. However
it can be shown that $N_{n, \beta}$ does not depend on 
a choice of $\omega$. 
(cf.~\cite[Theorem~6.16]{JS}.)
\end{rmk}

\subsection{Local generalized DT invariants}
There is also a local version of 
(generalized) DT invariant, 
which we explain below.  
Let us fix a reduced curve $C$ in $X$, 
 \begin{align*}
i \colon C \hookrightarrow X, 
\end{align*}
with irreducible components $C_1, \cdots, C_N$. 
Then a one cycle $\gamma$ on $X$ supported on $C$
 is identified with an element of $H_2(C, \mathbb{Z})$,  
\begin{align*}
\gamma \in H_2(C, \mathbb{Z}) 
\cong \bigoplus_{i=1}^{N} \mathbb{Z}[C_i]. 
\end{align*}
Suppose that $\beta=i_{\ast}\gamma$
and $n\in \mathbb{Z}$ satisfies the condition (\ref{gcd}). 
Then we have the fine moduli space (\ref{moduli:M}), and 
the closed subscheme, 
\begin{align}\label{closed:M}
M_n(C, \gamma) \subset M_n(X, \beta), 
\end{align}
corresponding to $\omega$-stable sheaves $F$
satisfying 
\begin{align}\label{Fgn}
[F]=\gamma, \quad \chi(F)=n. 
\end{align}
Here $[F]=\gamma$ is an equality as a one cycle on $X$.
Then the local DT invariant is defined by 
\begin{align}\label{loc:DT:Nng}
N_{n, \gamma} \cneq \int_{M_n(C, \gamma)} \nu d\chi. 
\end{align}
Here $\nu$ is the Behrend function on $M_n(X, \beta)$
restricted to $M_n(C, \gamma)$. We remark that 
$\nu$ may not coincide with the Behrend function on $M_n(C, \gamma)$. 

Even if $(n, \beta)$ does not satisfy the condition (\ref{gcd}), 
we can similarly define the local generalized DT invariant, 
\begin{align}\label{Nng}
N_{n, \gamma} \in \mathbb{Q}, 
\end{align}
counting one dimensional $\omega$-semistable sheaves $F$
satisfying (\ref{Fgn}). 
Instead of using the stack (\ref{stack:M}), 
we use the substack,
\begin{align}\label{M(Cg)}
\mM_n(C, \gamma) \subset \mM_n(X, \beta), 
\end{align}
parameterizing one dimensional $\omega$-semistable 
sheaves $F$ on $X$ satisfying (\ref{Fgn}).
We can similarly take the logarithm of the substack (\ref{M(Cg)})
in the Hall algebra $H(X)$, and the invariant (\ref{Nng})
is defined by integrating  
the Behrend function on $\cC oh(X)$ over it.
See~\cite[Subsection~4.4]{Todpara} for some more detail. 
Similarly to $N_{n, \beta}$, the local invariant 
$N_{n, \gamma}$ also does not depend on $\omega$. 
(cf.~Remark~\ref{rmk:omega}.)

\subsection{Multiple cover formula}\label{subsec:mult}
As we discussed in the previous 
subsections, the invariant (\ref{Nnb}) 
is an integer if the condition (\ref{gcd})
is satisfied. 
In particular, 
for $\beta \in H_2(X, \mathbb{Z})$, 
we have the $\mathbb{Z}$-valued 
invariant, 
\begin{align*}
N_{1, \beta} \in \mathbb{Z}. 
\end{align*}
The above invariant is introduced by Katz~\cite{Katz}
as a sheaf theoretic definition of 
genus zero Gopakumar-Vafa invariant. 
On the other hand if
$(n, \beta)$ does not satisfy the condition (\ref{gcd}), 
then $N_{n, \beta}$ may not be an integer and 
hence does not coincide with $N_{1, \beta}$. 
However the invariants $N_{n, \beta}$ for 
$n \neq 1$ are conjectured to be 
related to $N_{1, \beta}$
via the multiple cover formula: 
\begin{conj}
{\bf\cite[Conjecture~6.20]{JS}, 
\cite[Conjecture~6.3]{Tsurvey}}\label{conj:mult2}
We have the following formula, 
\begin{align}\label{form:mult}
N_{n, \beta}=\sum_{k\ge 1, k|(n, \beta)}
\frac{1}{k^2}N_{1, \beta/k}. 
\end{align}
\end{conj}
In~\cite[Theorem~6.4]{Tsurvey},
it is shown that the above conjecture 
is equivalent to Pandharipande-Thomas's
strong rationality conjecture~\cite[Conjecture~3.14]{PT}. 
We refer~\cite[Section~6]{Tsurvey} 
for discussions on strong rationality conjecture 
and its relation to Conjecture~\ref{conj:mult2}. 

For a 
reduced curve $C\subset X$, $n\in \mathbb{Z}$
 and 
$\gamma \in H_2(C, \mathbb{Z})$, 
we have the
local (generalized) DT invariants 
as in (\ref{Nng}). 
The local version of the above conjecture is 
also similarly formulated:
\begin{conj}{\bf\cite[Conjecture~4.13]{Todpara}}
\label{conj:mult:loc}
For $n\in \mathbb{Z}$ and $\gamma \in H_2(C, \mathbb{Z})$, 
 we have the formula, 
\begin{align}\label{form:mult:loc}
N_{n, \gamma}=\sum_{k\ge 1, k|(n, \gamma)}
\frac{1}{k^2}N_{1, \gamma/k}. 
\end{align}
\end{conj}
As shown in~\cite[Corollary~4.18]{Todpara}, the local 
multiple cover formula 
 is enough to show the global
multiple cover formula: 
\begin{lem}{\bf\cite[Corollary~4.18]{Todpara}}\label{lem:glo:loc}
For $n\in \mathbb{Z}$ and $\beta \in H_2(X, \mathbb{Z})$, 
suppose that the 
formula (\ref{form:mult:loc}) holds 
for any reduced curve $i \colon C \hookrightarrow X$
and $\gamma \in H_2(C, \mathbb{Z})$
with $\beta=i_{\ast} \gamma$. 
Then $N_{n, \beta}$ satisfies the formula (\ref{form:mult}). 
\end{lem} 

As we discussed in the Introduction, 
our purpose is to study Conjecture~\ref{conj:mult:loc}
in terms of Jacobian localizations and parabolic stable 
pair invariants, which we recall in the next subsection. 
\subsection{(Local) parabolic stable pair theory}
The notion of parabolic stable pairs is
introduced in~\cite{Todpara}. It is determined 
by fixing a divisor, 
\begin{align*}
H \in \lvert \oO_X(h) \rvert, 
\end{align*}
for some $h>0$. 
In what follows, we say a one cycle $\gamma$
on $X$ intersects with 
$H$ transversally if it
satisfies $\dim H \cap \gamma=0$. 
Equivalently, any irreducible component in $\gamma$ is not 
contained in $H$. 
\begin{defi}\label{defi:para}
For a fixed divisor $H$ on $X$ as above, 
a parabolic stable pair is defined to be a pair 
\begin{align}\label{para:pair}
(F, s), \quad s \in F \otimes \oO_{H},
\end{align}
such that the following conditions are satisfied. 
\begin{itemize}
\item The sheaf $F$ is a one dimensional 
$\omega$-semistable sheaf on $X$. 
\item The one cycle $[F]$ intersects with $H$ transversally. 
\item For any surjection $F \stackrel{\pi}{\twoheadrightarrow} F'$
with $\mu_{\omega}(F)=\mu_{\omega}(F')$, we have 
\begin{align*}
(\pi \otimes \oO_{H})(s) \neq 0. 
\end{align*}
\end{itemize}
\end{defi} 
The moduli space of parabolic 
stable pairs $(F, s)$ satisfying 
$[F]=\beta$, $\chi(F)=n$ 
is denoted by 
\begin{align}\label{moduli:para}
M_n^{\rm{par}}(X, \beta). 
\end{align}
By~\cite[Theorem~2.10]{Todpara}, 
if $H$ satisfies an additional condition 
given in~\cite[Lemma~2.9]{Todpara}, 
then the moduli space (\ref{moduli:para})
is a projective scheme
even if $(n, \beta)$ does not satisfy the condition (\ref{gcd}). 
In the case that 
$H$ does not satisfy the condition in~\cite[Lemma~2.9]{Todpara}, 
the moduli space (\ref{moduli:para})
is at least a quasi-projective variety. 
(cf.~\cite[Remark~2.13]{Todpara}.) 

Suppose that a reduced one dimensional subscheme 
$i\colon C\hookrightarrow X$ satisfies $\dim H\cap C=0$.
Then for any $\gamma \in H_2(C, \mathbb{Z})$
with $\beta=i_{\ast}\gamma$, we have the subscheme, 
\begin{align}\label{MnCg}
M_n^{\rm{par}}(C, \gamma) \subset M_n^{\rm{par}}(X, \beta), 
\end{align}
corresponding to parabolic stable pairs $(F, s)$
with $F$ supported on $C$, 
$[F]=\gamma$ as a one cycle on $X$ and $\chi(F)=n$. 

Let 
\begin{align*}
\nu_{M} \colon M_n^{\rm{par}}(X, \beta) \to \mathbb{Z},
\end{align*}
be the Behrend's constructible
 function~\cite{Beh} on $M_n^{\rm{par}}(X, \beta)$. 
The local parabolic stable pair invariant is 
defined in the following way. 
\begin{defi}\label{defi:para:DT}
For 
$\gamma \in H_2(C, \mathbb{Z})$, 
we define
 $\DT_{n, \gamma}^{\rm{par}} \in \mathbb{Z}$ to be
\begin{align}\label{DTpar:loc}
\DT_{n, \gamma}^{\rm{par}} \cneq 
\int_{M_n^{\rm{par}}(C, \gamma)} \nu_M d\chi. 
\end{align}
\end{defi}
Here as in the local DT theory, we use the Behrend function on 
$M_n^{\rm{par}}(X, \beta)$, not on $M_n^{\rm{par}}(C, \gamma)$, to define
the local invariant. 

\subsection{Multiple cover formula via parabolic stable pairs}\label{subsec:multiple:via}
In~\cite{Todpara}, we established a relationship
between (local) parabolic stable pair invariants and 
(local) generalized DT invariants. As a result, 
conjectures in Subsection~\ref{subsec:mult} can be 
translated into a formula relating
(local)
parabolic stable pair invariants 
and (local) DT invariants, which are both 
integer valued. 

Let $C \subset X$ be a reduced curve,
with irreducible components $C_1, \cdots, C_N$, 
which intersects with $H$ transversally. 
As in Definition~\ref{defi:para:DT}, we have 
the local parabolic stable pair invariants w.r.t. $H$. 
 For each $\mu \in \mathbb{Q}$, we set 
the generating series $\DT^{\rm{par}}(\mu, C)$
to be 
\begin{align*}
\mathrm{DT}^{\rm{par}}(\mu, C)
\cneq 1+ \sum_{\begin{subarray}{c}
n\in \mathbb{Z}, \ 
\gamma \in H_2(C, \mathbb{Z})_{>0}, \\
n/\omega \cdot \gamma=\mu
\end{subarray}}
\mathrm{DT}_{n, \gamma}^{\rm{par}}
q^n t^{\gamma}.
\end{align*}
Here $H_2(C, \mathbb{Z})_{>0} \subset H_2(C, \mathbb{Z})$
is defined by 
\begin{align*}
H_2(C, \mathbb{Z})_{>0}
\cneq \left\{ \sum_{i=1}^{N} a_i [C_i] :
a_i \ge 0 \right\} \setminus \{0\} 
\subset H_2(C, \mathbb{Z}). 
\end{align*}
The statement of Conjecture~\ref{conj:mult:loc} can 
be translated into a  product expansion 
formula (\ref{par:prod}) of $\DT^{\rm{par}}(\mu, C)$ below:
\begin{prop}{\bf\cite[Proposition~4.5]{Todpara}}\label{prop:translate}
We have the formula (\ref{form:mult:loc}) for any 
$(n, \gamma) \in \mathbb{Z} \oplus H_2(C, \mathbb{Z})_{>0}$
with $n/\omega \cdot \gamma=\mu$
if and only if the following formula holds, 
\begin{align}\label{par:prod}
\mathrm{DT}^{\rm{par}}(\mu, C)
=\prod_{\begin{subarray}{c}
\gamma \in H_2(C, \mathbb{Z})_{>0}, \\
n/\omega \cdot \gamma=\mu
\end{subarray}}
\left(1-(-1)^{\gamma \cdot H}
q^n t^{\gamma}  \right)^{(\gamma \cdot H)N_{1, \gamma}}.
\end{align}
\end{prop}
If we are interested in the formula (\ref{form:mult:loc})
for a specified $(n, \gamma)$, then it is enough to check 
the formula (\ref{log:form}) below: 
let us take the logarithm of $\DT^{\rm{par}}(\mu, C)$
and write
\begin{align*}
\log 
\DT^{\rm{par}}(\mu, C)=
\sum_{\begin{subarray}{c}
\gamma \in H_2(C, \mathbb{Z})_{>0}, \\
n/\omega \cdot \gamma=\mu
\end{subarray}}
\widehat{\DT}_{n, \gamma}^{\rm{par}}q^n t^{\gamma}. 
\end{align*} 
Note that $\widehat{\DT}^{\rm{par}}_{n, \gamma}$ 
is written as 
\begin{align}\label{DT:hat}
\widehat{\DT}^{\rm{par}}_{n, \gamma}
=\sum_{l\ge 1} \frac{(-1)^{l-1}}{l}
\sum_{\begin{subarray}{c}
\gamma_1+ \cdots +\gamma_l=\gamma, 
\gamma_i \in H_2(C, \mathbb{Z})_{>0}\\
n_1 + \cdots +n_l=n, n_i \in \mathbb{Z}, \\
n_i/\omega \cdot \gamma_i=n/\omega \cdot \gamma
\end{subarray}}
\prod_{i=1}^{l}\DT_{n_i, \gamma_i}^{\rm{par}}. 
\end{align}
Then we should have the formula, 
\begin{align}\label{log:form}
\widehat{\DT}_{n, \gamma}^{\rm{par}}
=\sum_{k\ge 1, k|(n, \gamma)} \frac{(-1)^{\gamma \cdot H -1}}{k^2}
(\gamma \cdot H) N_{1, \gamma/k}. 
\end{align}
Here the RHS of (\ref{log:form})
is $q^n t^{\gamma}$-coefficient of the RHS of (\ref{par:prod}).
Note that (\ref{log:form}) is a relationship 
between $\mathbb{Z}$-valued invariants.  
(cf.~\cite[Corollary~4.18]{Todpara}.)

\subsection{Jacobian actions on the moduli space of parabolic 
stable pairs}\label{subsec:Jact}
In this subsection, we discuss 
Jacobian actions on the moduil space of
parabolic stable pairs. 

Let $i \colon C \hookrightarrow X$ be a reduced 
curve, and $H \subset X$ a divisor which 
intersects with $C$ transversally. 
Let us take
\begin{align*}
n \in \mathbb{Z}, \
\gamma \in H_2(C, \mathbb{Z}),
\end{align*}
and set $\beta=i_{\ast} \gamma \in H_2(X, \mathbb{Z})$. 
Let $U$ be
a complex analytic neighborhood of $C$ in $X$, 
\begin{align*}
C \subset U \subset X. 
\end{align*}
Then we have the analytic open subset of
the moduli space (\ref{moduli:M}), 
\begin{align*}
M_n(U, \beta) \subset M_n(X, \beta), 
\end{align*}
corresponding to $\omega$-semistable one dimensional 
sheaves $F$ with $\Supp(F) \subset U$. 

Let $\Pic^{0}(U)$ be the group of 
line bundles on $U$, whose restriction to
any projective curve in $U$ has degree zero. 
Then we have the action of $\Pic^0(U)$
on $M_n(U, \beta)$ via 
\begin{align*}
L \cdot F =F\otimes L,
\end{align*}
for $L \in \Pic^0(U)$ and 
$F \in M_n(U, \beta)$. 
The $\Pic^0(U)$-action preserves the closed subscheme, 
\begin{align*}
M_n(C, \gamma) \subset M_n(U, \beta),
\end{align*}
where the LHS is given in (\ref{closed:M}). 

Let us consider parabolic stable pairs 
w.r.t. the divisor $H$ as above. 
Similarly, we have the analytic open subspace, 
\begin{align*}
M_n^{\rm{par}}(U, \beta) \subset M_n^{\rm{par}}(X, \beta), 
\end{align*}
corresponding to parabolic stable pairs 
$(F, s)$ with $\Supp(F) \subset U$. 
Let $\widehat{\Pic^0}(U)$ be the group defined by 
\begin{align}\label{hat:pic}
\widehat{\Pic^0}(U) \cneq 
\{ (L, \phi) : 
L \in \Pic^0(U), \lambda \colon \oO_{H \cap U} \stackrel{\cong}{\to}
 \oO_{H\cap U} \otimes L \}. 
\end{align}
Note that the forgetting 
map $\widehat{\Pic^0}(U) \ni (L, \phi) \mapsto L \in \Pic^{0}(U)$
is surjective if $U$ is a sufficiently small analytic 
neighborhood of $C$. 
The group $\widehat{\Pic^0}(U)$ acts on 
$M_n^{\rm{par}}(U, \beta)$ via 
\begin{align}\label{LpFs}
(L, \lambda) \cdot (F, s)=
(F\otimes L, s'), 
\end{align}
where $s'$ is the image of $s$ by the isomorphism, 
\begin{align*}
\id_F \otimes \lambda \colon 
F\otimes \oO_H \stackrel{\cong}{\to}
F \otimes L \otimes \oO_H.
\end{align*}
The above isomorphism makes sense 
since $F$ is supported on $U$. 
Obviously the action (\ref{LpFs}) preserves the 
closed subspace, 
\begin{align}\label{sub:par}
M_n^{\rm{par}}(C, \gamma) \subset M_n^{\rm{par}}(U, \beta), 
\end{align}
where the LHS is given by the LHS of (\ref{MnCg}). 
Also the action (\ref{LpFs}) is compatible
with the $\Pic^{0}(U)$-action on $M_n(U, \beta)$ 
and the forgetting morphisms, 
\begin{align*}
M_n^{\rm{par}}(U, \beta) \ni 
(F, s) &\mapsto F \in  M_n(U, \beta), \\
\widehat{\Pic^0}(U) \ni (L, \lambda) 
&\mapsto L \in \Pic^0(U). 
\end{align*}

\begin{rmk}
By Chow's theorem, the complex analytic 
spaces $M_n(U, \beta)$, $M_n^{\rm{par}}(U, \beta)$
are regarded as the moduli spaces of
$\omega$-semistable sheaves, parabolic stable pairs 
on $U$ in an analytic sense respectively. Hence the above 
$\Pic^0(U)$, $\widehat{\Pic^0}(U)$-actions make sense. 
\end{rmk}

\subsection{Local multiple cover formula in simple cases}
Finally in this section, we 
discuss some situations in which 
the formula (\ref{form:mult:loc}) is
easily proved. 
Let $C \subset U \subset X$ be as in the previous subsection.
In the following lemma, 
which is partially obtained in~\cite[Proposition~6.19]{JS}, 
we reduce the problem 
to the case that $C$ has only rational irreducible 
components. 
\begin{lem}\label{lem:higher}
Let $C_1, \cdots, C_N$ be the irreducible 
components of $C$, and 
take 
\begin{align*}
\gamma=\sum_{i=1}^{N} a_i[C_i]
\in H_2(C, \mathbb{Z})_{>0}.
\end{align*}
Suppose that there is $1\le i \le N$
such that $a_i>0$ and 
the geometric genus of $C_i$ is 
is bigger than or equal to one. 
Then for any $n\in \mathbb{Z}$, we have 
$N_{n, \gamma}=0$. 
In particular, the formula (\ref{form:mult:loc}) holds. 
\end{lem}
\begin{proof}
Let us consider the $\Pic^{0}(U)$ action on 
$M_n(U, \gamma)$ as in the previous subsection. 
Note that any point $p \in M_n(U, \gamma)$
is represented by an $\omega$-semistable 
sheaf $F$ which is a direct sum of $\omega$-stable sheaves. 
If $p$ is fixed by 
the action of $\lL \in \Pic^{0}(U)$, 
then we have $F \otimes \lL \cong F$. 
For the normalization 
$f \colon C_i^{\dag} \to C_i$, we have 
\begin{align*}
f^{\ast}(F|_{C_i}) \cong f^{\ast}(F|_{C_i}) \otimes f^{\ast}(\lL|_{C_i}). 
\end{align*}
Taking the determinant of the both sides, we have
\begin{align*}
f^{\ast}(\lL|_{C_i})^{\otimes k} \cong \oO_{C_i^{\dag}}, 
\end{align*}
for some $k \in \mathbb{Z}_{\ge 1}$. 
Given $\gamma$, there is only a finite 
number of possibilities for the above $k$, 
say $k_1, \cdots, k_l$. 
Since $\Pic^0(C_i^{\dag})$ is a complex
torus of positive dimension, 
we can find a subgroup 
\begin{align}\label{sub:S1}
S^1 \subset \Pic^0(C_i^{\dag}),
\end{align}
which does not pass through any $k_i$-torsion 
points for $1\le i\le l$. 
On the other hand, we have the composition of the
 pull-backs
\begin{align}\label{Pic:rest}
\Pic^{0}(U) \to \Pic^0(C_i) \to \Pic^{0}(C_i^{\dag}). 
\end{align}
Since $U$ is a sufficiently small analytic neighborhood of $C$, 
an argument similar to Subsection~\ref{subsec:JacC} below
shows that 
both of the arrows in (\ref{Pic:rest}) are surjective. 
Furthermore, the same argument
also easily shows that
 there is a
subgroup $S^1 \subset \Pic^0(U)$
which restricts to the 
subgroup (\ref{sub:S1})
under the restriction (\ref{Pic:rest}). 
Then the action of $\Pic^0(U)$ on $M_n(C, \gamma)$
restricted to $S^1 \subset \Pic^0(U)$ is free, 
hence the same localization argument of~\cite[Proposition~6.19]{JS} shows 
the vanishing $N_{n, \gamma}=0$. 
\end{proof}
Next we discuss the case that 
the class $\gamma \in H_2(C, \mathbb{Z})$ is primitive, 
i.e. $\gamma$ is not a multiple of some other element
of $H_2(C, \mathbb{Z})$.  
\begin{lem}\label{N:primitive}
Suppose that $\gamma \in H_2(C, \mathbb{Z})$ is 
primitive. Then $N_{n, \gamma}$
does not depend on $n$. 
In particular, the formula (\ref{form:mult:loc}) holds. 
\end{lem}
\begin{proof}
Let $\Coh_C(X)$
be the category of coherent sheaves on $X$
supported on $C$. 
We first generalize $\mu_{\omega}$-stability 
to twisted stability on $\Coh_C(X)$. 
Let $C \subset U \subset X$ be a sufficiently small 
analytic neighborhood, and 
take an element 
\begin{align*}
B+i\omega \in H^2(U, \mathbb{C}), 
\end{align*}
such that $\omega|_{C}$ is ample. 
For a one dimensional sheaf
$F \in \Coh_C(X)$, 
we set $\mu_{B, \omega}(F) \in \mathbb{Q}$ to be
\begin{align*}
\mu_{B, \omega}(F) \cneq \frac{\chi(F)-[F] \cdot B}{[F] \cdot \omega}. 
\end{align*}
Similarly to Definition~\ref{defi:semi}, we have the 
notion of $\mu_{B, \omega}$-stability on $\Coh_{C}(X)$,  
called \textit{twisted stability}. 
As in the case of $\mu_{\omega}$-stability,  
we can construct the moduli stack 
$\mM_n(C, \gamma, B+i\omega)$ parameterizing 
$\mu_{B, \omega}$-semistable objects $F \in \Coh_{C}(X)$
with $[F]=\gamma$ and $\chi(F)=n$, 
and the generalized DT invariant defined by the above moduli stack.  
The same argument of~\cite[Theorem~6.16]{JS} shows that the 
resulting invariant does not depend on a choice 
of $B$ and $\omega$, thus coincides with $N_{n, \gamma}$. 

Let $C_1, \cdots, C_N$
be the irreducible components of $C$, 
and set $\gamma=\sum_{i=1}^{N} a_i[C_i]$.  
Since $\gamma$ is primitive, 
we have $\mathrm{g.c.d.}(a_1, \cdots, a_N)=1$. 
Hence we can find $m_1, \cdots, m_N \in \mathbb{Z}$ 
such that 
$\sum_{i=1}^{N} m_i a_i=1$. 
Let us take divisors $D_1, \cdots, D_N$ on 
$U$ such that $D_i \cdot C_j =\delta_{ij}$, 
and set $D=\sum_{i=1}^{N} m_i D_i$. 
(This is possible since $U$ is taken to be a
sufficiently small analytic neighborhood
of $C$ in $X$.)
Then we have the isomorphism of stacks, 
\begin{align*}
\mM_n(C, \gamma, B+i\omega) \stackrel{\cong}{\to}
\mM_{n+1}(C, \gamma, B-D+i\omega), 
\end{align*}
given by $F \mapsto F\otimes \oO_U(D)$. 
Since the generalized DT invariants do not depend on 
$B$ and $\omega$, the above isomorphism of stacks 
immediately implies $N_{n, \gamma}=N_{n+1, \gamma}$
for all $n\in \mathbb{Z}$. 
\end{proof}

\section{Cyclic covers of nodal rational curves}
Let $X$ be a smooth projective Calabi-Yau 3-fold 
over $\mathbb{C}$. 
In what follows, we fix a connected
reduced curve $C$
and an embedding,  
\begin{align*}
i\colon C \hookrightarrow X, 
\end{align*}
satisfying the following conditions. 
\begin{itemize}
\item Any irreducible component of $C$ has geometric genus zero.
(We call such a curve as a \textit{rational curve}.)
\item The curve $C$ has at worst nodal singularities. 
\end{itemize}
Note that for our purpose, 
we can always assume the 
first condition by Lemma~\ref{lem:higher}.
The geometric genus of $C$ is defined by 
\begin{align*}
g(C) \cneq \dim H^1(C, \oO_C). 
\end{align*}
When $g(C)=0$, 
each irreducible component of $C$
is $\mathbb{P}^1$, and 
the dual graph of $C$ is 
simply connected. 
(See Subsection~\ref{subsec:Jac:C} below.)
In this case, we say $C$ is a \textit{tree of} $\mathbb{P}^1$. 
Below we assume that $g(C)>0$. 

We also fix an ample divisor $H$ in $X$ which 
is smooth, connected, and 
intersects with $C$ transversally
at non-singular points of $C$. 

\subsection{Jacobian group of $C$}\label{subsec:Jac:C}
We first recall the description of the 
Jacobian group of
a nodal curve $C$. 
Suppose that $C$ has
$\delta_n$-nodes and has $\delta_c$-irreducible components. 
Let us take the normalization of $C$, 
\begin{align*}
f \colon C^{\dag} \to C. 
\end{align*}
We have the exact sequence of sheaves, 
\begin{align*}
0 \to \oO_C \to f_{\ast} \oO_{C^{\dag}} \to \mathbb{C}^{\oplus \delta_n} \to 0. \end{align*}
By the long exact sequence of cohomologies, we obtain the isomorphism, 
\begin{align*}
\mathbb{C}^{\delta_n -\delta_c +1} \cong H^1(C, \oO_C).
\end{align*}
In particular the arithmetic genus $g(C)$
satisfies 
\begin{align*}
g(C) =\delta_n -\delta_c +1. 
\end{align*}
Combining the above argument with
the standard exact sequence, 
\begin{align*}
0 \to \mathbb{Z} \to \oO_C \to \oO_C^{\ast} \to 1, 
\end{align*}
we can easily see that $H^1(C, \mathbb{Z})$
generates $H^1(C, \oO_C)$
as a $\mathbb{C}$-vector space, 
\begin{align}\label{gen:H1}
H^1(C, \mathbb{Z}) \otimes_{\mathbb{Z}} \mathbb{C}
\cong H^1(C, \oO_C). 
\end{align}
Hence we have the isomorphisms, 
\begin{align}\notag
(\mathbb{C}^{\ast})^{g(C)} &\cong
H^1(C, \oO_C)/H^1(C, \mathbb{Z}) \\
\label{isom:Pic}
& \cong \Pic^0(C).
\end{align}
We can interpret the above isomorphism 
in terms of the dual graph $\Gamma_C$
associated to $C$, determined
in the following way:  
\begin{itemize}
\item The vertices and edges of $\Gamma_C$
correspond to 
irreducible components of $C$ and 
nodal points of $C$ respectively. 
\item For an edge $e$ corresponding to a nodal 
point $x \in C$, it connects vertices $v_1$, $v_2$ if 
the corresponding 
irreducible components $C_1$, $C_2$ satisfies 
$p \in C_1 \cap C_2$. 
(Note that the case of $v_1=v_2$
corresponds to a self node.)
\end{itemize}
Then $\Gamma_C$ is a connected graph 
satisfying $b_1(\Gamma_C)=g(C)$. 
We can interpret (\ref{isom:Pic}) as 
the isomorphism, 
\begin{align}\label{isom:pic}
H^1(\Gamma_C, \mathbb{Z}) \otimes_{\mathbb{Z}}\mathbb{C}^{\ast}
\cong \Pic^0(C). 
\end{align}
The isomorphism (\ref{isom:pic}) can be constructed in
the following way. For an oriented loop
$\alpha$ in $\Gamma_C$, we choose 
an edge $e \subset \alpha$ so that $\Gamma_C \setminus \{ e\}$
is still connected. Let $x \in C$ 
be a nodal point corresponding to the edge $e$. 
Let 
$v_1, v_2$ be vertices in $\Gamma_C$
connected by $e$
so that $e$ starts from $v_1$ and ends at $v_2$.
Let $C_1$, $C_2$ be the irreducible components of $C$
which correspond to $v_1$, $v_2$ respectively. 
We partially normalize $C$ at $x$, and obtain 
$C^{\dag}_x$, 
\begin{align}\label{part:norm}
f_x \colon 
C^{\dag}_x \to C. 
\end{align}
For $i=1, 2$, let $C^{\dag}_i$
be the irreducible component 
of $C^{\dag}_x$ which is mapped to 
$C_i$ by $f_x$. 
The preimage of $x$ by $f_x$ is denoted by 
\begin{align*}
x_i \in C^{\dag}_i, \quad i=1, 2. 
\end{align*}
(In case of $v_1=v_2$, 
i.e. $x\in C$ is a self node, 
we need to 
fix a correspondence between an orientation of $e$ and 
a numbering of two points $f_x^{-1}(x)$.)
For $z \in \mathbb{C}^{\ast}$, 
we glue the trivial line bundle on $C_x^{\dag}$
by the isomorphism at $x_i$, 
\begin{align*}
 \oO_{C^{\dag}_x} \otimes k(x_1) 
\ni a \mapsto za \in 
\oO_{C^{\dag}_x} \otimes k(x_2). 
\end{align*}
The 
above gluing procedure produces a line bundle 
on $C$. The 
resulting line bundle is 
independent of a choice of $e$, and 
denoted by 
$L_{\alpha, z}$. 
The isomorphism (\ref{isom:pic}) is given by 
sending $\alpha \otimes z$ to $L_{\alpha, z}$.

\subsection{Jacobian group of an analytic neighborhood of $C$}
\label{subsec:JacC}
Let $C \subset X$ be as in the previous subsection. 
We take an analytic 
open neighborhood $U$ of $C$ in $X$, 
\begin{align*}
C \subset U \subset X. 
\end{align*}
If we take $U$ sufficiently small so that 
it is homotopically equivalent to $C$, we have the 
commutative diagram of exact sequences,  
\begin{align*}
\xymatrix{
0 \ar[r] & H^1(U, \mathbb{Z}) \ar[r]\ar[d]^{\cong} 
& H^1(U, \oO_U) \ar[r]\ar[d] & \Pic^0(U) \ar[r]\ar[d] & 1 \\
0 \ar[r] & H^1(C, \mathbb{Z}) \ar[r] & H^1(C, \oO_C) \ar[r] & 
\Pic^0(C) \ar[r] & 1. 
}
\end{align*}
Here all the vertical morphisms are pull-backs 
with respect to the inclusion $C \hookrightarrow U$. 
Let $W$ be the sub $\mathbb{C}$-vector space 
of $H^1(U, \oO_U)$ generated by 
$H^1(U, \mathbb{Z})$. 
Then the above commutative diagram and the isomorphism (\ref{gen:H1}) 
implies that
\begin{align}\label{emb:pic}
\Pic^0(C) \cong W/H^1(U, \mathbb{Z}) \subset \Pic^0(U). 
\end{align}
By composing the isomorphism (\ref{isom:pic}) and the 
embedding (\ref{emb:pic}), 
we obtain the embedding, 
\begin{align}\label{emb:G}
H^1(\Gamma_C, \mathbb{Z}) \otimes_{\mathbb{Z}} \mathbb{C}^{\ast}
\hookrightarrow \Pic^0(U). 
\end{align}
The embedding (\ref{emb:G}) 
can be described in the following way. 
For each nodal point $x\in C^{\rm{sing}}$, let us 
fix a norm $\lVert \ast \rVert$ on an analytic neighborhood of $x$ 
in $X$, and set 
\begin{align}\label{open:V}
V_{x}(\epsilon) =\{ x' \in X : \lVert x' -x \rVert <\epsilon \},
\end{align}
for $\epsilon>0$. 
We construct the following open subsets in $U$, 
\begin{align*}
U_{x}(\epsilon) &\cneq U \cap V_{x}(\epsilon), \\
U_{x}'(\epsilon) &\cneq U \setminus \overline{V}_{x}(\epsilon). 
\end{align*}
Then for $0<\epsilon'<\epsilon$,
the collection  
$\{U_{x}(\epsilon), U_{x}'(\epsilon')\}$
is an open cover of $U$. 

Suppose that $x\in C$ is contained in 
two irreducible components $C_1$, $C_2$. 
Then we have 
\begin{align*}
U_{x}(\epsilon) \cap U_x'(\epsilon') \cap C
= \coprod_{j=1}^{2}
(U_x(\epsilon) \cap U_x'(\epsilon') \cap C_j), 
\end{align*}
and each $U_x(\epsilon) \cap U_x'(\epsilon) \cap C_j$ is homeomorphic to an 
open annulus in $\mathbb{C}$.
 Hence if $U$ is chosen to be sufficiently small, 
we have the decomposition, 
\begin{align}\label{decomp}
U_x(\epsilon) \cap U_x'(\epsilon') =W_1 \coprod W_2, 
\end{align}
such that we have 
\begin{align*}
U_x(\epsilon) \cap U_x'(\epsilon') \cap C_j \subset W_j. 
\end{align*}

Let $\alpha$ be an oriented loop in $\Gamma_C$
and take $z \in \mathbb{C}^{\ast}$. 
As in the previous subsection, we 
take an edge $e \subset \alpha$
such that $\Gamma_C \setminus \{e\}$ is connected.
If $e$ corresponds to the node $x\in C$, 
we construct the line bundle 
on $U$
by gluing the trivial line bundles on 
$U_{x}(\epsilon)$ and $U_{x}'(\epsilon')$ by the 
isomorphism, 
\begin{align*}
\oO_{W_1} \oplus \oO_{W_2} &\ni (a_1, a_2) \\
&\mapsto (z a_1, a_2) \in \oO_{W_1} \oplus \oO_{W_2}. 
\end{align*}
Here we have used the identification by (\ref{decomp}),
\begin{align*}
\oO_{U_{x}(\epsilon) \cap U_x'(\epsilon')}=
\oO_{W_1} \oplus \oO_{W_2}. 
\end{align*}
The resulting line bundle is independent of $e$, 
and denoted by $\lL_{\alpha, z}$. 
Note that $\lL_{\alpha, z}$ restricts to
a line bundle $L_{\alpha, z}$ on $C$, 
constructed 
in the previous subsection. 
When $x$ is a self node, $\lL_{\alpha, z}$ can 
be similarly constructed by replacing 
$C_1$, $C_2$ by analytic branches of $C$ near $x$. 
The embedding (\ref{emb:G}) is given by sending 
$\alpha \otimes z$ to $\lL_{\alpha, z}$. 

If we fix an oriented loop $\alpha$ in 
$\Gamma_{C}$, we have the complex subtorus, 
\begin{align}\label{ctorus}
\mathbb{C}^{\ast} \subset \Pic^{0}(U), 
\end{align}
given by the embedding 
$z \mapsto \lL_{\alpha, z}$. 
Here recall that, in the first part of 
this section, we took a divisor $H \subset X$
so that it intersects with $C$ at non-singular 
points on $C$. 
Therefore if we furthermore fix an edge 
$e \subset \alpha$, the 
above construction of $\lL_{\alpha, z}$
yields a canonical isomorphism, 
\begin{align*}
\phi_{\alpha, e, z} \colon 
\oO_{H \cap U} 
\stackrel{\cong}{\to}
\oO_{H\cap U} \otimes 
\lL_{\alpha, z}.
\end{align*}
This implies that the embedding (\ref{ctorus})
lifts to a group homomorphism, 
\begin{align}\label{lift}
\mathbb{C}^{\ast} \hookrightarrow
 \widehat{\Pic^0}(U), 
\end{align}
given by 
\begin{align*}
\alpha \otimes z \mapsto (\lL_{\alpha, z}, \phi_{\alpha, e, z}),
\end{align*}
which is an embedding. Here 
$\widehat{\Pic^0}(U)$ 
is defined by (\ref{hat:pic}). 
Combined with the argument
in Subsection~\ref{subsec:Jact}, we have the action 
of the subtorus (\ref{lift})
on the moduli space of parabolic
stable pairs $M_n^{\rm{par}}(U, \beta)$, 
which restricts to the action on 
the subspace (\ref{sub:par}).

\subsection{Cyclic covers of $U$}\label{subsec:cyclic}
Let us fix a loop $\alpha \subset \Gamma_{C}$ and 
consider the subtorus (\ref{ctorus}), 
$z \mapsto \lL_{\alpha, z}$. 
The root of unity $z=e^{2\pi i/m}$ corresponds 
to the line bundle, 
\begin{align*}
\lL_{\alpha, m} \cneq \lL_{\alpha, e^{2\pi i/m}}
\in \Pic^0(U), 
\end{align*}
which is an $m$-torsion element, i.e. 
there is an isomorphism of line bundles, 
\begin{align}\label{iso:bun}
\psi_{\alpha, m} \colon 
\oO_{U} \stackrel{\cong}{\to} \lL_{\alpha, m}^{\otimes m}.
\end{align}
Given an isomorphism $\psi_{\alpha, m}$ as above, we can construct 
the complex manifold $\widetilde{U}_{\alpha, m}$ as follows: 
\begin{align*}
\widetilde{U}_{\alpha, m} \cneq \{ y \in \lL_{\alpha, m} : 
y^{\otimes m} =\psi_{\alpha, m}(1) \}. 
\end{align*} 
Here we have regarded the line bundle 
$\lL_{\alpha, m}$ as its total space on $U$. 
The projection $\lL_{\alpha, m} \to U$
induces the morphism, 
 \begin{align}\label{sigma:U}
\sigma_{\alpha, m} \colon \widetilde{U}_{\alpha, m} \to U,
\end{align}
which is a covering map of covering degree $m$. 
By taking the pull-back of $C \subset U$ by 
$\sigma_{\alpha, m}$, we obtain the 
$m$-fold \'{e}tale cover of $C$, 
\begin{align}\label{etale:C}
\sigma_{\alpha, m}|_{\widetilde{C}_{\alpha, m}} \colon 
\widetilde{C}_{\alpha, m} \to C. 
\end{align}
Note that $\widetilde{U}_{\alpha, m}$ is a complex manifold containing 
$\widetilde{C}_{\alpha, m}$, 
and satisfies
\begin{align*}
\bigwedge^{3}T_{\widetilde{U}_{\alpha, m}}^{\vee} 
\cong \oO_{\widetilde{U}_{\alpha, m}}.
\end{align*}

The $m$-fold cover (\ref{etale:C})
is determined by the isomorphism (\ref{iso:bun}) 
restricted to $C$, which is described in the following way. 
Let us choose an edge $e \subset \alpha$ corresponding 
to the node $x\in C$, and take a partial normalization
$C_x^{\dag}$ as in (\ref{part:norm}). 
Let $x_1, x_2 \in C_x^{\dag}$ be the preimages
of $x$. 
We take the $m$-copies of
$\{C_x^{\dag}, x_1, x_2\}$, 
\begin{align*}
\{ C_{x, i}, x_{1, i}, x_{2, i}\}, \quad i \in \mathbb{Z}/m\mathbb{Z}. 
\end{align*}
Then $\widetilde{C}_{\alpha, m}$ is given by 
\begin{align}\label{Cxi}
\widetilde{C}_{\alpha, m} =\bigcup_{i \in \mathbb{Z}/m\mathbb{Z}}
C_{x, i}^{}, 
\end{align}
where $C_{x, i}$ and $C_{x, i+1}$ are 
glued at $x_{1, i}$ and $x_{2, i+1}$. 
(See Figure~\ref{fig:two}.)

\begin{figure}[htbp]
 \begin{center}
  \includegraphics[width=100mm]{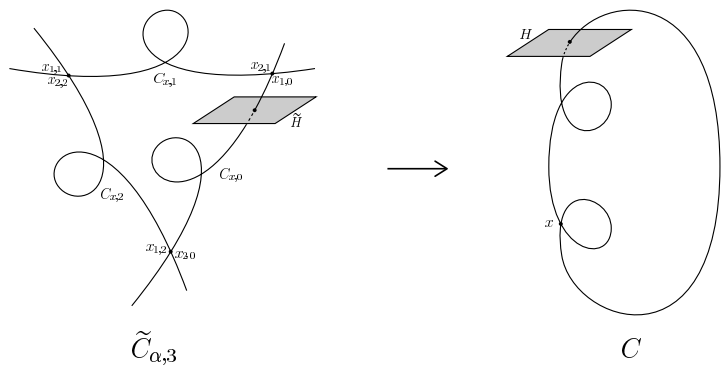}
 \end{center}
 \caption{3-fold cover of a rational curve with two nodes}
 \label{fig:two}
\end{figure}

\subsection{Coherent sheaves on $\widetilde{U}$}\label{subsec:Coh}
Let us consider the $m$-fold cover (\ref{sigma:U})
constructed in the previous subsection. 
In this subsection, we 
investigate coherent sheaves on the covering 
(\ref{sigma:U}). 

In what follows, we fix $\alpha$ and $m$, so we omit 
these symbols in the notation, e.g. we write 
$\widetilde{U}_{\alpha, m}$, $\widetilde{C}_{\alpha, m}$
as $\widetilde{U}$, $\widetilde{C}$, etc.
We first discuss the category of coherent sheaves on 
the covering $\widetilde{U}$. By the construction, 
the $\oO_U$-algebra $\sigma_{\ast}\oO_{\widetilde{U}}$ is given by 
\begin{align*}
\sigma_{\ast}\oO_{\widetilde{U}}
\cong 
\oO_{U} \oplus \lL^{-1} \oplus \cdots \oplus \lL^{-m+1}. 
\end{align*}
The algebra structure on the RHS is given by, 
for $(a, b) \in \lL^{-i} \times \lL^{-j}$,  
\begin{align*}
(a, b)
\mapsto \left\{ \begin{array}{cc}
a\otimes b, & i+j<m, \\
\psi (a\otimes b), & i+j\ge m. 
\end{array}   \right. 
\end{align*}
Here $\psi=\psi_{\alpha, m}$ is the isomorphism given in (\ref{iso:bun}). 
By the above isomorphism of $\oO_{U}$-algebras, 
it is easy to show the following lemma: 
\begin{lem}\label{lem:forward}
The push-forward functor $\sigma_{\ast}$ identifies
the category $\Coh(\widetilde{U})$ with the category of pairs, 
\begin{align*}
(F, \phi_F), \quad \phi_F \colon F \to F \otimes \lL, 
\end{align*}
where $F \in \Coh(U)$ and $\phi_F$ is a morphism in 
$\Coh(U)$ satisfying 
\begin{align}\label{cond:cyc}
\overbrace{\phi_F \circ \cdots \circ \phi_F}^{m}
=\mathrm{id}_F \otimes \psi, 
\end{align}
as morphisms $F \to F \otimes \lL^{\otimes m}$. 
\end{lem}
Note that 
a choice of a loop $\alpha \subset \Gamma_C$
and an edge $e \subset \alpha$
 yields a 
lift of (\ref{ctorus}) by (\ref{lift}), 
which sends $z \in \mathbb{C}^{\ast}$
to $(\lL_{\alpha, z}, \phi_{\alpha, e, z})$ in the notation of 
Subsection~\ref{subsec:JacC}. 
In particular, a cyclic subgroup of order $m$ in  
$\Pic^0(U)$ generated by the 
 line bundle $\lL=\lL_{\alpha, e^{2\pi i/m}}$
lifts to an embedding into 
$\widehat{\Pic^0}(U)$, 
\begin{align}\label{mPic2}
\mathbb{Z}/m\mathbb{Z} \subset \mathbb{C}^{\ast} \subset
\widehat{\Pic^0}(U). 
\end{align}
In other words, the structure sheaf $\oO_{H\cap U}$ of the divisor $H\cap U$
in $U$
is equipped with an isomorphism, 
\begin{align}\label{lambda}
\lambda_H \colon \oO_{H} \stackrel{\cong}{\to}
\oO_{H} \otimes \lL, 
\end{align}
such that 
$\phi_{\oO_H} \cneq \lambda_H$
 satisfies the condition (\ref{cond:cyc}). 
Here we have denoted $H\cap U$ just by $H$ for simplicity. 
Hence $\lambda_H$ determines a lift of 
$\oO_{H}$ to a coherent sheaf on $\widetilde{U}$. 
This lift is a structure sheaf of some divisor 
in $\widetilde{U}$, 
\begin{align}\label{data:HU}
\widetilde{H} \subset \widetilde{U}, 
\end{align}
which intersects with $\widetilde{C}$ transversally. 
(cf.~Figure~\ref{fig:two}.)
Later we will need the following lemma 
on the compatibility of $\psi$ with $\lambda_H$. 
\begin{lem}\label{lem:i+ii+iii}
(i) The isomorphism $\psi$ in (\ref{iso:bun}) induces the 
isomorphism, 
\begin{align}\label{wpsi}
\widetilde{\psi}
 \colon \oO_{\widetilde{U}} \stackrel{\cong}{\to} \sigma^{\ast} \lL. 
\end{align}

(ii) The isomorphism $\lambda_H$ in (\ref{lambda}) induces 
the isomorphism, 
\begin{align}\label{lambdaH}
\widetilde{\lambda}_{H} \colon 
\bigoplus_{g \in \mathbb{Z}/m\mathbb{Z}}
g_{\ast}\oO_{\widetilde{H}} \stackrel{\cong}{\to}
\sigma^{\ast}\oO_H. 
\end{align}
Here $\mathbb{Z}/m\mathbb{Z}$ acts on $\widetilde{U}$ as
a deck transformation of the covering $\sigma \colon 
\widetilde{U} \to U$. 

(iii) The compositions, 
\begin{align}\label{comp1}
&(\widetilde{\lambda}_{H} \otimes \mathrm{id}_{\sigma^{\ast}\lL})^{-1} \circ 
(\widetilde{\psi} \otimes \mathrm{id}_{\sigma^{\ast}\oO_H})
 \circ \widetilde{\lambda}_H, \\ 
\label{comp2}
&(\widetilde{\lambda}_H \otimes \mathrm{id}_{\sigma^{\ast}\lL})^{-1}
\circ \sigma^{\ast}\lambda_H \circ \widetilde{\lambda}_H,  
\end{align}
determine two isomorphisms, 
\begin{align}\label{two:isom}
\bigoplus_{g \in \mathbb{Z}/m\mathbb{Z}}
g_{\ast}\oO_{\widetilde{H}}
\to \bigoplus_{g \in \mathbb{Z}/m\mathbb{Z}}
g_{\ast}\oO_{\widetilde{H}} \otimes \sigma^{\ast}\lL,
\end{align}
Both of (\ref{comp1}) and (\ref{comp2}) preserve
direct summands of both sides of (\ref{two:isom}), 
and they are related by 
\begin{align*}
(\ref{comp1})|_{g_{\ast}\oO_{\widetilde{H}}}
= e^{2\pi gi/m}  \cdot (\ref{comp2})|_{g_{\ast}\oO_{\widetilde{H}}}. 
\end{align*}
\end{lem}
\begin{proof}
(i) By Lemma~\ref{lem:forward}, it is enough to construct 
$\sigma_{\ast}\oO_{\widetilde{U}}$-module 
isomorphism $\sigma_{\ast}\widetilde{\psi}$
between $\sigma_{\ast}\oO_{\widetilde{U}}$
and $\sigma_{\ast}\sigma^{\ast}\lL \cong \lL 
\otimes \sigma_{\ast}\oO_{\widetilde{U}}$. 
It is constructed as 
\begin{align}\notag
&(x_0, x_1, \cdots, x_{m-1}) \in \oO_U \oplus \lL \oplus \cdots \oplus 
\lL^{-m+1} \\
\label{sigma:ast}
& \quad \mapsto (\psi(x_{m-1}), x_0, \cdots, x_{m-2})
\in \lL \oplus \oO_U \oplus \cdots \oplus \lL^{-m+2}.
\end{align}

(ii) As in (i), it is enough to construct
$\sigma_{\ast}\widetilde{\lambda}_{H}$. 
Note that $\sigma_{\ast}g_{\ast}\oO_{\widetilde{H}}$
is isomorphic to $\oO_H$, whose 
$\sigma_{\ast}\oO_{\widetilde{U}}$-module structure 
is given by $e^{2\pi gi/m} \cdot \lambda_H$. 
The morphism $\sigma_{\ast}\widetilde{\lambda}_{H}$
restricted to $\sigma_{\ast}g_{\ast}\oO_{\widetilde{H}}$
is constructed to be
\begin{align}\label{restrict}
x\in \oO_H  \mapsto & (x, e^{-2\pi gi/m}\lambda_H^{-1}(x), 
\cdots, e^{-2\pi (m-1)gi/m} \lambda_H^{-m+1}(x)) \\
\notag
&\in \oO_H \oplus \lL|_{H}^{-1} \oplus \cdots \oplus 
\lL|_{H}^{-m+1} \cong \sigma_{\ast}\sigma^{\ast}\oO_H. 
\end{align}
It is easy to check that the $\sigma_{\ast}\widetilde{\lambda}_H$
constructed as above is an isomorphism of 
$\sigma_{\ast}\oO_{\widetilde{U}}$-modules. 

(iii) The statement of (iii) easily follows by 
comparing (\ref{sigma:ast}) and (\ref{restrict}). 
\end{proof}

\subsection{Cyclic neighborhoods}\label{subsec:Cyclic}
In this subsection, we generalize the 
construction in the previous subsections 
and define the notion of cyclic neighborhoods. 
This will be used for the induction argument 
in the proof of Theorem~\ref{thm:main:cov} below. 
\begin{defi}\label{def:cyclic}
Let $X$, $C \subset U \subset X$ be as before. 
Let $C'$ be a connected nodal curve 
which 
is embedded into a 
three dimensional complex manifold $U'$. 
We say $C' \subset U'$ is a cyclic 
neighborhood of $C\subset X$ if
there is a sequence of local immersions,  
\begin{align}\label{unrami}
\sigma' \colon 
U'=U_{(R)} \stackrel{\sigma'_{(R)}}{\to}
 U_{(R-1)} \to \cdots \to U_{(1)} \stackrel{\sigma'_{(1)}}{\to} U_{(0)}=U, 
\end{align}
and connected nodal curves $C_{(i)} \subset U_{(i)}$ 
such that the following conditions hold: 
\begin{itemize}
\item For each $i$, 
$U_{(i)}$ is a small analytic neighborhood of 
$C_{(i)}$, and $C_{(R)}=C'$, $C_{(0)}=C$. 
\item For each $i$, the map $\sigma'_{(i)} \colon U_{(i)} \to U_{(i-1)}$
factorizes as 
\begin{align*}
\sigma_{(i)}' \colon 
U_{(i)} \subset \widetilde{U}_{(i-1)} \stackrel{\sigma_{(i)}}{\to} U_{(i-1)}, 
\end{align*}
where $\sigma_{(i)}$ is a cyclic covering with respect to
some non-trivial loop in the graph $\Gamma_{C_{(i-1)}}$, and
$U_{(i)} \subset \widetilde{U}_{(i-1)}$ is an open 
immersion. 
\item The curves $C_{(i)}$ satisfy 
$\sigma_{(i)}(C_{(i)}) \subset C_{(i-1)}$. 
\end{itemize}
\end{defi} 
Below, a cyclic 
neighborhood $C' \subset U'$ of $C\subset X$ will be written as 
\begin{align*}
(C'\subset U') \stackrel{\sigma'}{\to} (C\subset X),
\end{align*}	 
where $\sigma'$ is a composition of local immersions (\ref{unrami}). 
In order to discuss parabolic stable pair invariants on 
cyclic neighborhoods, we 
define the notion of a lift of $H\subset X$ as follows:
\begin{defi}
For a cyclic neighborhood 
$C'\subset U'$ of $C\subset X$,
a divisor $H' \subset U'$ is called a lift
of $H\subset X$
if the following conditions hold:
\begin{itemize}
\item There is a sequence (\ref{unrami})
together with divisors $H_{(i)} \subset U_{(i)}$
such that $H_{(i)}$ is obtained by lifting 
$H_{(i-1)}$ to $\widetilde{U}_{(i-1)} \to U_{(i-1)}$
as in (\ref{data:HU}) and restricting to 
$U_{(i)} \subset \widetilde{U}_{(i-1)}$. 
\item We have $H\cap U=H_{(0)}$ and $H'=H_{(R)}$. 
\end{itemize}
\end{defi}
Given a cyclic neighborhood 
$(C'\subset U') \stackrel{\sigma'}{\to} (C\subset X)$
with a lift $H' \subset U'$ of $H\subset X$, 
we can 
similarly define the notion of parabolic stable pairs, 
\begin{align}\label{para:tilde}
(F', s') \quad s'
 \in F'\otimes \oO_{H'}, 
\end{align}
with $F'$ one dimensional $\sigma^{'\ast}\omega$-semistable 
coherent sheaf on $U'$,
 satisfying the same axiom as in Definition~\ref{defi:para}. 
Here we have to assume that the support of $F'$ is 
compact in order to define $\sigma^{'\ast}\omega$-semistability. 
In what follows, we always assume that $\sigma^{'\ast}\omega$-semistable
sheaves on $U'$ have compact supports. 

\subsection{Moduli spaces and counting invariants on cyclic neighborhoods}\label{moduli:cyclic}
Let 
$(C'\subset U') \stackrel{\sigma'}{\to} (C\subset X)$
be a cyclic neighborhood and $H' \subset U'$ a lift of 
$H\subset X$. 
Since $U'$ is just a complex manifold, we need to work 
with an analytic category in discussing 
moduli spaces of semistable sheaves or parabolic stable pairs. 
A general moduli theory of sheaves on complex analytic 
spaces seems to be not yet established. However, 
since a cyclic neighborhood admits a sequence (\ref{unrami}), 
we can inductively show the existence of 
analytic moduli spaces on cyclic neighborhoods. 
The result is formulated as follows: 
\begin{lem}\label{lem:stack}
In the above situation, take $n\in \mathbb{Z}$
and $\beta' \in H_2(U', \mathbb{Z})$. 

(i) There is an analytic stack of finite type  
 $\mM_n(U', \beta')$, 
which parameterizes $\sigma^{'\ast}\omega$-semistable 
one dimensional sheaves $F' \in \Coh(U')$
satisfying 
\begin{align}\label{Fprime}
[F']=\beta', \quad \chi(F')=n'. 
\end{align} 

(ii) There is an analytic space of finite type
$M_n^{\rm{par}}(U', \beta')$, which 
represents a functor of families of 
parabolic stable pairs $(F', s')$
satisfying (\ref{Fprime}). 
\end{lem}
\begin{proof}
(i) If $(C', U')=(C, U)$, then 
$\mM_n(U', \beta')$ is obtained as an analytic open 
substack of an Artin stack $\mM_n(X, \beta')$. 
Suppose that $U'$ is an open subset of 
an $m$-fold cover $\sigma \colon \widetilde{U} \to U$ 
given by $\lL \in \Pic^0(U)$, 
as in Subsection~\ref{subsec:Coh}
and Subsection~\ref{subsec:Cyclic}. 
As an abstract stack, there is a 1-morphism, 
\begin{align}\label{1-mor}
\mM_{n}(\widetilde{U}, \beta') \to \mM_n(U, \sigma_{\ast}\beta'), 
\end{align}
by sending $F'$ to $\sigma_{\ast}F'$. 
In fact, since we have the decomposition, 
\begin{align*}
\sigma^{\ast}\sigma_{\ast}F' \cong
\bigoplus_{g\in \mathbb{Z}/m\mathbb{Z}}
g_{\ast}F', 
\end{align*}
the sheaf $\sigma_{\ast}F'$ is $\omega$-semistable
by~\cite[Lemma~3.2.2]{Hu}. 
By Lemma~\ref{lem:forward}, the fiber 
of (\ref{1-mor}) at $[F] \in \mM_n(U, \sigma_{\ast}\beta')$
is given by the closed subset of the finite dimensional vector space, 
\begin{align*}
\phi_F \in \Hom(F, F\otimes \lL)
\end{align*}
satisfying (\ref{cond:cyc}). 
Therefore (\ref{1-mor}) is representable, and 
$\mM_n(U', \beta')$ is an analytic stack of finite type. 
A general case is obtained by applying the above argument 
to the sequence (\ref{unrami}). 

(ii)  Let $\mM_n^{\rm{par}}(U', \beta')$
be an abstract stack of families of 
parabolic stable pairs on $U'$. 
We have the forgetting 1-morphism, 
\begin{align}\label{forget}
\mM_n^{\rm{par}}(U', \beta') \to \mM_n(U', \beta'), 
\end{align}
sending $(F', s')$ to $F'$. 
The fiber of (\ref{forget}) at $[F'] \in \mM_n(U', \beta')$
is given by an open subset of 
\begin{align*}
s' \in F'\otimes \oO_{H'} \cong \mathbb{C}^{\beta' \cdot H'},
\end{align*}
 giving parabolic stable pair structures on $F'$. 
Hence (\ref{forget}) is a representable 
smooth morphism, and in particular $\mM_n^{\rm{par}}(U', \beta')$
is an analytic stack of finite type. However, as in~\cite[Lemma~2.7]{Todpara}, 
there are no non-trivial stabilizer groups in $\mM_n^{\rm{par}}(U', \beta')$. 
This implies that $\mM_n^{\rm{par}}(U', \beta')$ is represented by 
an analytic space of finite type, 
$M_n^{\rm{par}}(U', \beta')$. 
\end{proof}
In the above situation, let us take 
\begin{align*}
\gamma' \in H_2(C', \mathbb{Z}), \quad 
i'_{\ast}\gamma'=\beta',
\end{align*}
where $i' \colon C' \hookrightarrow U'$ is the embedding. 
Similarly to (\ref{M(Cg)}), there is the sub analytic stack,
\begin{align}\label{sub:anay:stack}
\mM_n(C', \gamma') \subset \mM_n(U', \beta'),  
\end{align}
parameterizing $\sigma^{'\ast}\omega$-semistable sheaves 
$F'$ with $[F']=\gamma'$ as a one cycle supported on $C'$
and $\chi(F')=n$.  
Also similarly to (\ref{MnCg}), we have the sub analytic space, 
\begin{align*}
M_n^{\rm{par}}(C', \gamma') \subset M_n^{\rm{par}}(U', \beta'),
\end{align*}
parameterizing parabolic stable pairs $(F', s')$ as above. 
It is straightforward to generalize the notion of Hall algebras, 
Behrend functions, to our analytic category on $U'$. 
Consequently, we have the invariants, 
\begin{align}\label{inv:onU'}
N_{n, \gamma'}(U') \in \mathbb{Q}, \quad 
\DT_{n, \gamma'}^{\rm{par}}(U') \in \mathbb{Z}, 
\end{align}
as in (\ref{Nng}), (\ref{DTpar:loc}) respectively. 
By replacing $\DT^{\rm{par}}_{n, \gamma}$
by $\DT^{\rm{par}}_{n, \gamma'}(U')$ in the RHS of (\ref{DT:hat}), 
we can also define the invariant, 
\begin{align}\label{DT:hat2}
\widehat{\DT}^{\rm{par}}_{n, \gamma'}(U') \in \mathbb{Q}. 
\end{align}
In principle, 
as we discussed in Subsection~\ref{subsec:multiple:via}, 
the same arguments in the proof of~\cite[Corollary~4.18]{Todpara}
should show the equivalence between 
the multiple cover formula 
of $N_{n, \gamma'}(U')$
and the formula (\ref{log:form})
 for $\widehat{\DT}^{\rm{par}}_{n, \gamma'}(U')$.
However there is one technical obstruction to do this,
namely we need to show that the moduli stack 
$\mM_n(U', \beta')$
is locally written as a critical locus of some 
holomorphic function. 
(This is used in the proof of~\cite[Theorem~3.16]{Todpara}.)
The moduli stack $\mM_n(U, \beta)$ satisfies this condition, 
due to the fact that $U$ is an open subset of a projective 
Calabi-Yau 3-fold $X$, and the result by 
Joyce-Song~\cite[Theorem~5.3]{JS}. 
Unfortunately we are not able to prove this 
critical locus condition for $\mM_n(U', \beta')$.
The required condition is formulated 
in the following conjecture:
\begin{conj}\label{conj:crit}
Let $X$ be a smooth projective Calabi-Yau 3-fold over $\mathbb{C}$, 
and $C$ a connected nodal curve with $C \subset X$.
Let $(C'\subset U') \stackrel{\sigma'}{\to} (C\subset X)$
be a cyclic neighborhood, and  
for a point $[F'] \in \mM_n(U', \beta')$, 
let $G$ be a maximal reductive subgroup
in $\Aut(F')$. Then there 
exists a $G$-invariant analytic open subset 
$V$ of $0$ in $\Ext^1(F', F')$, 
a $G$-invariant holomorphic function 
$f \colon V\to \mathbb{C}$ with $f(0)=df|_{0}=0$,
and a smooth morphism of complex analytic stacks, 
\begin{align*}
[\{df=0 \}/G] \to \mM_n(U', \beta'), 
\end{align*}
of relative dimension $\dim \Aut(F')-\dim G$. 
\end{conj}
The above conjecture is 
a complex analytic version of~\cite[Theorem~5.3]{JS}, 
and true if $U'$ is an open subset of 
a projective Calabi-Yau 3-fold by~\cite[Theorem~5.3]{JS}. 
As an analogy of Proposition~\ref{prop:translate}, 
under the assumption of Conjecture~\ref{conj:crit}, 
we have the following: 
\begin{prop}\label{prop:similar:trans}
Let $C'\subset U'$ be a cyclic neighborhood of $C\subset X$
with a lift $H' \subset U'$ of $H \subset X$. 
Suppose that $C' \subset U'$ satisfies 
the condition of 
Conjecture~\ref{conj:crit}. 
Then we have the formula 
\begin{align}\label{mult:U'}
N_{n, \gamma'}(U')=\sum_{k\ge 1, k|(n, \gamma')}
\frac{1}{k^2} N_{1, \gamma'/k}(U'), 
\end{align} 
if and only if we have the formula, 
\begin{align}\label{DThatU'}
\widehat{\DT}^{\rm{par}}_{n, \gamma'}
(U')=\sum_{k\ge 1, k|(n, \gamma')}
\frac{(-1)^{\gamma' \cdot H' -1}}{k^2}
(\gamma' \cdot H') N_{1, \gamma'/k}
(U'). 
\end{align}
\end{prop}
\begin{proof}
Since we assume Conjecture~\ref{conj:crit}, 
the same argument of~\cite[Corollary~4.18]{Todpara}
works. 
\end{proof}

\section{Counting invariants under cyclic coverings}
This section is a core of this paper. 
We will compare the invariants (\ref{inv:onU'}), (\ref{DT:hat2})
constructed in the previous section under 
cyclic coverings. Using this, we will 
show our main result which reduces 
the multiple cover formula of $N_{n, \gamma}$ to 
that of $N_{n, \gamma'}(U')$ for all cyclic neighborhood 
$C'\subset U'$ with $C'$ a tree of $\mathbb{P}^1$. 
\subsection{Comparison of moduli spaces of parabolic stable pairs}
\label{subsec:compare:para}
Let $C \subset U \subset X$ be as in the previous section. 
As in Subsection~\ref{subsec:Coh}
and Subsection~\ref{subsec:Cyclic}, 
let 
$\sigma \colon \widetilde{U} \to U$ be a cyclic covering of 
order $m$, and $\widetilde{H} \subset \widetilde{U}$ is a lift of $H$. 
Note that $\widetilde{C} \subset \widetilde{U}$ is a cyclic 
neighborhood of $C \subset X$, so we have the moduli space of  
parabolic stable pairs on $\widetilde{U}$ by  
Lemma~\ref{lem:stack}. 
In this subsection, we compare moduli spaces of 
parabolic stable pairs under the above covering. 

Note that by (\ref{mPic2}) and the argument in Subsection~\ref{subsec:Jact}, 
we have the $\mathbb{C}^{\ast}$-action on 
the moduli spaces in (\ref{sub:par}), 
which restricts to the $\mathbb{Z}/m\mathbb{Z}$-action 
on these moduli spaces. 
We have the following lemma: 
\begin{lem}\label{lem:nat:comp}
For $\widetilde{\beta} \in H_2(\widetilde{U}, \mathbb{Z})$
with $\beta=\sigma_{\ast}\widetilde{\beta}$, there 
is a natural morphism of complex analytic spaces, 
\begin{align}\label{mor:com}
\sigma_{\ast} \colon M_n^{\rm{par}}(\widetilde{U}, \widetilde{\beta})
\to M_n^{\rm{par}}(U, \beta)^{\mathbb{Z}/m\mathbb{Z}}. 
\end{align}
\end{lem}
\begin{proof}
For a point $(\widetilde{F}, \widetilde{s}) \in M_n^{\rm{par}}(\widetilde{U}, \widetilde{\beta})$, we construct the pair
\begin{align*}
(F, s) \cneq \sigma_{\ast}(\widetilde{F}, \widetilde{s}),
\end{align*}
in the following way: 
first we set $F \cneq \sigma_{\ast} \widetilde{F}$, 
which is $\omega$-semistable as in the proof of Lemma~\ref{lem:stack} (i).  
Then we have 
\begin{align*}
F \otimes \oO_{H} 
&\cong \sigma_{\ast}(\widetilde{F} \otimes \sigma^{\ast}\oO_H) \\
&\cong \bigoplus_{g\in \mathbb{Z}/m\mathbb{Z}} \sigma_{\ast}(\widetilde{F} \otimes
g_{\ast}\oO_{\widetilde{H}}), 
\end{align*}
where 
the second isomorphism is induced by (\ref{lambdaH}). 
Then we have the embedding into the direct summand, 
\begin{align}\label{emb:sum}
\sigma_{\ast} \colon \widetilde{F} \otimes \oO_{\widetilde{H}} 
\hookrightarrow F\otimes \oO_H, 
\end{align}
and we set $s \cneq \sigma_{\ast}\widetilde{s}$. 

We would like to see that $(F, s)$ is a
parabolic 
stable pair on $U$. 
Suppose by contradiction
that there is a surjection 
$\pi \colon F\twoheadrightarrow F'$, 
where $F'$ is an $\omega$-semistable sheaf
with $\mu_{\omega}(F)=\mu_{\omega}(F')$, 
satisfying 
\begin{align}\label{piHs}
(\pi \otimes \oO_H) (s) =0. 
\end{align}
By taking the adjunction, we have the non-zero map, 
\begin{align}\label{mor:adj}
\widetilde{F} \to \sigma^{!}F' =\sigma^{\ast}F'. 
\end{align}
By~\cite[Lemma~3.2.2]{Hu}, 
$\sigma^{\ast}F'$ is $\sigma^{\ast}\omega$-semistable 
with $\mu_{\sigma^{\ast}\omega}(\widetilde{F})=
\mu_{\sigma^{\ast}\omega}(\sigma^{\ast}F')$, 
hence the image of the morphism (\ref{mor:adj}), 
denoted by $A$, is also $\sigma^{\ast}\omega$-
semistable with $\mu_{\sigma^{\ast}\omega}(A)
=\mu_{\sigma^{\ast}\omega}(\widetilde{F})$. 
We have the sequence, 
\begin{align*}
\widetilde{F} \otimes \oO_{\widetilde{H}}
\twoheadrightarrow A \otimes \oO_{\widetilde{H}}
\hookrightarrow \sigma^{\ast}F' \otimes \oO_{\widetilde{H}}, 
\end{align*}
which takes $\widetilde{s}$ to zero by 
the construction of $s$ and (\ref{piHs}).
Since the right arrow of the above sequence is 
injective, the surjection 
$\widetilde{F} \twoheadrightarrow A$ violates the 
condition of parabolic stability of $(\widetilde{F}, \widetilde{s})$. 
This is a contradiction, hence $(F, s)$ is 
a parabolic stable pair.  

We check that $(F, s)$ is $\mathbb{Z}/m\mathbb{Z}$-invariant. 
This is equivalent to that the parabolic stable pair 
\begin{align*}
(F\otimes \lL, (\id_F \otimes \lambda_H)(s)), 
\end{align*}
is isomorphic to $(F, s)$,
where $\lambda_H$ is given in (\ref{lambda}).  
Since $F=\sigma_{\ast}\widetilde{F}$, the 
isomorphism (\ref{wpsi}) and the projection 
formula induce the isomorphism
$\psi_F \cneq \sigma_{\ast}(\id_{\sigma_{\ast}\widetilde{F}}
 \otimes \widetilde{\psi})$, 
\begin{align}\label{psi_F}
\psi_{F} \colon \sigma_{\ast}\widetilde{F}
\to \sigma_{\ast}(\widetilde{F} \otimes \sigma^{\ast}\lL)
\cong \sigma_{\ast} \widetilde{F} \otimes \lL. 
\end{align}
We need to check that 
\begin{align*}
(\psi_F \otimes \oO_H)(s) =(\id_F \otimes \lambda_H)(s).
\end{align*}
The above equality follows from that
$s$ comes from the LHS of (\ref{emb:sum})
and Lemma~\ref{lem:i+ii+iii} (iii). 

The above argument shows that we have a set theoretic 
map $\sigma_{\ast}$. 
It is straightforward to generalize the above arguments to 
families of parabolic stable pairs. Namely 
for a complex analytic space $S$, 
let 
\begin{align*}
\Hom(S, M_n^{\rm{par}}(U, \beta)),
\end{align*}
be the 
set of morphisms from $S$ to $M_n^{\rm{par}}(U, \beta)$
as complex analytic spaces. Then, since $M_n^{\rm{par}}(U, \beta)$
is a fine moduli space,  
giving an element in $\Hom(S, M_n^{\rm{par}}(U, \beta))$
is equivalent to giving a flat family of 
parabolic stable pairs over $S$. 
We can easily generalize the 
construction of $\sigma_{\ast}$ to 
a functorial map, 
\begin{align*}
\Hom(S, M_n^{\rm{par}}(\widetilde{U}, \widetilde{\beta})) \to 
\Hom(S, M_n^{\rm{par}}(U, \beta))^{\mathbb{Z}/m\mathbb{Z}}, 
\end{align*}
which gives a morphism (\ref{mor:com}) as complex analytic spaces. 
\end{proof}

We have the following proposition. 
\begin{prop}\label{prop:isom:para}
The morphism (\ref{mor:com})
induces the isomorphism
of 
complex analytic spaces, 
\begin{align}\label{isom:para}
\sigma_{\ast} \colon 
\coprod_{\begin{subarray}{c}
\widetilde{\beta}\in H_2(\widetilde{U}, \mathbb{Z}), \\
\sigma_{\ast} \widetilde{\beta}=\beta
\end{subarray}}
M_n^{\rm{par}}(\widetilde{U}, \widetilde{\beta})
\stackrel{\cong}{\longrightarrow}
M_n^{\rm{par}}(U, \beta)^{\mathbb{Z}/m\mathbb{Z}},
\end{align}
which restricts to the isomorphism, 
\begin{align}\label{isom:para2}
\sigma_{\ast} \colon 
\coprod_{\begin{subarray}{c}
\widetilde{\gamma} \in H_2(\widetilde{C}, \mathbb{Z}), \\
\sigma_{\ast} \widetilde{\gamma}=\gamma
\end{subarray}}
M_n^{\rm{par}}(\widetilde{C}, \widetilde{\gamma})
\stackrel{\cong}{\longrightarrow}
M_n^{\rm{par}}(C, \gamma)^{\mathbb{Z}/m\mathbb{Z}}.
\end{align}
\end{prop}
\begin{proof}
It is enough to show the isomorphism (\ref{isom:para}). 
We first show that the morphism $\sigma_{\ast}$ is 
injective. Suppose that there are two parabolic 
stable pairs, 
\begin{align*}
(\widetilde{F}_i, \widetilde{s}_i) \in M_n^{\rm{par}}(\widetilde{U}, \widetilde{\beta}_i), \ i=1, 2, 
\end{align*}
which are sent to the same point by $\sigma_{\ast}$. 
This implies that there is an isomorphism of sheaves, 
\begin{align*}
\phi \colon 
\sigma_{\ast}\widetilde{F}_1 \stackrel{\cong}{\to}
\sigma_{\ast} \widetilde{F}_2, 
\end{align*}
such that $\phi\otimes \id_{\oO_H}$ sends 
$\sigma_{\ast}\widetilde{s}_1$
to $\sigma_{\ast}\widetilde{s}_2$. 
Let us check that $\phi$ is $\sigma_{\ast}\oO_{\widetilde{U}}$-module
homomorphism. 
This is equivalent to that the following diagram commutes: 
\begin{align}\label{diagm:L}
\xymatrix{
\sigma_{\ast}\widetilde{F}_1 \ar[r]^{\phi} \ar[d]_{\psi_{1}} &
\sigma_{\ast}\widetilde{F}_2 \ar[d]^{\psi_{2}} \\
\sigma_{\ast}\widetilde{F}_1 \otimes \lL \ar[r]^{\phi \otimes \id_{\lL}} &
\sigma_{\ast}\widetilde{F}_2 \otimes \lL.
}
\end{align}
Here $\psi_{i}$ are induced by the projection formula and
the isomorphism (\ref{wpsi})
as in (\ref{psi_F}). We check that 
\begin{align}\label{check}
\left\{(\psi_{2} \circ \phi) \otimes \id_{\oO_H}\right\} (\sigma_{\ast}\widetilde{s}_1)
=\left\{\{(\phi \otimes \id_{\lL}) \circ \psi_{1} \} \otimes \id_{\oO_H}\right\}(\sigma_{\ast}\widetilde{s}_1), 
\end{align}
as elements in $\sigma_{\ast}\widetilde{F}_2  \otimes \lL \otimes \oO_H$. 
In fact if (\ref{check}) holds, then 
$\phi^{-1} \circ \psi_{2}^{-1} \circ
(\phi \otimes \id_{\lL}) \circ \psi_{1}$
is an automorphism of $(\sigma_{\ast}F_1, \sigma_{\ast}s_1)$, 
hence identity by~\cite[Lemma~2.7]{Todpara}, 
i.e. the diagram (\ref{diagm:L}) commutes. 

The equality (\ref{check}) can be checked as follows. 
The LHS of (\ref{check}) is 
\begin{align*}
&(\psi_{2} \otimes \id_{\oO_H}) (\phi \otimes \id_{\oO_H})
(\sigma_{\ast}\widetilde{s}_1) \\
&= (\psi_{2} \otimes \id_{\oO_H}) (\sigma_{\ast}s_2) \\
&= (\id_{\sigma_{\ast}\widetilde{F}_2} \otimes \lambda_{H})(\sigma_{\ast}\widetilde{s}_2), 
\end{align*}
where we have used Lemma~\ref{lem:i+ii+iii} (iii) 
for the second equality. 
Similarly the RHS of (\ref{check}) is 
\begin{align*}
&(\phi \otimes \id_{\lL \otimes \oO_H})(\psi_{1} \otimes \id_{\oO_H})
(\sigma_{\ast}\widetilde{s}_1) \\
&=(\phi \otimes \id_{\lL \otimes \oO_H})(\id_{\sigma_{\ast}\widetilde{F}_1}
\otimes \lambda_H)(\sigma_{\ast}\widetilde{s}_1) \\
&=(\id_{\sigma_{\ast}\widetilde{F}_2} \otimes \lambda_H)(\phi \otimes \id_{\oO_H})(\sigma_{\ast}\widetilde{s}_1) \\
&=(\id_{\sigma_{\ast}\widetilde{F}_2} \otimes \lambda_{H})(\sigma_{\ast}\widetilde{s}_2).
\end{align*}
Therefore the equality (\ref{check}) holds. 

Now since the diagram (\ref{diagm:L}) commutes, 
the isomorphism $\phi$ lifts to an isomorphism 
$\widetilde{\phi}$ between 
$\widetilde{F}_1$ and $\widetilde{F}_2$, 
such that $\widetilde{\phi}\otimes \id_{\oO_{\widetilde{H}}}$
takes $\widetilde{s}_1$ to $\widetilde{s}_2$. 
Hence $(\widetilde{F}_1, \widetilde{s}_1)$
and $(\widetilde{F}_2, \widetilde{s}_2)$ are 
isomorphic,
and $\sigma_{\ast}$ is injective. 

Next we prove that $\sigma_{\ast}$ is surjective. 
Let us take a point 
$(F, s) \in M_n^{\rm{par}}(U, \beta)$, 
which is $\mathbb{Z}/m\mathbb{Z}$-invariant. 
This means that there is an isomorphism of sheaves, 
\begin{align*}
\phi_F \colon F \to F \otimes \lL, 
\end{align*}
which satisfies that 
\begin{align}\label{satisfy:phi}
(\phi_F \otimes \id_{\oO_H})(s)
=(\id_F \otimes \lambda_H)(s). 
\end{align}
We show that $\phi_F$ satisfies 
the condition (\ref{cond:cyc}).
We consider the morphism, 
\begin{align}\label{consider}
(\id_F \otimes \psi)^{-1} \circ
\overbrace{\phi_F \circ \cdots \circ \phi_F}^{m}
\colon F \to F. 
\end{align} 
After applying $\otimes \oO_H$, the above 
morphism takes $s$ to $s$ since (\ref{satisfy:phi}) holds 
and 
$\phi_{\oO_H} =\lambda_H$ satisfies the condition (\ref{cond:cyc}). 
Therefore the morphism (\ref{satisfy:phi})
is an identity by~\cite[Lemma~2.7]{Todpara},
which implies that $\phi_F$ satisfies
(\ref{cond:cyc}).  

Since $\phi_F$ satisfies (\ref{cond:cyc}), 
the sheaf $F$ is written as $\sigma_{\ast} \widetilde{F}$
for some $\widetilde{F} \in \Coh(\widetilde{U})$, 
such that the morphism $\phi_F$ is identified with 
\begin{align}\label{psi_F2}
\psi_{F} \colon \sigma_{\ast}\widetilde{F}
\to \sigma_{\ast}(\widetilde{F} \otimes \sigma^{\ast}\lL)
\cong \sigma_{\ast} \widetilde{F} \otimes \lL,
\end{align}
which is a composition of (\ref{wpsi}) and the projection
formula as in (\ref{psi_F}). 
To show that $\sigma_{\ast}$
is surjective, it is enough to check that $s=\sigma_{\ast}\widetilde{s}$
for some $\widetilde{s} \in \widetilde{F} \otimes \oO_{\widetilde{H}}$. 
This follows from (\ref{satisfy:phi}), the fact that 
$\phi_F$ is identified with (\ref{psi_F2}) and Lemma~\ref{lem:i+ii+iii} (iii). 

Now we have proved that $\sigma_{\ast}$ is a set theoretic
bijection. 
Similarly to Lemma~\ref{lem:nat:comp}, the above arguments
can be easily generalized to families of parabolic 
stable pairs. Namely in the notation of 
the proof of Lemma~\ref{lem:nat:comp}, we have the bijection,
\begin{align*}
\sigma_{\ast} \colon 
\coprod_{\sigma_{\ast} \widetilde{\beta}=\beta}
\Hom(S, M_n^{\rm{par}}(\widetilde{U}, \widetilde{\beta})) 
\stackrel{\cong}{\to} 
\Hom(S, M_n^{\rm{par}}(U, \beta))^{\mathbb{Z}/m\mathbb{Z}}.
\end{align*}
Therefore the morphism $\sigma_{\ast}$
is an isomorphism as complex analytic spaces. 
\end{proof}

\subsection{Comparison of moduli spaces of stable sheaves}
Similarly to the previous subsection, we can 
also compare moduli spaces of 
one dimensional stable sheaves under the covering 
$\sigma \colon \widetilde{U} \to U$. 
Note that the moduli space $M_1(U, \beta)$
consists of one dimensional $\omega$-stable 
sheaves on $U$, and it admits 
$\Pic^0(U)$-action. 
In particular it restricts to 
$\mathbb{C}^{\ast}$-action 
w.r.t. the embedding (\ref{ctorus}), 
and we have the $\mathbb{Z}/m\mathbb{Z}$-action
by the embedding (\ref{mPic2}). 

Let 
\begin{align}\label{M1:sub}
M_1(\widetilde{C}, \widetilde{\gamma}) \subset
M_1(\widetilde{U}, \widetilde{\beta}),
\end{align}
be the coarse moduli spaces of analytic 
stacks (\ref{sub:anay:stack}) for 
$n=1$, $U'=\widetilde{U}$,
$\gamma'=\widetilde{\gamma}$
and $\beta'=\widetilde{\beta}$. 
Similarly to $M_1(U, \beta)$, 
points in the analytic spaces (\ref{M1:sub})
correspond to $\sigma^{\ast} \omega$-stable 
sheaves. 
We have the following proposition. 
\begin{prop}\label{prop:sigma:stable}
We have the morphisms of complex analytic spaces, 
\begin{align}\label{isom:para:st}
\sigma_{\ast} \colon 
\coprod_{\begin{subarray}{c}
\widetilde{\beta}\in H_2(\widetilde{U}, \mathbb{Z}), \\
\sigma_{\ast} \widetilde{\beta}=\beta
\end{subarray}}
M_1^{}(\widetilde{U}, \widetilde{\beta})
\longrightarrow
M_1^{}(U, \beta)^{\mathbb{Z}/m\mathbb{Z}}, \\
\label{isom:para:st2}
\sigma_{\ast} \colon 
\coprod_{\begin{subarray}{c}
\widetilde{\gamma} \in H_2(\widetilde{C}, \mathbb{Z}), \\
\sigma_{\ast} \widetilde{\gamma}=\gamma
\end{subarray}}
M_1^{}(\widetilde{C}, \widetilde{\gamma})
\longrightarrow
M_1^{}(C, \gamma)^{\mathbb{Z}/m\mathbb{Z}}.
\end{align}
If $m\gg 0$, then the 
above morphisms are covering maps of covering degree
$m$. 
\end{prop}
\begin{proof}
It is enough to show the claim for (\ref{isom:para:st}). 
A proof similar to Lemma~\ref{lem:nat:comp}
shows the existence of the morphism (\ref{isom:para:st}). 
In order to show that (\ref{isom:para:st}) is a covering map, 
it is enough to show that there is a free 
$\mathbb{Z}/m\mathbb{Z}$-action on 
the LHS of (\ref{isom:para:st}) whose quotient space 
is isomorphic to the RHS of (\ref{isom:para:st}). 

Note that $\sigma \colon \widetilde{U} \to U$ is a covering 
map whose covering transformation group is $\mathbb{Z}/m\mathbb{Z}$. 
Hence $\mathbb{Z}/m\mathbb{Z}$ acts on 
the LHS of (\ref{isom:para:st})
by $F \mapsto g_{\ast} F$ for $g\in \mathbb{Z}/m\mathbb{Z}$. 
Note that the support of $F$ are connected, and 
if $m\gg 0$ and $g\neq 0$, then the support of $F$ and that of $g_{\ast}F$
are different. 
Hence $F$ and $g_{\ast}F$ are not isomorphic for 
$g\neq 0$, which implies that $\mathbb{Z}/m\mathbb{Z}$-action 
on the LHS of (\ref{isom:para:st}) is free. 

For two $\sigma^{\ast}\omega$-stable sheaves $\widetilde{F}_i$, $i=1, 2,$
corresponding to points in 
the LHS of (\ref{isom:para:st}), suppose that 
$\sigma_{\ast}\widetilde{F}_1$ and $\sigma_{\ast}\widetilde{F}_2$
are isomorphic. By adjunction, there is a non-trivial 
morphism, 
\begin{align*}
\sigma^{\ast}\sigma_{\ast}\widetilde{F}_2 
\cong \bigoplus_{g\in \mathbb{Z}/m\mathbb{Z}}
g_{\ast}\widetilde{F}_2 \to \widetilde{F}_1.
\end{align*}
Hence there are $g\in \mathbb{Z}/m\mathbb{Z}$
and a non-trivial morphism $g_{\ast} \widetilde{F}_2 \to \widetilde{F}_1$. 
Since both of $\widetilde{F}_1$ and $g_{\ast}\widetilde{F}_2$
are $\sigma^{\ast}\omega$-stable with 
$\mu_{\sigma^{\ast}\omega}(\widetilde{F}_1)
=\mu_{\sigma^{\ast}\omega}(g_{\ast}\widetilde{F}_2)$, 
we have 
$g_{\ast} \widetilde{F}_2 \cong \widetilde{F}_1$. 

Next we check that (\ref{isom:para:st}) is surjective. 
For an $\omega$-stable sheaf $F \in M_1(U, \beta)$, suppose that 
$F$ is $\mathbb{Z}/m\mathbb{Z}$-invariant. 
This implies that there is an isomorphism of sheaves, 
\begin{align*}
\phi_F \colon F \to F\otimes \lL. 
\end{align*}
The morphism $\phi_F$ may not satisfy
the condition (\ref{cond:cyc}). 
However since $F$ is $\omega$-stable, we have 
$\Aut(F)=\mathbb{C}^{\ast}$, so by 
replacing $\phi_F$ by a non-zero multiple, we 
can assume that $\phi_F$ satisfies (\ref{cond:cyc}). 
Hence $F$ is isomorphic to $\sigma_{\ast}\widetilde{F}$
for some sheaf $\widetilde{F} \in \Coh(\widetilde{U})$. 
The sheaf $\widetilde{F}$ must be 
$\sigma^{\ast}\omega$-stable since
$\sigma_{\ast} \colon \Coh(\widetilde{U}) \to \Coh(U)$
is an exact functor. This shows that 
(\ref{isom:para:st}) is surjective. 

The above argument shows that 
$\sigma_{\ast}$ induces a bijection between 
the quotient space of the LHS of (\ref{isom:para:st})
by the $\mathbb{Z}/m\mathbb{Z}$-action and the RHS of (\ref{isom:para:st}). 
Similarly to Lemma~\ref{lem:nat:comp}, Proposition~\ref{prop:isom:para},
the above bijection is an isomorphism between complex analytic 
spaces. Namely for a complex analytic space $S$, it is 
straightforward to generalize the above argument to the
bijection, 
\begin{align*}
\sigma_{\ast} \colon 
\left(
\coprod_{\sigma_{\ast} \widetilde{\beta}=\beta}
\Hom(S, M_1(\widetilde{U}, \widetilde{\beta})) 
 \right)/(\mathbb{Z}/m\mathbb{Z}) 
\stackrel{\cong}{\to} 
\Hom(S, M_1(U, \beta))^{\mathbb{Z}/m\mathbb{Z}}.
\end{align*}
Therefore we obtain the desired assertion. 
\end{proof}

\subsection{The formula for $\widehat{\DT}_{n, \gamma}^{\rm{par}}$
 under the cyclic covering}
In this subsection, we investigate the formula (\ref{log:form}) under the 
cyclic
covering $\sigma \colon \widetilde{U} \to U$. 
In the notation of previous subsections, we denote by 
\begin{align*}
\nu_{\widetilde{M}^{\rm{par}}}, \ 
\nu_{M^{\rm{par}}}, \ 
\nu_{\widetilde{M}}, \ \nu_{M},
\end{align*}
the Behrend functions on 
the spaces, 
\begin{align*}
M_n^{\rm{par}}(\widetilde{U}, \widetilde{\beta}), \ 
M_n^{\rm{par}}(U, \beta), \ 
M_1(\widetilde{U}, \widetilde{\beta}), \ 
M_1(U, \beta), 
\end{align*}
respectively. 
We need the following compatibility of 
the above Behrend functions on the 
morphisms discussed in Proposition~\ref{prop:isom:para} 
and Proposition~\ref{prop:sigma:stable}. 
The proof is postponed until Section~\ref{subsec:Behrend}. 
\begin{lem}\label{lem:identity}
If $m$ is a sufficiently big odd number, we have the 
following: 

(i) Under the morphism (\ref{isom:para}), we have the identity, 
\begin{align*}
(\sigma_{\ast})^{\ast}\nu_{M^{\rm{par}}}|_{M_n^{\rm{par}}(\widetilde{U}, \widetilde{\beta})}=(-1)^{\beta \cdot H -\widetilde{\beta} \cdot \widetilde{H}}
\nu_{\widetilde{M}^{\rm{par}}}.
\end{align*}
(ii) Under the morphism (\ref{isom:para:st}), we have the identity, 
\begin{align*}
(\sigma_{\ast})^{\ast} \nu_{M}=\nu_{\widetilde{M}}. 
\end{align*}
\end{lem}
\begin{proof}
The proof will be given in Subsection~\ref{subsec:proof}. 
\end{proof}
As a corollary of 
Proposition~\ref{prop:sigma:stable}, Proposition~\ref{prop:isom:para}
and Lemma~\ref{lem:identity}, we have the following: 
\begin{cor}\label{cor:formula}
For $\gamma \in H_2(C, \mathbb{Z})$
and $n\in \mathbb{Z}$, we take a sufficiently 
big odd number $m$ and an $m$-fold cover
$\sigma \colon \widetilde{U} \to U$ 
as in (\ref{etale:C}). 
We have the formulas, 
\begin{align}\label{formula1}
\DT_{n, \gamma}^{\rm{par}}
&=\sum_{\sigma_{\ast}\widetilde{\gamma}=\gamma}
(-1)^{\gamma \cdot H -\widetilde{\gamma} \cdot \widetilde{H}}
\DT_{n, \widetilde{\gamma}}^{\rm{par}}(\widetilde{U}), \\
\label{formula2}
N_{1, \gamma} &=
\frac{1}{m}
\sum_{\sigma_{\ast}\widetilde{\gamma}=\gamma}
N_{1, \widetilde{\gamma}}(\widetilde{U}). 
\end{align}
\end{cor}
\begin{proof}
Let us consider $\mathbb{C}^{\ast}$-actions on 
$M_n^{\rm{par}}(C, \gamma)$, $M_1(C, \gamma)$, 
determined by the embeddings (\ref{lift}), (\ref{ctorus}) respectively. 
Then for $m\gg 0$, we have 
\begin{align}\label{Cast:loc}
M_n^{\rm{par}}(C, \gamma)^{\mathbb{C}^{\ast}}
&=M_n^{\rm{par}}(C, \gamma)^{\mathbb{Z}/m\mathbb{Z}}, \\
\label{Cast:loc2}
M_1(C, \gamma)^{\mathbb{C}^{\ast}}
&=M_1(C, \gamma)^{\mathbb{Z}/m\mathbb{Z}}. 
\end{align}
Therefore the formulas (\ref{formula1}), (\ref{formula2})
follow from Proposition~\ref{prop:isom:para}, 
Proposition~\ref{prop:sigma:stable}, Lemma~\ref{lem:identity}, 
(\ref{Cast:loc}), (\ref{Cast:loc2}) and the $\mathbb{C}^{\ast}$-localizations. 
\end{proof}
Furthermore we have the following proposition. 
\begin{prop}\label{prop:reduce:cov}
In the situation of Corollary~\ref{cor:formula}, 
suppose that the 
following formula holds on $\widetilde{U}$, 
\begin{align}\label{assum:form}
\widehat{\DT}_{n, \widetilde{\gamma}}^{\rm{par}}(\widetilde{U})
=\sum_{k\ge 1, k|(n, \widetilde{\gamma})}
\frac{(-1)^{\widetilde{\gamma} \cdot \widetilde{H}-1}}{k^2}
(\widetilde{\gamma} \cdot \widetilde{H})N_{1, \widetilde{\gamma}/k}(\widetilde{U}),
\end{align}
for any 
$\widetilde{\gamma} \in H_2(\widetilde{C}, \mathbb{Z})$ with $\sigma_{\ast} \widetilde{\gamma}=\gamma$. 
Then the formula (\ref{log:form}) holds. 
\end{prop}
\begin{proof}
By Corollary~\ref{cor:formula}, 
the LHS of (\ref{log:form}) is 
\begin{align}\label{LHS1}
&\sum_{l\ge 1}\frac{(-1)^{l-1}}{l}
\sum_{\begin{subarray}{c}
\gamma_1 + \cdots +\gamma_l=\gamma, \\
n_1 + \cdots +n_l=n, \\
n_i/\omega \cdot \gamma_i =
n/\omega \cdot \gamma
\end{subarray}}
\prod_{i=1}^{l}
\left( \sum_{\sigma_{\ast}
\widetilde{\gamma}_i=\gamma_i}
(-1)^{\gamma_i \cdot H - \widetilde{\gamma}_i \cdot
\widetilde{H}}
\DT_{n_i, \widetilde{\gamma}_i}^{\rm{par}}(\widetilde{U})  \right) \\
\notag
&=\sum_{\sigma_{\ast} \widetilde{\gamma}=\gamma}
(-1)^{\gamma \cdot H - \widetilde{\gamma} \cdot \widetilde{H}}
\sum_{l\ge 1} \frac{(-1)^{l-1}}{l}
\sum_{\begin{subarray}{c}
\widetilde{\gamma}_1 + \cdots +
\widetilde{\gamma}_l=\widetilde{\gamma}, \\
n_1 + \cdots +n_l=n, \\
n_i/\sigma^{\ast}\omega \cdot \widetilde{\gamma}_i
= n/\sigma^{\ast}\omega \cdot \widetilde{\gamma} 
\end{subarray}}
\prod_{i=1}^{l} \DT_{n_i, \widetilde{\gamma}_i}^{\rm{par}}(\widetilde{U}) \\
\label{LHS2}
&= \sum_{\sigma_{\ast}\widetilde{\gamma}=\gamma}
\sum_{k\ge 1, k|(n, \widetilde{\gamma})}
\frac{(-1)^{\gamma \cdot H-1}}{k^2}
(\widetilde{\gamma} \cdot \widetilde{H})N_{1, \widetilde{\gamma}/k}(\widetilde{U}) \\
&\notag
=\sum_{k\ge 1, k|(n, \gamma)}
\frac{(-1)^{\gamma \cdot H -1}}{k^2}
\sum_{\sigma_{\ast} \widetilde{\gamma}=\gamma, \ k|\widetilde{\gamma}}
(\widetilde{\gamma} \cdot \widetilde{H})
N_{1, \widetilde{\gamma}/k}(\widetilde{U}) \\
\notag
&= \sum_{k\ge 1, k|(n, \gamma)}
\frac{(-1)^{\gamma \cdot H -1}}{k^2} \cdot
\frac{1}{m} \sum_{\begin{subarray}{c}
\sigma_{\ast}\widetilde{\gamma}=\gamma, \ k|\widetilde{\gamma} \\
g\in \mathbb{Z}/m\mathbb{Z}
\end{subarray}}
(g_{\ast}\widetilde{\gamma} \cdot \widetilde{H}) 
N_{1, g_{\ast}\widetilde{\gamma}/k}(\widetilde{U}) \\
\notag
&=\sum_{k\ge 1, k|(n, \gamma)}
\frac{(-1)^{\gamma \cdot H -1}}{k^2} \cdot
\frac{1}{m} \sum_{\sigma_{\ast}\widetilde{\gamma}=\gamma, 
\ k|\widetilde{\gamma}}
(\gamma \cdot H)N_{1, \widetilde{\gamma}/k}(\widetilde{U}) \\
\label{LHS3}
&= \sum_{k\ge 1, k|(n, \gamma)}
\frac{(-1)^{\gamma \cdot H -1}}{k^2}
(\gamma \cdot H)
N_{1, \gamma/k}. 
\end{align}
Here we have used (\ref{formula1}), (\ref{assum:form}), (\ref{formula2})
in (\ref{LHS1}), (\ref{LHS2}), (\ref{LHS3}) 
respectively. 
Therefore the formula (\ref{log:form}) holds. 
\end{proof}

\subsection{Reduction to trees of $\mathbb{P}^1$}
Now we show our main result. 
\begin{thm}\label{thm:main:cov}
Let $X$ be a smooth projective Calabi-Yau 3-fold over 
$\mathbb{C}$, $C \subset X$ a reduced
rational curve with at worst nodal singularities, 
and take 
 $\gamma \in H_2(C, \mathbb{Z})$. 
Suppose that for any cyclic neighborhood
$(C'\subset U') \stackrel{\sigma'}{\to} (C\subset X)$
with $C'$ a tree of $\mathbb{P}^1$,
the following conditions hold: 
\begin{itemize}
\item 
The cyclic neighborhood $C' \subset U'$ 
satisfies the condition of 
Conjecture~\ref{conj:crit}. 
\item For any $\gamma' \in H_2(C', \mathbb{Z})$ with 
$\sigma_{\ast}'\gamma'=\gamma$, the invariant 
$N_{n, \gamma'}(U')$ satisfies the 
formula (\ref{mult:U'}). 
\end{itemize}
Then $N_{n, \gamma}$ satisfies the formula (\ref{form:mult:loc}). 
\end{thm}
\begin{proof}
An element $\gamma \in H_2(C, \mathbb{Z})_{>0}$
can be written as 
\begin{align*}
\gamma=\sum_{i=1}^{N} a_i [C_i],
\end{align*} 
for $a_i \in \mathbb{Z}_{\ge 0}$ where 
$C_1, \cdots, C_N$ are irreducible components of $C$. 
The support of $\gamma$, denoted by $C_{\gamma}$, 
is defined to be the reduced curve, 
\begin{align*}
C_{\gamma} \cneq \bigcup_{a_i>0} C_i 
\subset C. 
\end{align*}
We also set $d(\gamma)$ and $l(\gamma)$ to be
\begin{align*}
d(\gamma) \cneq \sum_{i=1}^{N} a_i, \quad
l(\gamma) \cneq \sharp\{ 1\le i\le N : a_i>0\}. 
\end{align*}
Note that we have $l(\gamma) \le d(\gamma)$. 

We note that, if $\widehat{\DT}_{n, \gamma}^{\rm{par}}$ or $N_{1, \gamma}$
is non-zero, then $C_{\gamma}$ is a connected curve. 
In fact, by~\cite[Equation (95)]{Todpara}, 
the invariant $\widehat{\DT}_{n, \gamma}^{\rm{par}}$ is a multiple of 
$N_{n, \gamma}$. And if $N_{n, \gamma}$ is non-zero, then 
the same argument of~\cite[Lemma~11.6]{TodK3} shows that 
$C_{\gamma}$ is connected. 
Hence we may assume that 
$C$ is connected and $C_{\gamma}=C$. 
Note that $l(\gamma)=N$ in this case. 

If $g(C)=0$, then the result follows from the 
assumption. Suppose that $g(C)>0$, and 
 take a
sufficiently small analytic neighborhood $C \subset U$
in $X$. 
Then we can take an 
 $m$-fold cover 
\begin{align*}
\sigma \colon \widetilde{U} \to U, \quad
\widetilde{C}=\sigma^{-1}(C),
\end{align*} as in Subsection~\ref{subsec:cyclic},
 for a sufficiently big odd number $m$. 
We take $\widetilde{\gamma} \in H_2(\widetilde{C}, \mathbb{Z})$
satisfying that $\sigma_{\ast}\widetilde{\gamma}=\gamma$, 
$C_{\widetilde{\gamma}}$ is connected and 
intersects with $\widetilde{H}$. 
Then either 
one of the following conditions hold:
\begin{align}
\label{oneof1}
&l(\widetilde{\gamma})>l(\gamma)=N, \mbox{ or } \\
\label{oneof2}
&l(\widetilde{\gamma})=
l(\gamma)=N, \ g(C_{\widetilde{\gamma}})<g(C_{\gamma})=g(C).
\end{align}
In fact, since $C_{\widetilde{\gamma}} \to C$
is surjective, we have 
$l(\widetilde{\gamma}) \ge l(\gamma)$. 
Suppose that $l(\widetilde{\gamma})=l(\gamma)$. 
Then for each irreducible component 
$C_j \subset C$, the preimage 
$\sigma^{-1}(C_j) \cap C_{\widetilde{\gamma}}$ is also irreducible. 
Because $C_{\widetilde{\gamma}}$ is connected, this 
easily implies that $C_{\widetilde{\gamma}}$ is written as 
\begin{align*}
C_{\widetilde{\gamma}}=A \cup \tau\overline{(C_{x, i} \setminus A)}, 
\end{align*}
where $A \subset C_{x, i}$
is a connected subcurve
for some $i\in \mathbb{Z}/m\mathbb{Z}$ in the notation of 
(\ref{Cxi}) and $\tau=1 \in \mathbb{Z}/m\mathbb{Z}$. 
The curves $A \subset C_{x, i}$ and $\tau\overline{(C_{x, i} \setminus A)} \subset C_{x, i+1}$ are connected at the node $x_{1, i}=x_{2, i+1}$. 
Since $g(C_{x, i})=g(C)-1$, 
it follows that 
\begin{align*}
g(C_{\widetilde{\gamma}})
&=g(C_{x, i}) -\sharp(A \cap \overline{(C_{x, i} \setminus A)}) +1 \\
&=g(C)- \sharp(A \cap \overline{(C_{x, i} \setminus A)}) \\
&<g(C).
\end{align*}
Therefore one of (\ref{oneof1}) or (\ref{oneof2}) holds.

Now we replace $\widetilde{U}$ by a small analytic neighborhood of 
$C_{\widetilde{\gamma}}$, 
say $U_{(1)}$, and set 
\begin{align*}
\gamma_{(1)} =\widetilde{\gamma}, \ 
C_{(1)}=C_{\widetilde{\gamma}}, \ H_{(1)}=\widetilde{H} \cap U_{(1)}.
\end{align*}
Repeating the same procedures, 
we obtain the sequence of local immersions,  
\begin{align}\label{seq:U}
\cdots \to U_{(i)} \stackrel{\sigma_{(i)}}{\to}
 U_{(i-1)} \to \cdots \stackrel{\sigma_{(2)}}{\to} U_{(1)} 
\stackrel{\sigma}{\to} U_{(0)}=U, 
\end{align}
and data, 
\begin{align*}
C_{(i)} \subset U_{(i)}, \ 
\gamma_{(i)} \in H_2(C_{(i)}, \mathbb{Z}), \
H_{(i)} \subset U_{(i)},
\end{align*}
where $C_{(i)}$ is a connected nodal curve, 
$\gamma_{(i)}$ satisfies $C_{\gamma_{(i)}}=C_{(i)}$
and $H_{(i)}$ is a lift of $H$. 
Note that for each $i$, $C_{(i)} \subset U_{(i)}$ is a
cyclic
neighborhood of $C\subset X$. 
Similarly as above, either one of the following conditions hold:
\begin{align*}
&l(\gamma_{(i)})<l(\gamma_{(i+1)}), \ \mbox{ or } \\
&l(\gamma_{(i)})=l(\gamma_{(i+1)}), \ 
g(C_{(i+1)})<g(C_{(i)}).
\end{align*} 
Because we have 
\begin{align*}
l(\gamma_{(i)}) \le d(\gamma_{(i)})=d(\gamma), 
\end{align*}
the sequence (\ref{seq:U}) terminates
at some $i$, say $i=R$.  
Then we have $g(C_{(R)})=0$, i.e. 
$C_{(R)}$ is a tree of $\mathbb{P}^1$.  

Below we say that the invariant 
$\widehat{\DT}_{n, \gamma_{(i)}}^{\rm{par}}(U_{(i)})$
satisfies (\ref{DThatU'})
if the formula (\ref{DThatU'}) holds for 
$U'=U_{(i)}$, $C'=C_{(i)}$
and $\gamma'=\gamma_{(i)}$. 
By the assumption and 
 Proposition~\ref{prop:similar:trans},
the invariant $\widehat{\DT}_{n, \gamma_{(i)}}^{\rm{par}}(U_{(i)})$
satisfies (\ref{DThatU'}) when $i=R$.  
Also it is straightforward to see that the 
 argument of Proposition~\ref{prop:reduce:cov}
for $C \subset U$ can be applied to
$C_{(i)} \subset U_{(i)}$. 
Therefore 
$\widehat{\DT}_{n, \gamma_{(i)}}^{\rm{par}}(U_{(i)})$
satisfies (\ref{DThatU'}) if the same formula 
holds on any
cyclic covering $\widetilde{U}_{(i)} \to U_{(i)}$. 
By the induction argument, 
it follows that the invariant
$\widehat{\DT}_{n, \gamma_{(i)}}^{\rm{par}}(U_{(i)})$
satisfies 
(\ref{DThatU'}) for all $i$.  
Hence (\ref{log:form}) holds, and 
the formula (\ref{form:mult:loc}) holds as well.  
\end{proof}

\subsection{Euler characteristic version}\label{subsec:Euler}
For $n\in \mathbb{Z}$ and $\beta \in H_2(X, \mathbb{Z})$, 
there is also the Euler characteristic 
version of the invariant $N_{n, \beta}$,
as discussed in~\cite{Tcurve1}, \cite{Tolim2}, \cite{TodK3}.   
Namely in the definition of $N_{n, \beta}$, we replace
the Behrend function $\nu$ by the identity function. 
The resulting invariant is denoted by 
\begin{align}\label{Nnb:chi}
N_{n, \beta}^{\chi} \in \mathbb{Q}. 
\end{align}
If $\gamma$ is a one cycle on $X$, the 
Euler characteristic version of the 
local invariant $N_{n, \gamma}$
can be similarly defined, 
\begin{align*}
N_{n, \gamma}^{\chi} \in \mathbb{Q}. 
\end{align*}
The argument of Theorem~\ref{thm:main:cov} can be also applied to
the invariant $N_{n, \gamma}$, which 
is easier since we do not have to take care of the Behrend functions at all. 
Also in this case, one may expect the formula, 
\begin{align}\label{mult:chi}
N_{n, \gamma}^{\chi} =\sum_{k\ge 1, k|(n, \gamma)}
\frac{1}{k^2} N_{1, \gamma/k}^{\chi}. 
\end{align}
However unfortunately, the above formula 
is known to be false
as the following example indicates: 
\begin{exam}\label{exam:counter}
Let $C\subset X$ be a smooth rational curve
whose normal bundle is $\oO(-1) \oplus \oO(-1)$. 
Then the same computation of $N_{0, m[C]}$
in~\cite[Example~4.14]{Todpara} shows that 
\begin{align*}
N_{0, m[C]}^{\chi}=\frac{(-1)^{m-1}}{m^2}.
\end{align*}
On the other hand, $N_{1, m[C]}^{\chi}=1$ if $m=1$
and $0$ if $m\ge 2$. Hence (\ref{mult:chi}) does not hold. 
\end{exam}
Although the formula (\ref{mult:chi}) is not true in general, 
there are some situations in which the formula (\ref{mult:chi})
should hold, as discussed in~\cite{TodK3}. 
Similarly to Theorem~\ref{thm:main:cov}, such a 
case can be reduced to the cases of trees of 
$\mathbb{P}^1$ on cyclic neighborhoods of 
$C\subset X$. 
Note that, for a cyclic neighborhood 
$C' \subset U'$ of $C\subset X$ and 
$\gamma' \in H_2(C', \mathbb{Z})$, we can also define the
Euler characteristic invariant, 
\begin{align*}
N_{n, \gamma'}^{\chi}(U') \in \mathbb{Q}, 
\end{align*}
by replacing the Behrend function by the identity 
function in the definition of $N_{n, \gamma'}(U')$. 
We have the following theorem: 
\begin{thm}\label{thm:cov:Eu}
Let $X$ be a smooth projective Calabi-Yau 3-fold over 
$\mathbb{C}$, $C \subset X$ a reduced rational
curve with at worst nodal singularities,
 and take $\gamma \in H_2(C, \mathbb{Z})$. 
Suppose that for any cyclic neighborhood
$(C'\subset U') \stackrel{\sigma'}{\to} (C\subset X)$
with $C'$ a tree of $\mathbb{P}^1$,
and an element $\gamma' \in H_2(C', \mathbb{Z})$
with $\sigma_{\ast}'\gamma'=\gamma$, 
the following equality holds: 
\begin{align}\label{mult:chi2}
N_{n, \gamma'}^{\chi}(U') =\sum_{k\ge 1, k|(n, \gamma')}
\frac{1}{k^2} N_{1, \gamma'/k}^{\chi}(U'). 
\end{align}
Then $N_{n, \gamma}^{\chi}$ satisfies the formula (\ref{mult:chi}). 
\end{thm}
\begin{proof}
The same proof of Theorem~\ref{thm:main:cov} works. The only 
difference is that we do not have to take care of the Behrend 
functions in deducing the equivalence between the formula 
(\ref{mult:chi2}) and the formula for parabolic 
stable pair invariants. Namely for any 
cyclic neighborhood $C' \subset U'$ of $C\subset X$
with a lift $H' \subset U'$ of $H\subset X$,
and $\gamma'\in H_2(C', \mathbb{Z})$, we 
can define the 
Euler characteristic versions of 
parabolic stable pair invariants, 
\begin{align*}
\DT_{n, \gamma'}^{\rm{par}, \chi}(U') \in \mathbb{Z}, \quad 
\widehat{\DT}_{n, \gamma'}^{\rm{par}, \chi}(U') \in \mathbb{Q},
\end{align*}
by replacing the Behrend function by the 
identity function in the definitions of (\ref{inv:onU'}), (\ref{DT:hat2})
respectively. 
The same proof of~\cite[Corollary~4.18]{Todpara}
shows that the formula (\ref{mult:chi2}) is equivalent to the 
formula, (without assuming Conjecture~\ref{conj:crit},)
\begin{align*}
\widehat{\DT}_{n, \gamma'}^{\rm{par}, \chi}(U')
=(\gamma' \cdot H') \sum_{k\ge 1, k|(n, \gamma')} \frac{1}{k^2}
N_{1, \gamma'/k}^{\chi}(U'). 
\end{align*}
Then we can apply the same induction argument as in the proof of 
Theorem~\ref{thm:main:cov}, and conclude the assertion. 
\end{proof}

\section{Applications}\label{sec:apply}
In this section, we apply Theorem~\ref{thm:main:cov} to
prove the multiple cover formula 
under some situations. 
\subsection{$0$-super rigid and surface type neighborhoods}
Let $C \subset U\subset X$ be as in the previous sections. 
We study the local multiple cover formula in 
the following situations:
$C \subset U$ is $0$-super rigid or 
$C \subset U$ is of surface type. 
These concepts are defined as follows: 
\begin{defi}\label{def:rigid:surface}

(i) We say $C \subset U$ is $0$-super rigid if for any 
projective curve $C'$ of arithmetic genus $0$ and 
a local immersion $f \colon C' \to C$, we have 
\begin{align*}
H^0(C', f^{\ast}N_{C/U})=0.
\end{align*} 

(ii) We say $C \subset U$ is of surface type 
if there is a complex surface $U_0$
and an analytic neighborhood $\Delta$ of 
$0 \in \mathbb{C}$ such that 
\begin{align*}
U \cong U_0  \times \Delta, 
\end{align*}
and $C$ is
contained in $U_0 \times \{0\}$
\end{defi}
The $0$-super rigidity is
a genericity condition for the
pair $(C \subset U)$, and 
a concept adapted in~\cite{BPrig}. 
The following proposition shows that there are
some situations in which the assumptions in Theorem~\ref{thm:main:cov}
are satisfied. 
\begin{prop}\label{prop:compN}
Let $C' \subset U'$ be a cyclic neighborhood 
of $C\subset U$, with $C'$ a tree of $\mathbb{P}^1$. 

(i) Suppose that $C \subset U$ is $0$-super rigid 
and $C'$ is a chain of $\mathbb{P}^1$, 
say $C_1', \cdots, C_N'$. Then 
$C' \subset U'$ satisfies the condition of Conjecture~\ref{conj:crit}. 
Moreover, for any $n\in \mathbb{Z}$, 
and $a_1', \cdots, a_N' \in \mathbb{Z}_{\ge 1}$, we have 
\begin{align}\label{N:rigid}
N_{n, a_1'[C_1']+ \cdots +a_N'[C_N']}(U')
=\left\{\begin{array}{cc} 1/k^2, & a_1'= \cdots =a_N'=k, \ k|n,  \\
0, & \mbox{ otherwise. }    
\end{array} \right.
\end{align}
In particular, 
the invariant 
$N_{n, \gamma'}(U')$ satisfies the formula (\ref{mult:U'}). 

(ii) Suppose that $C \subset U$ is of surface type 
and the dual graph of $C'$ is of ADE type. 
Then $C' \subset U'$ satisfies the condition of 
Conjecture~\ref{conj:crit}. 
If $C'$ is a chain of $\mathbb{P}^1$, say 
$C_1', \cdots, C_N'$, then 
for any $n\in \mathbb{Z}$, 
and $a_1', \cdots, a_N' \in \mathbb{Z}_{\ge 1}$, we have 
\begin{align}\label{N:sur}
N_{n, a_1'[C_1']+ \cdots +a_N'[C_N']}(U')
=\left\{\begin{array}{cc} -1/k^2, & a_1'= \cdots =a_N', \ k|n, \\
0, & \mbox{ otherwise. }    
\end{array} \right.
\end{align}
In particular, 
the invariant 
$N_{n, \gamma'}(U')$ satisfies the formula (\ref{mult:U'}). 
\end{prop}
\begin{proof}
(i) The proof of~\cite[Lemma~3.1]{BPrig} shows that 
each irreducible component $C_i'$ 
is a $(-1, -1)$-curve in $U'$, and 
there is a bimeromorphic contraction, 
\begin{align*}
f\colon U' \to U'', 
\end{align*}
which contracts $C'$ to a $cA_N$-singularity $0 \in U''$. 
The argument of Van den Bergh~\cite{MVB} works in our situation, and 
we have the derived equivalence, 
\begin{align*}
\Phi \colon 
D^b \Coh_{C'}(U') \cong D^b \mathrm{Mod}_{\rm{nil}}(A),
\end{align*}
for some non-commutative $\oO_{U''}$-algebra $A$, 
such that $\Phi^{-1} \mathrm{Mod}_{\rm{nil}}(A)$
corresponds to Bridgeland's perverse
coherent sheaves with $0$-perversity~\cite{Br1}.  
Here $\Coh_{C'}(U')$ is the category of 
coherent sheaves on $U'$ supported on $C'$, 
and $\mathrm{Mod}_{\rm{nil}}(A)$ is 
the category of finite dimensional nilpotent
right $A$-modules.  
Let us take $\lL \in \Pic(U')$
such that $\lL|_{C'}$ is an ample line bundle. 
By the construction of perverse coherent sheaves
in~\cite{Br1}, 
an object $E\in \Coh_{C'}(U)$ satisfies 
$\Phi(E) \in \mathrm{Mod}_{\rm{nil}}(A)$
if and only if $R^1 f_{\ast} E=0$. The latter
condition is satisfied if we replace $E$ by 
$E\otimes \lL^{\otimes m}$ for $m\gg 0$. 
Since any object
$[E] \in \mM_n(U', \beta')$
is supported on $C'$, the above argument 
implies that 
the moduli stack $\mM_n(U', \beta')$ is regarded
as an 
analytic open substack of objects in 
$\Phi^{-1}\mathrm{Mod}_{\rm{nil}}(A) \otimes \lL^{\otimes -m}$
for $m\gg 0$. 
On the other hand, the algebra $A$ is a Calabi-Yau 3-algebra. 
Hence the completion $\widehat{A}$ of $A$ at $0\in U''$
is written as a completion of a path algebra of a
quiver $Q$ with a super potential $W$
by~\cite{Bergh}.
Since each component of the moduli stack of representations of 
$(Q, W)$ is written as a quotient stack of a 
critical locus of some holomorphic function on a finite 
dimensional vector space, (cf.~\cite[Subsection~7.2]{JS},)
we conclude that $\mM_n(U', \beta')$ 
satisfies the condition of Conjecture~\ref{conj:crit}. 

Next let us take $n\in \mathbb{Z}$ and
$a_1', \cdots, a_N' \ge 1$. 
By the argument in~\cite[Proposition~2.10]{BKL},
there is a family of complex manifolds 
$U_t'$ for $t\in \mathbb{C}$
such that $U_0'=U'$ and 
curves in 
$U_{\varepsilon}'$ for $0<\varepsilon \ll 1$
consist of only $(-1, -1)$-curves. 
A subcuve $C'' \subset C'$ deforms to 
a curve in $U_{\varepsilon}'$ if and only if 
$C''$ is a sub $\mathbb{P}^1$-chain of $C'$. 
Let $C_{\varepsilon}' \subset U_{\varepsilon}'$ be
a $(-1, -1)$-curve, obtained by deforming 
$C'$. For $(n, k) \in \mathbb{Z}^{\oplus 2}$
with $k\ge 1$, we have $N_{n, k[C_{\varepsilon}']} \neq 0$
only if $k|n$, and in this case we have
(cf.~\cite[Example~4.14]{Todpara},) 
\begin{align*}
N_{n, k[C_{\varepsilon}']}(U_{\varepsilon}')=\frac{1}{k^2}.
\end{align*}
 On the other hand, the invariant 
$N_{n, \beta'}(U')$ is invariant under deformation of 
$U'$ by~\cite[Corollary~5.28]{JS}. (See Remark~\ref{rmk:ncpt} below.)
Hence 
the LHS of (\ref{N:rigid}) 
is non-zero only if $a_1'=\cdots =a_N'(=k)$, 
$k|n$, and equal to $1/k^2$ in this case. 

(ii) 
If $C\subset U$ is of surface type, 
then the 
cyclic neighborhood $C' \subset U'$ 
is also of surface type:
there is a complex surface
$U_0'$ such that $U' \cong U_0' \times \Delta$
and $C' \subset U_0' \times \{0\}$. 
Then $C'$ is a 
tree of $\mathbb{P}^1$ of ADE type 
in the surface $U_0'$, hence 
there is a bimeromorphic morphism to 
a singular complex surface $U_0''$,
\begin{align*}
U_0' \to U_{0}'', 
\end{align*}
whose exceptional locus is $C'$. 
Since $U_0''$ is a small analytic neighborhood of $C'$, 
$U_0''$ is an analytic neighborhood of a singular point 
in $U_0''$, which is isomorphic to an analytic 
neighborhood of the quotient singularity 
$\mathbb{C}^2/G$
for a finite subgroup 
in $G \subset \SL(2, \mathbb{C})$. 

Let 
\begin{align}\label{minimal:W}
f \colon 
V \to \mathbb{C}^2 /G,
\end{align}
be the minimal resolution of singularities. 
Note that $C'$ is regarded as an exceptional locus 
of (\ref{minimal:W}). 
The above argument shows that 
$U'$ is isomorphic to an analytic neighborhood
of $C'$ in $V\times \mathbb{C}$, 
where $C'$ lies in $V\times \{0\}$. 
As explained in~\cite[Subsection~2.2]{GY}, 
we have the derived equivalence, 
\begin{align*}
\Phi \colon 
D^b \Coh(V\times \mathbb{C}) \cong 
D^b \mathrm{Rep}(Q, W),
\end{align*}
for a certain quiver $Q$ with a superpotential $W$, 
and $\Phi^{-1} \mathrm{Rep}(Q, W)$
is Bridgeland's perverse coherent sheaves. 
Then the same argument of (i) shows that 
$\mM_n(U', \beta')$ satisfies the condition of Conjecture~\ref{conj:crit}. 

Let us compute the LHS of (\ref{N:sur}) when 
$C'$ is a chain of $\mathbb{P}^1$. 
In the surface type case, the moduli space of 
stable pairs (\ref{pair:JS})
in Remark~\ref{rmk:ncpt} below is not compact, 
so the deformation argument is more subtle. 
Instead, we use the explicit result of 
the computation of DT type invariants on 
$V\times \mathbb{C}$ in~\cite{GY}.  
In~\cite{GY}, it is proved that the
 generating series of PT 
invariants is written as a Gopakumar-Vafa form, 
that is a local version of the conjecture in~\cite[Conjecture~6.2]{Tsurvey}.
By~\cite[Theorem~6.4]{Tsurvey}, Conjecture~\ref{conj:mult2} is equivalent 
to~\cite[Conjecture~6.2]{Tsurvey}, 
hence $N_{n, \gamma'}(U')$ satisfies (\ref{mult:U'}) 
for any $n\in \mathbb{Z}$ and $\gamma' \in H_2(C', \mathbb{Z})$. 
By~\cite[Corollary~1.6]{GY} and~\cite[Theorem~6.4]{Tsurvey},
for $a_1', \cdots, a_N' \in \mathbb{Z}_{\ge 1}$, we have 
\begin{align*}
N_{1, a_1'[C_1']+ \cdots +a_N'[C_N']}(U')
=\left\{\begin{array}{cc} -1, & a_1'= \cdots =a_N'=1, \\
0, & \mbox{ otherwise. }    
\end{array} \right.
\end{align*}
Therefore the result of (ii) holds. 
\end{proof}

\begin{rmk}\label{rmk:ncpt}
The argument of~\cite[Corollary~5.28]{JS}
works when the ambient space $U'$ is a 
projective Calabi-Yau 3-fold. In our case, 
$U'$ is an analytic small neighborhood of $C'$, 
so we need to modify the argument. 
Let $\lL$ be a line bundle on $U'$
such that $\lL|_{C'}$ is ample. 
Then we consider moduli space of pairs, 
\begin{align}\label{pair:JS}
(F, u), 
\end{align}
where $F$ is a compactly supported 
one dimensional coherent sheaf on $U'$, 
and $u\in H^0(F\otimes \lL^{\otimes m})$ for 
$m\gg 0$, satisfying the stability 
condition as in~\cite[Definition~5.20]{JS}. 
If $C' \subset U'$ is a chain of $(-1, -1)$-curves, 
then any sheaf $F$ as above is 
supported on $C'$, hence the moduli space
of pairs (\ref{pair:JS}) is a projective scheme. 
Let us consider the object, 
\begin{align*}
E=(\lL^{\otimes -m} \stackrel{u}{\to} F) \in D^b \Coh(U'). 
\end{align*}
Although $U'$ is non-compact, the groups 
$\Ext_{U'}^{i}(E, E)$ for $i=1, 2$ are finite dimensional, 
hence determine a symmetric perfect obstruction 
theory on the moduli space of pairs (\ref{pair:JS})
as in~\cite[Theorem~5.23]{JS}. 
Then the deformation invariance of $N_{n, \beta'}(U')$
follows from the same argument of~\cite[Corollary~5.28]{JS}. 
\end{rmk}

\subsection{Local generalized DT invariants on a 
nodal rational curve of type $I_N$}
In this subsection, using the results in the previous 
subsection, we compute some 
generalized DT invariants which have not been 
computed so far. 
Recall that 
a nodal curve $C$ is 
called type $I_N$ if $C$ is one of the following: 
\begin{itemize}
\item
$C$ is of type $I_1$ if $C$ is an irreducible 
rational curve with one node. 
\item $C$ is of type $I_N$ for $N\ge 2$ if 
$C$ is a circle of irreducible components 
$C_1, \cdots, C_N$
such that $C_i \cong \mathbb{P}^1$
for all $i$. 
\end{itemize}
Note that the above notation is used in
the Kodaira's 
classification of singular fibers of 
elliptic fibrations. 
(See Figure~\ref{fig:three}.)

\begin{figure}[htbp]
 \begin{center}
  \includegraphics[width=90mm]{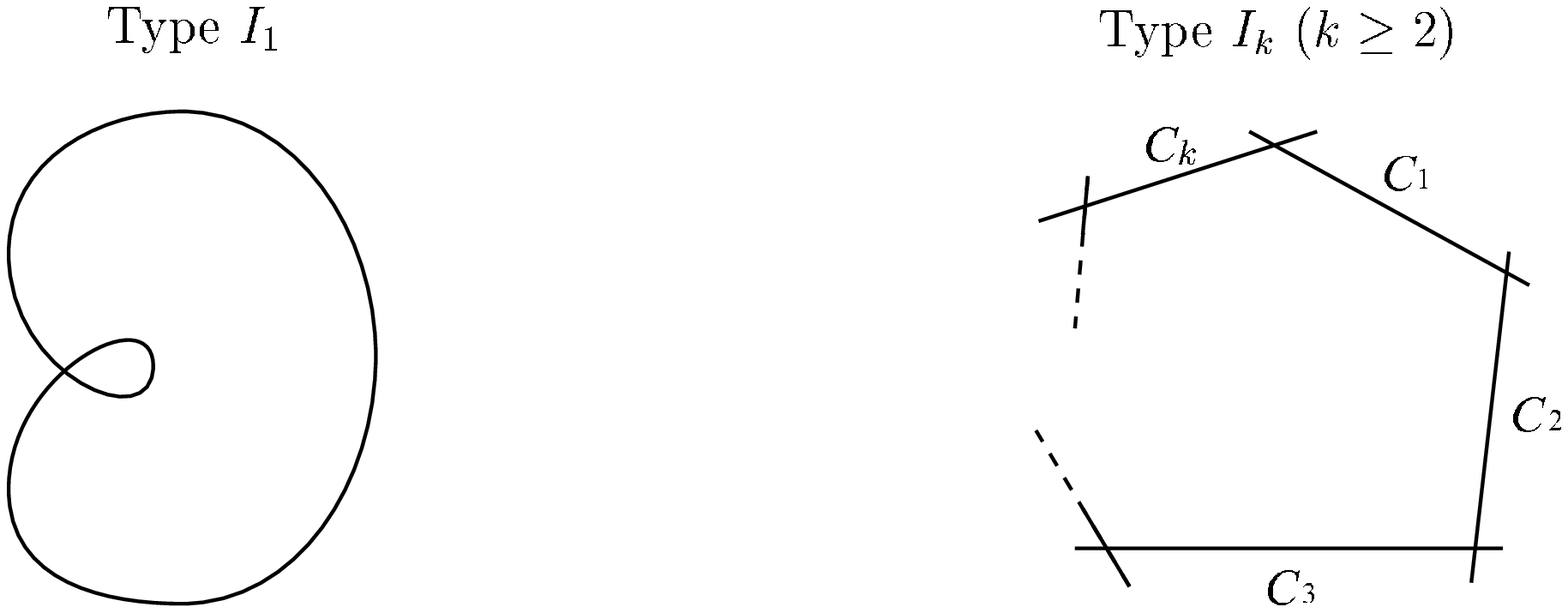}
 \end{center}
 \caption{}
 \label{fig:three}
\end{figure}
We have the following theorem. 
\begin{thm}\label{thm:typeI}
Let $C$ be a nodal curve of type $I_N$
with irreducible components $C_1, \cdots, C_N$, 
which is embedded into a Calabi-Yau 3-fold $X$. 
Suppose that an analytic neighborhood 
$C\subset U \subset X$ is either $0$-super rigid or 
of surface type. Then for $n\in \mathbb{Z}$
and $a_1, \cdots, a_N \in \mathbb{Z}_{\ge 1}$, 
the invariant $N_{n, a_1[C_1]+ \cdots +a_{N}[C_N]}$
is non-zero only if $a_1= \cdots =a_N$, say $m$. 
In this case, we have 
\begin{align*}
N_{n, m[C_1]+ \cdots +m[C_N]}
=\left\{ \begin{array}{cc}
\sum_{k\ge 1, k|(n, m)}
N/k^2, & 0\mbox{-super rigid case, } \\
-\sum_{k\ge 1, k|(n, m)} N/k^2, 
& \mbox{ surface type case. }
\end{array}   \right. 
\end{align*}
\end{thm}
In particular, the invariant $N_{n, a_1[C_1]+\cdots +a_N[C_N]}$
satisfies the formula (\ref{form:mult:loc}). 
\begin{proof}
Let $(C', U')$ be a cyclic neighborhood of 
$C \subset U$. 
Then the construction of cyclic coverings 
in Subsection~\ref{subsec:cyclic} shows that, 
if $C'$ is a tree of $\mathbb{P}^1$, 
then it must be a chain of $\mathbb{P}^1$. 
By Proposition~\ref{prop:compN}, the assumptions in 
Theorem~\ref{thm:main:cov} are satisfied, hence 
the invariant $N_{n, a_1[C_1]+ \cdots +a_{N}[C_N]}$
satisfies the formula (\ref{form:mult:loc}). 
It is enough to compute 
$N_{1, a_1[C_1]+ \cdots +a_{N}[C_N]}$, 
which follows from 
the computation on cyclic neighborhoods in Proposition~\ref{prop:compN}
and the comparison formula under 
the covering (\ref{formula2}).  
For instance if $C\subset U$ is $0$-super rigid, then 
it easily follows that 
\begin{align*}
N_{1, a_1[C_1]+ \cdots +a_{N}[C_N]}
=\left\{ \begin{array}{cc} 
N, & a_1= \cdots =a_N=1, \\
0, & \mbox{ otherwise. }
\end{array}  \right. 
\end{align*}
The case of surface type is similar. 
\end{proof}

\subsection{Local generalized DT invariants on 
irreducible nodal rational curves}
In this subsection, using Lemma~\ref{N:primitive}, Proposition~\ref{prop:compN}
and Theorem~\ref{thm:main:cov}, we 
study the local multiple cover formula 
of generalized DT invariants on 
irreducible nodal curves. 
We have the following result:
\begin{thm}\label{prop:prime}
Let $C$ be an irreducible rational curve
with at worst nodal singularities, 
embedded into a Calabi-Yau 3-fold $X$. 
For a sufficiently small analytic 
neighborhood $C\subset U \subset X$, 
suppose that it is $0$-super rigid or 
of surface type. 
Moreover assume that 
any cyclic neighborhood $C' \subset U'$
with $C'$ tree of $\mathbb{P}^1$
satisfies the condition of Conjecture~\ref{conj:crit}. 
The for any prime number
$p$, the invariant $N_{n, p[C]}$
satisfies the formula (\ref{form:mult:loc}). 
\end{thm}
\begin{proof}
Let $(C'\subset U') \stackrel{\sigma'}{\to} (C\subset X)$ be a cyclic
neighborhood of $C\subset X$ with 
$C'$ tree of $\mathbb{P}^1$.
By Theorem~\ref{thm:main:cov}, 
it is enough to show that
the invariant $N_{n, \gamma'}(U')$ satisfies the formula (\ref{mult:U'})
for any $n\in \mathbb{Z}$ and $\gamma' \in H_2(C', \mathbb{Z})_{>0}$
with $\sigma'_{\ast}\gamma'=p[C]$. 
Let $C_1', \cdots, C_N'$ be the irreducible 
components of $C'$, and we write 
$\gamma' \in H_2(C', \mathbb{Z})_{>0}$
as $\gamma'=\sum_{i=1}^{N} a_i'[C_i']$
for $a_i' \in \mathbb{Z}_{\ge 0}$.  
Then $\sigma_{\ast}'\gamma'=p[C]$ is equivalent to that 
\begin{align}\label{a:prime}
a_1' + \cdots + a_N'=p.
\end{align}
Suppose that there are at least two 
$1\le i\le N$ with $a_i'>0$. Then, 
since $p$ is a prime number,
the equality (\ref{a:prime}) implies 
that  
\begin{align*}
\mathrm{g.c.d.}(a_1', \cdots, a_N')=1,
\end{align*} 
i.e. $\gamma' \in H_2(C', \mathbb{Z})$ is primitive. 
Then the invariant $N_{n, \gamma'}(U')$ satisfies 
(\ref{mult:U'}) by Lemma~\ref{N:primitive}. 
If there is only one $1\le i\le N$ with 
$a_i'>0$, then we can assume that $C'$ is a single 
$\mathbb{P}^1$. In this case, the invariant $N_{n, \gamma'}(U')$
satisfies the formula (\ref{mult:U'}) by Proposition~\ref{prop:compN}. 
\end{proof}
In the situation of the above 
theorem, we are not able to prove 
the condition of Conjecture~\ref{conj:crit} on 
cyclic neighborhoods. However we 
don't have to take care of this condition 
in the Euler characteristic version. 
We have the following:
\begin{thm}\label{prop:prime2}
Let $C$ be an irreducible rational curve
with at worst nodal singularities, 
embedded into a Calabi-Yau 3-fold $X$. 
For a sufficiently small analytic 
neighborhood $C\subset U \subset X$, 
suppose that it is 
of surface type. 
The for any prime number
$p$, the invariant $N_{n, p[C]}^{\chi}$
satisfies the formula (\ref{mult:chi}). 
\end{thm}
\begin{proof}
The same proof of Theorem~\ref{prop:prime}
works, using Theorem~\ref{thm:cov:Eu} instead of Theorem\ref{thm:main:cov}. 
In the notation of the proof of Theorem~\ref{prop:prime}, 
suppose that there are at least two $1\le i\le N$
with $a_i'>0$. Then 
the same proof of Lemma~\ref{N:primitive}
shows the equality $N_{n, \gamma'}^{\chi}(U')=N_{1, \gamma'}^{\chi}(U')$. 
Otherwise, we may assume that $C'$ is a single 
$\mathbb{P}^1$. In this case, 
the invariant $N_{n, p[C']}^{\chi}(U')$ can 
be checked to satisfy (\ref{mult:chi2})
by comparing the formula in~\cite[Theorem~1.3]{Tolim2}
and the Euler characteristic version of the 
formula in~\cite[Theorem~1.2]{GY}. 
\end{proof}

\begin{rmk}
The result of Theorem~\ref{prop:prime2} is not 
true for a $0$-super rigid case, 
as we discussed in Example~\ref{exam:counter}. 
\end{rmk}

In the situation of Theorem~\ref{prop:prime2}, we 
can say more for the invariants $N_{n, m[C]}^{\chi}$
when $m$ is small:
\begin{lem}\label{mult:small}
In the situation of Theorem~\ref{prop:prime2}, 
the invariant $N_{n, m[C]}^{\chi}$ 
satisfies the formula (\ref{mult:chi})
when $m\le 10$. 
\end{lem}
\begin{proof}
First suppose that $m<10$. 
In the notation of the 
proof of Theorem~\ref{prop:prime}, 
suppose that 
$\gamma'=\sum_{i=1}^{N}a_i'[C_i']$ 
satisfies $\sigma_{\ast}\gamma'=m[C]$, 
i.e. 
\begin{align*}
a_1' + \cdots +a_N'=m.
\end{align*}
Then we
have either $\mathrm{g.c.d.}(a_1', \cdots, a_N')=1$
or $\gamma$ is supported on an
ADE configuration of $\mathbb{P}^1$. 
As in the proof of Theorem~\ref{prop:prime2}, 
using the results in~\cite{GY}, \cite{Tolim2}, 
it is straightforward to check 
the Euler characteristic version of 
the results in Proposition~\ref{prop:compN} (ii).
Therefore the result follows from
the Euler characteristic 
version of Lemma~\ref{N:primitive}
and Theorem~\ref{thm:cov:Eu}. 
When $m=10$, we have the  
following exceptional case: 
$N=5$, 
$C'_1, \cdots, C_5'$ satisfy
\begin{align*}
C_1' \cdot C_i'=1, \ (i\ge 2), \quad
C_i' \cdot C_j'=0, \ (i, j \ge 2). 
\end{align*} 
and $\gamma'=2[C']$. 
In this case, $C_1', \cdots, C_5'$ is not 
an ADE configuration.
However we can check that $N_{n, 2[C']}^{\chi}(U')$
satisfies (\ref{mult:chi2})
by a direct calculation as in
~\cite[Proposition~6.9]{TodK3}. 
In fact, 
by the Riemann-Roch theorem, 
one can show that there is no stable 
sheaf $F'$ on $U'$ with $[F']=\gamma'$. 
Then a computation similar to~\cite[Proposition~6.9]{TodK3}
works, whose detail is
left to the readers.    
\end{proof}

\begin{rmk}\label{rmk:small}
The result of Lemma~\ref{mult:small} can be 
generalized as follows. 
Let $C$ be a reduced (not necessary irreducible)
rational
 curve
with at worst nodal singularities, 
 which is embedded into a 
Calabi-Yau 3-fold $X$,
 Suppose that a sufficiently 
small analytic neighborhood $C\subset U\subset X$
is of surface type. Let $C_1, \cdots, C_N$
be the irreducible components of $C$. Then 
the invariant $N_{n, a_1[C_1]+\cdots +a_N[C_N]}^{\chi}$
satisfies the formula (\ref{mult:chi})
if $a_1, \cdots, a_N$ satisfies
\begin{align*}
a_1 + \cdots +a_N \le 10.
\end{align*}
The proof is same as in Lemma~\ref{mult:small}. 
\end{rmk}

\subsection{Euler characteristic 
invariants on K3 surfaces}
Let $S$ be a smooth projective 
K3 surface over $\mathbb{C}$, 
i.e.
\begin{align*}
K_S=\oO_S, \quad H^1(S, \oO_S)=0, 
\end{align*}
and 
$X$ is the total space of the canonical 
line bundle on $S$, i.e. 
\begin{align*}
X=S\times \mathbb{C}.
\end{align*}
In~\cite{TodK3}, we established a formula 
relating Euler characteristic of 
the moduli space of 
PT stable pairs and Joyce type 
Euler characteristic invariants 
counting semistable sheaves on the fibers of the projection 
\begin{align*}
X=S\times \mathbb{C} \to \mathbb{C}.
\end{align*}
For a vector 
\begin{align*}
v=(r, \beta, n) \in H^0(S, \mathbb{Z}) 
\oplus H^2(S, \mathbb{Z}) \oplus H^4(S, \mathbb{Z}),
\end{align*}
the latter invariant was denoted by 
\begin{align*}
J(r, \beta, n) \in \mathbb{Q}. 
\end{align*}
The invariant $J(v)$ is the Euler characteristic version of 
DT type invariants on $X$, counting 
$p^{\ast}\omega$-semistable sheaves $F \in \Coh(X)$, 
with compact supports, and satisfies 
\begin{align*}
\ch(p_{\ast}F) \sqrt{\td_S} =v. 
\end{align*}
Here $p\colon X=S\times \mathbb{C} \to S$ is the first
projection, and $\omega$ is an ample divisor on $S$. 
In the notation of Subsection~\ref{subsec:Euler}, 
after identifying $H^2(S, \mathbb{Z})$
with $H_2(X, \mathbb{Z})$, we have 
\begin{align}\label{J=N}
J(0, \beta, n)=N_{n, \beta}^{\chi}. 
\end{align}
If $v\in H^{\ast}(S, \mathbb{Z})$ is 
a primitive algebraic class, then 
$J(v)$ is written as 
\begin{align}\label{J:primitive}
J(v)=\chi(\Hilb^{(v, v)/2 +1}(S)). 
\end{align}
(cf.~\cite[Equation~(65)]{TodK3}.)
Here $\Hilb^m(S)$ is the Hilbert scheme of $m$-points in 
$S$ and $(\ast, \ast)$ is the Mukai inner product, 
\begin{align*}
((r_1, \beta_1, n_1), (r_2, \beta_2, n_2))
=\beta_1 \cdot \beta_2 -r_1 n_2 -r_2 n_1.
\end{align*}
The RHS of (\ref{Hilb}) is determined 
by the G$\ddot{\rm{o}}$ttsche's formula~\cite{Got},
\begin{align*}
\sum_{m\ge 0} \chi(\Hilb^{m}(S))q^d 
= \prod_{m\ge 1} \frac{1}{(1-q^m)^{24}}. 
\end{align*}
If $v$ is not necessary primitive, 
we proposed the following multiple 
cover conjecture
in~\cite[Conjecture~1.3]{TodK3}:
\begin{conj}{\bf (\cite[Conjecture~1.3]{TodK3})}
If $v\in H^{\ast}(S, \mathbb{Z})$
is an algebraic class, 
we have the equality, 
\begin{align}\label{Hilb}
J(v)=\sum_{k\ge 1, k|v}
\frac{1}{k^2} \chi(\Hilb^{(v/k, v/k)/2 +1}(S)). 
\end{align}
\end{conj}
Using the results in this paper, we can 
give some evidence of the conjecture. 
Below we say a (not necessary reduced)
curve $C \subset S$ is 
\textit{rational}, or has
\textit{at worst nodal singularities} 
if the reduced curve $C^{\rm{red}}$
satisfies the corresponding property. 
We have the following theorem:
\begin{thm}\label{thm:K3}
Let $S$ be a smooth projective K3 surface
over $\mathbb{C}$ and $X=S\times \mathbb{C}$. 
Suppose that $\Pic(S)$ is generated by an 
ample line bundle $L$ on $S$, such that 
any rational member in $\lvert L \rvert$ 
has at worst nodal singularities. 
Then the invariant $J(0, m c_1(L), n)$
satisfies the formula (\ref{Hilb}) 
if $m\le 10$ or $m$ is a prime number. 
In particular, if $L^2=2d-2$
for $d\in \mathbb{Z}$ and $p$ is a prime number, 
we have 
\begin{align}\label{Jformula:pd}
J(0, p c_1(L), 0)
= \chi(\Hilb^{(d-1)p^2 +1}(S))
+\frac{1}{p^2} \chi(\Hilb^{d}(S)). 
\end{align}
\end{thm}
\begin{proof}
Let $\gamma$ be a one cycle on 
$X$ whose support is compact. 
We show that $N_{n, \gamma}^{\chi}$
satisfies the formula (\ref{mult:chi})
when the homology class of $\gamma$
coincides with $m c_1(L)$
under the isomorphism 
$H_2(X, \mathbb{Z}) \cong H^2(S, \mathbb{Z})$. 

As in the proof of Theorem~\ref{thm:main:cov}, we may 
assume that $\gamma$ is connected. 
Then $\gamma$ is supported on 
$S\times \{t\}$ for some $t \in \mathbb{C}$, hence 
regarded as $\gamma \in \lvert mL \rvert$.
We write $\gamma$
as 
\begin{align*}
\gamma=a_1 [C_1] +\cdots +a_N[C_N], 
\end{align*}
where $C_1, \cdots, C_N$ are irreducible 
curves in $S \times \{t\}$. 
By Lemma~\ref{lem:higher}, we may assume that all 
$C_i$ have geometric genus zero. 
By the assumption, the reduced curve 
$C=\cup_{i=1}^{N}C_i$ has at worst nodal singularities,  
and there is an analytic neighborhood $C \subset U \subset X$
of surface type. 
Note that $C_i \in \lvert m_i L \rvert$
for some $m_i \in \mathbb{Z}_{\ge 1}$, and we have 
\begin{align*}
a_1 m_1 + \cdots + a_N m_N=m.
\end{align*}
Since $m\le 10$ or $m$ is a prime number, 
either one of the following conditions hold: 

(i) $a_1 + \cdots + a_N \le 10$. 

(ii) $\mathrm{g.c.d.}(a_1, \cdots, a_N)=1$.

(iii) $N=1$, $a_1=m$, $m_1=1$.

In these cases, the formula (\ref{mult:chi}) 
follows from Remark~\ref{rmk:small}, Lemma~\ref{N:primitive}
and Theorem~\ref{prop:prime2}. 

The proof of Lemma~\ref{lem:glo:loc}
in~\cite[Corollary~4.11]{Todpara}
is also applied to the Euler characteristic version
in our situation. 
Consequently, we have the global 
multiple cover formula, 
\begin{align}\label{global:K3}
N_{n, mc_1(L)}^{\chi} =
\sum_{k\ge 1, k|(n, m)}
\frac{1}{k^2} N_{1, mc_1(L)/k}^{\chi}. 
\end{align}
On the other hand, since 
$v=(0, a c_1(L), 1) \in H^{\ast}(X, \mathbb{Z})$ is primitive
for any $a \in \mathbb{Z}$, we have 
\begin{align*}
N_{1, ac_1(L)}^{\chi}
=\chi(\Hilb^{(v, v)/2+1}(S)),
\end{align*}
by (\ref{J:primitive}) and (\ref{J=N}).
Also noting (\ref{J=N}) in the
LHS of (\ref{global:K3}), 
the invariant 
$J(0, m c_1(L), n)$ satisfies the formula (\ref{Hilb}). 
\end{proof}
\begin{rmk}
The assumption of the K3 surface $S$
in Theorem~\ref{thm:K3}
is a genericity condition
for polarized K3 surfaces. 
Namely for $d\in \mathbb{Z}$, let us consider
the moduli space of polarized K3 surfaces
$(S, L)$ with $L^2=2d-2$. 
Then $(S, L)$ satisfies the assumption in Theorem~\ref{thm:K3}
when $(S, L)$ is a general point of the above 
moduli space. 
\end{rmk}

\begin{rmk}
In~\cite[Proposition~6.9]{TodK3}, 
we proved the 
formula (\ref{Jformula:pd})
for $d=2$ and $p=2$. 
Even in this case, 
it is not obvious to compute the
LHS of (\ref{Jformula:pd}) directly. 
Indeed in~\cite[Proposition~6.9]{TodK3}, 
we used the result by Mozgovoy~\cite{Moz}. 
The proof of Theorem~\ref{thm:K3}
does not require the result of~\cite{Moz}, 
and can be applied to more general cases. 
\end{rmk}
\begin{rmk}
Note that there is a weight two Hodge structure on 
$H^{\ast}(S, \mathbb{C})$ by 
\begin{align*}
H^{\ast 2, 0}=H^{2, 0}, \ H^{\ast 0, 2}=H^{0, 2}, \\
H^{\ast 1, 1}=H^{0, 0} \oplus H^{1, 1} \oplus H^{2, 2}. 
\end{align*}
Let $G$
be the group of Hodge isometries of 
$H^{\ast}(S, \mathbb{Z})$. 
Then for any $g\in G$
and $v\in H^{\ast}(S, \mathbb{Z})$, 
it is proved in~\cite[Theorem~4.21]{TodK3} that  
\begin{align}\label{Jgv=Jv}
J(gv)=J(v). 
\end{align}
Let $v=(0, mc_1(L), n) \in H^{\ast}(S, \mathbb{Z})$
be as in the statement of Theorem~\ref{thm:K3}. 
The result of Theorem~\ref{thm:K3}
and the formula (\ref{Jgv=Jv}) imply that 
$J(gv)$ also satisfies the 
formula (\ref{Hilb}) for any $g\in G$. 
\end{rmk}

\section{Identity of Behrend functions}\label{subsec:Behrend}
In this section, we give a proof of Lemma~\ref{lem:identity}.
In principle, the result is a consequence of 
$\mathbb{C}^{\ast}$-localizations
of the Behrend functions as in~\cite[Proposition~3.3]{BBr}, 
\cite[Theorem~C]{WeiQin}, together
with the results in Proposition~\ref{prop:isom:para}. 
However, as we discussed in~\cite[Remark~2.4]{Todpara},
we are unable to find a symmetric perfect obstruction
theory on the moduli space of parabolic stable pairs, 
which prevents us to use the 
$\mathbb{C}^{\ast}$-localizations 
on the Behrend functions directly. 
Instead, we consider some
other moduli spaces which admit
 $\mathbb{C}^{\ast}$-equivariant symmetric perfect obstruction theories,
and apply $\mathbb{C}^{\ast}$-localizations to the Behrend 
functions on them. 
Then the result follows by comparing 
their Behrend functions with those on 
the moduli 
spaces of parabolic stable pairs.

\subsection{Deformations of sheaves}
In this subsection, we recall 
a result on deformations of sheaves on 
algebraic varieties given by Huybrechts-Thomas~\cite{HT2}, 
which will be used in the next subsection. 
Let $X$ be a smooth projective variety 
and $T$ an affine scheme with a closed point 
$0 \in T$. 
Suppose that we are given a $T$-flat 
coherent sheaf on $X\times T$, 
\begin{align*}
A \in \Coh(X\times T). 
\end{align*}
We would like to extend $A$ to a square zero 
extension $j \colon T \hookrightarrow \overline{T}$, i.e. 
there is an ideal $J\subset \oO_{\overline{T}}$
such that 
\begin{align}\label{square:zero}
\oO_T \cong \oO_{\overline{T}}/J, \quad 
J^2=0. 
\end{align}
Let us take 
the distinguished triangle in $D^b \Coh(X\times T)$, 
\begin{align}\label{dist:Q}
Q_A \to \dL j^{\ast}j_{\ast} A \to A. 
\end{align}
Here the right arrow is the adjunction, 
and we have denoted $\id_X \times j$ just by $j$
for simplicity. Following~\cite{HT2}, we construct
the morphism 
\begin{align*}
\pi_{A} \colon 
Q_{A} \to A \otimes_{\oO_T} J[1],
\end{align*}
in the following way. 
Let $h$ be the embedding 
\begin{align*}
h\colon 
X\times T \times X \times T \hookrightarrow
X\times T \times 
 X \times \overline{T},
\end{align*}
and $H$ the object
\begin{align*}
H \cneq \dL h^{\ast} h_{\ast}\Delta_{\ast}\oO_{X\times T},
\end{align*}
where $\Delta$ is the diagonal embedding
of $X\times T$. 
By~\cite[Subsection~3.1]{HT2}, there are distinguished
triangles on $X\times T$, 
\begin{align}\label{dist:H}
&\tau^{\le -1}H \to H \to \Delta_{\ast}\oO_{X\times T}, \\
\notag
&\tau^{\le -2}H \to \tau^{\le -1}H \stackrel{\pi_H}{\to} \Delta_{\ast}J[1].
\end{align} 
Applying Fourier-Mukai transforms of
$A$ for the triangle 
(\ref{dist:H}), we obtain 
the triangle (\ref{dist:Q}). 
Then the morphism $\pi_A$ is obtained by 
taking the Fourier-Mukai transform of $A$ with 
for the morphism $\pi_H$. 

On the other hand, suppose that the following morphism
on $X\times \overline{T}$ is given, 
\begin{align*}
e_{A} \colon j_{\ast}A \to j_{\ast}(A\otimes J)[1]. 
\end{align*}
Then we construct the morphism $\Psi_{e_A}$
to be the composition, 
\begin{align*}
\Psi_{e_A} \colon 
Q_A \to \dL j^{\ast} j_{\ast}A \stackrel{\dL j^{\ast} e_A}{\to}
\dL j^{\ast} j_{\ast}(A\otimes J)[1] \to A \otimes J[1]. 
\end{align*}
Here the left arrow is given by the left arrow of 
(\ref{dist:Q}), and the right arrow is given by the 
adjunction. 
Let us take the cone of $e_A$, 
\begin{align*}
j_{\ast}(A\otimes J) \to 
\overline{A} \to 
j_{\ast}A \stackrel{e_{A}}{\to} j_{\ast}(A\otimes J)[1].
\end{align*}
Note that $\overline{A}$ is a coherent sheaf
on $X\times \overline{T}$. 
By~\cite[Theorem~3.3]{HT2}, 
we have the following 
criteion for an object $\overline{A}$ to be 
a deformation of $A$: 
\begin{thm}{\bf(\cite[Theorem~3.3]{HT2}) }\label{thm:deform}
The sheaf $\overline{A} \in \Coh(X\times \overline{T})$
is flat over $\overline{T}$ with 
$\overline{A}|_{X\times T} \cong A$
if and only if we have the equality, 
\begin{align}\label{equal:Psi}
\pi_{A} =\Psi_{e_A}. 
\end{align}
\end{thm}

\subsection{Moduli spaces of simple sheaves and relative 
Quot schemes}
Let $X$ be a smooth projective Calabi-Yau 3-fold 
over $\mathbb{C}$, 
and $H\subset X$ a smooth 
and connected divisor. 
Recall that, giving a parabolic stable pair 
$(F, s)$ on $X$ is equivalent to giving a pair, 
\begin{align*}
N_{H/X}[-1] \to F, 
\end{align*}
where $N_{H/X}$ is the normal bundle of 
$H$ in $X$, 
satisfying a certain stability condition. 
(cf.~\cite[Proposition~3.9]{Todpara}.)
By taking the cone, we obtain the exact sequence of sheaves, 
\begin{align}\label{FEN}
0 \to F \to E \to N_{H/X} \to 0. 
\end{align}
Note that 
$N_{H/X} \cong \oO_H(H)$ is a simple sheaf, i.e. 
$\End(N_{H/X})=\mathbb{C}$. 
Together with the parabolic stability
of $(F, s)$, the sheaf $E$ is also a simple sheaf. 
(cf.~\cite[Corollary~3.10]{Todpara}.)

We consider the moduli space of simple sheaves on 
$X$, which we denote by $\mM$. The space $\mM$
is known to be an algebraic space locally of 
finite type over $\mathbb{C}$. 
(cf.~\cite{Inaba}.)
The universal sheaf is denoted by
\begin{align}\label{Univ:sheaf}
\eE \in \Coh(X \times \mM). 
\end{align}
Let
\begin{align}\label{tau}
\tau \colon 
\qQ \to \mM, 
\end{align}
be the algebraic space representing 
the relative 
Quot-functor for the family of 
simple sheaves (\ref{Univ:sheaf}).
Namely for each $[E] \in \mM$ the fiber of 
$\qQ \to \mM$ is Grothendieck's Quot-scheme
parameterizing quotient 
sheaves $E \twoheadrightarrow E'$.
We have the morphism, 
\begin{align}\label{iota}
\iota \colon 
M_n^{\rm{par}}(U, \beta) \to \qQ, 
\end{align} 
sending a parabolic stable 
pair $(F, s)$ to the surjection 
$E \twoheadrightarrow N_{H/X}$
given by the sequence (\ref{FEN}). 
We have the following lemma. 
\begin{lem}\label{lem:etale}
The morphism $\tau$ is \'{e}tale 
at any point in the image of $\iota$. 
\end{lem}
\begin{proof}
Let $(F, s)$ be a parabolic stable pair 
on $X$ and $E \twoheadrightarrow N_{H/X}$ a
point of $\qQ$ determined by (\ref{FEN}). 
Let $T$ be an affine scheme, $0 \in T$ a closed 
point, $j \colon T \hookrightarrow \overline{T}$ a square 
zero extension with an ideal $J\subset \oO_{\overline{T}}$
as in (\ref{square:zero}). 
Suppose that there is a commutative diagram, 
\begin{align}\label{diag:MQ}
\xymatrix{
T \ar[r]^{f} \ar[d]_{j} & \qQ \ar[d]^{\tau} \\
\overline{T} \ar[r]^{h} & \mM,
}
\end{align}
such that $f$ sends $0 \in T$ to 
$(E \twoheadrightarrow N_{H/X}) \in \qQ$. 
It is enough to show that, after replacing 
$T$ by an affine open neighborhood of $0 \in T$, 
the morphism $f$ uniquely extends to 
$\overline{f} \colon \overline{T} \to \qQ$
which commutes with all the arrows in (\ref{diag:MQ}). 

By pulling back $\eE \in \Coh(X\times \mM)$
to $T$ by the composition
$\tau \circ f \colon T \to \mM$, we obtain 
the sheaf
$\eE_T \in \Coh(X\times T)$, which is 
flat over $T$, and restricts to 
the sheaf $E$ on $X\times \{0\}$. 
The morphism $f$ corresponds to the exact 
sequence of sheaves on $X\times T$, 
\begin{align}\label{FEN2}
0 \to \fF_{T} \to \eE_T 
\to \nN_T \to 0, 
\end{align}
which restrict to the exact sequence (\ref{FEN})
on $X\times \{0\}$. 
Also the morphism $h$ corresponds to 
a $\overline{T}$-flat sheaf on $X\times \overline{T}$, 
\begin{align*}
\eE_{\overline{T}} \in \Coh(X\times \overline{T}), 
\end{align*}
which restricts to $\eE_{T}$ on $X\times T$. 
The existence of unique $\overline{f}$ is equivalent to the 
 existence of unique (up to isomorphism) exact sequence of 
sheaves on $X\times \overline{T}$, 
\begin{align}\label{FEN3}
0 \to \fF_{\overline{T}} \to \eE_{\overline{T}} \to 
\nN_{\overline{T}} \to 0, 
\end{align}
which restricts to the exact sequence (\ref{FEN2}) 
on $X\times T$. 

Let us consider the distinguished triangle 
on $X\times \overline{T}$, 
\begin{align*}
j_{\ast}(\eE_{T}\otimes J) \to \eE_{\overline{T}} \to j_{\ast}\eE_{T} 
\stackrel{e_{\eE}}{\to} j_{\ast}(\eE_{T} \otimes J)[1].  
\end{align*}
Since $\eE_{\overline{T}}$ is a deformation of 
$\eE_{T}$ to $X\times \overline{T}$, 
we have 
\begin{align}\label{deform:E}
\pi_{\eE_T}=\Psi_{e_{\eE}},
\end{align}
in the notation of the previous subsection, 
by Theorem~\ref{thm:deform}. 
Also we have the distinguished triangles
on $X\times \overline{T}$, 
\begin{align}\label{tri:jast}
\xymatrix{
j_{\ast} \fF_{T} \ar[r]^{}  & j_{\ast} \eE_{T} \ar[r]^{} \ar[d]^{e_{\eE}}
& j_{\ast} \nN_{T} \\
j_{\ast}(\fF_{T} \otimes J)[1] \ar[r] &
j_{\ast}(\eE_{T} \otimes J)[1] \ar[r] & 
j_{\ast}(\nN_{T} \otimes J)[1].
}
\end{align}
Since $\Hom(F, N_{H/X}[i])=0$ for $i=0, 1$, 
we have (after shrinking $T$ if necessary)
\begin{align}\label{vanish:FNJ}
\Hom_{X\times T}(\fF_{T}, \nN_{T} \otimes J[i])=0,
\end{align}
 for $i=0, 1$. 
By the distinguished triangle, 
\begin{align*}
\fF_{T} \otimes J[1] \to \dL j^{\ast} j_{\ast} \fF_{T}
\to \fF_{T}, 
\end{align*}
and the vanishing (\ref{vanish:FNJ}), 
we see that 
\begin{align*}
\Hom_{X\times \overline{T}}
(j_{\ast}\fF_{T}, j_{\ast}(\nN_T \otimes J)[i]) 
&\cong \Hom_{X\times T}(\dL j^{\ast} j_{\ast} \fF_{T}, \nN_{T}
\otimes J[i]) \\
&=0,
\end{align*}
for $i=0, 1$. 
Therefore there are unique morphisms, 
\begin{align*}
e_{\fF} &\colon j_{\ast}\fF_{T} \to j_{\ast} (\fF_{T} \otimes J)[1], \\
e_{\nN} &\colon j_{\ast} \nN_{T} \to j_{\ast}(\nN_{T} \otimes J)[1], 
\end{align*}
which make the diagram (\ref{tri:jast}) commutative. 
We need to show that $e_{\fF}$ and $e_{\nN}$ 
determine deformations of $\fF_{T}$ and $\nN_{T}$
to $X\times \overline{T}$. 
In order to see these, we consider the commutative diagram,
\begin{align*}
\xymatrix{
Q_{\fF_{T}} \ar[r] \ar[d]_{\pi_{\fF}-\Psi_{e_{\fF}}} & 
Q_{\eE_{T}} \ar[r] \ar[d]_{\pi_{\eE}-\Psi_{e_{\eE}}} &
Q_{\nN_{T}} \ar[d]_{\pi_{\nN}-\Psi_{e_{\nN}}} \\
\fF_{T}\otimes J[1] \ar[r] & \eE_{T} \otimes J[1] \ar[r]
& \nN_{T} \otimes J[1].
}
\end{align*}  
The middle arrow is zero by (\ref{deform:E}). 
Also we have 
$\Hom(Q_{\fF_{T}}, \nN_{T} \otimes J)=0$
since $\hH^i(Q_{\fF_T})=0$ for $i\ge 0$. 
Therefore by the above commutative diagram, 
we have 
\begin{align*}
\pi_{\fF}=\Psi_{e_{\fF}},
\end{align*}
which implies that 
$e_{\fF}$ determines a deformation of 
$\fF_{T}$ to $\fF_{\overline{T}}$
by Theorem~\ref{thm:deform}.
A similar argument shows that $e_{\nN}$
determines a deformation of $\nN_{T}$ to $\nN_{\overline{T}}$. 

By taking the cones $e_{\ast}$ for 
$\ast=\fF, \eE, \nN$ in the diagram (\ref{tri:jast}), 
 we obtain the 
exact sequence of sheaves, 
\begin{align*}
0 \to \fF_{\overline{T}} \to \eE_{\overline{T}} \to 
\nN_{\overline{T}} \to 0, 
\end{align*}
where $\fF_{\overline{T}}$ and $\nN_{\overline{T}}$
are deformations of $\fF_{T}$, $\nN_{\overline{T}}$
determined by $e_{\fF}$, $e_{\nN}$ respectively. 
It is straightforward to check that the above extension of
(\ref{FEN2}) to $X\times \overline{T}$ is unique 
up to isomorphisms, and we leave the readers to check the 
detail. 
\end{proof}

\subsection{Some identities of Behrend functions}
Let $C \subset U \subset X$ 
and $H\subset X$ be
as in the previous sections. 
Let 
\begin{align*}
\qQ_{U} \subset \qQ, \quad 
\mM_{U} \subset \mM,
\end{align*}
 be sufficiently small 
analytic neighborhoods of the images of 
$\iota$, $\tau \circ \iota$ respectively.
Here $\tau$, $\iota$
 are defined in (\ref{tau}), (\ref{iota}).
By Lemma~\ref{lem:etale}, the morphism $\tau$ restricts to a 
local immersion, 
\begin{align}\label{tau2}
\tau \colon \qQ_{U} \to \mM_U. 
\end{align} 
The arguments similar to Subsection~\ref{subsec:Jact}
and Subsection~\ref{subsec:compare:para}
show that $\mM_U$ and $\qQ_U$ admit 
$\mathbb{C}^{\ast}$-actions, 
where $\mathbb{C}^{\ast}$ is the subtorus (\ref{ctorus}),
 so that 
the morphisms (\ref{iota}) and (\ref{tau2})
are $\mathbb{C}^{\ast}$-equivariant. 
Let 
\begin{align*}
\nu_{\qQ}, \ \nu_{\qQ^{\mathbb{C}^{\ast}}}, \ 
\nu_{\mM}, \ \nu_{\mM^{\mathbb{C}^{\ast}}},
\end{align*}
be the Behrend functions on 
$\qQ$, $\qQ_U^{\mathbb{C}^{\ast}}$, 
$\mM$, $\mM_U^{\mathbb{C}^{\ast}}$ respectively. 
We have the following lemma:
\begin{lem}
For $(F, s) \in M_n^{\rm{par}}(U, \beta)^{\mathbb{C}^{\ast}}$
and the associated element 
$p=\iota(F, s) \in \qQ_U^{\mathbb{C}^{\ast}}$, we have 
\begin{align}\label{Beh3.2}
\nu_{\qQ}(p)=(-1)^{\dim T_p \qQ - \dim T_p \qQ_U^{\mathbb{C}^{\ast}}}
\nu_{\qQ^{\mathbb{C}^{\ast}}}(p). 
\end{align}
\end{lem}
\begin{proof}
Let us write 
$p=(E \twoheadrightarrow N_{H/X}) \in \qQ_U^{\mathbb{C}^{\ast}}$. 
Then by Lemma~\ref{lem:etale}, 
we have 
\begin{align}\label{nueq:1}
\nu_{\qQ}(p)=\nu_{\mM}([E]),
\quad \nu_{\qQ^{\mathbb{C}^{\ast}}}(p)=
\nu_{\mM^{\mathbb{C}^{\ast}}}([E]). 
\end{align}
Next note that, 
by~\cite{HT2}, the algebraic space 
$\mM$ admits a symmetric 
perfect obstruction theory determined by the universal 
sheaf (\ref{Univ:sheaf}). 
It is easy to check that the 
symmetric perfect obstruction theory on $\mM$, 
restricted to $\mM_{U}$, is 
$\mathbb{C}^{\ast}$-equivariant. 
Therefore the $\mathbb{C}^{\ast}$-localizations 
of the Behrend functions given 
in~\cite[Proposition~3.3]{BBr}, ~\cite[Theorem~C]{WeiQin}
are applied. 
The result is  
\begin{align}\label{nueq:2}
\nu_{\mM}(E)= (-1)^{\dim T_{[E]}\mM - \dim T_{[E]} \mM^{\mathbb{C}^{\ast}}}
\cdot 
\nu_{\mM^{\mathbb{C}^{\ast}}}(E). 
\end{align}
Again by Lemma~\ref{lem:etale}, we have
\begin{align}\label{eq:tangent}
\dim T_{p} \qQ =\dim T_{[E]} \mM, \quad 
\dim T_{p} \qQ_U^{\mathbb{C}^{\ast}}
= \dim T_{[E]}\mM^{\mathbb{C}^{\ast}}.  
\end{align}
The equality (\ref{Beh3.2}) 
follows from (\ref{nueq:1}), (\ref{nueq:2})
and (\ref{eq:tangent}). 
\end{proof}
Next we compare the Behrend functions on 
$\nu_{\qQ}$ and $\nu_{M^{\rm{par}}}$
under the morphism (\ref{iota}). 
(Recall that $\nu_{M^{\rm{par}}}$ is 
the Behrend function on $M_n(U, \beta)$.)
In what follows, for $E_1, E_2 \in \Coh(X)$, we write
\begin{align*}
\mathrm{hom}(E_1, E_2) &\cneq 
\dim \Hom(E_1, E_2), \\
\mathrm{ext}^1(E_1, E_2) &\cneq 
\dim \Ext_X^1(E_1, E_2).
\end{align*}
We have the following lemma:
\begin{lem}
For $(F, s) \in M_n^{\rm{par}}(U, \beta)$
with $p=\iota(F, s) \in \qQ$, we have the equality, 
\begin{align}\label{Beh3}
\nu_{\qQ}(p)=
(-1)^{\mathrm{ext}^1(N_{H/X}, N_{H/X})}
\cdot \nu_{M^{\rm{par}}}(F, s). 
\end{align}
\end{lem}
\begin{proof}
Let us write $p=(E \twoheadrightarrow N_{H/X}) \in \qQ$. 
We first note that, a point
of $\qQ_U$ near $p \in \qQ$
is represented by 
an exact sequence
\begin{align}\label{FEN'}
0 \to F' \to E' \to N' \to 0,
\end{align}
where $F'$, $E'$, $N'$ are small 
deformations of sheaves $F$, $E$, $N_{H/X}$
in (\ref{FEN}). 
Hence near $p\in \qQ_U$, 
we have the 1-morphism, 
\begin{align}\label{mor:QC}
\qQ_{U} \to \mM \times \cC oh(X), 
\end{align}
which sends the sequence (\ref{FEN'})
to $(N', F')$. 
Here $\cC oh(X)$ is the stack of 
all the objects in $\Coh(X)$, as in 
Subsection~\ref{subsec:Genera}. 
The fiber of the above 1-morphism 
at $(N', F')$ is an open subset of 
$\Ext_{X}^{1}(N', F')$. Since we have 
\begin{align*}
\Ext_{X}^i(N_{H/X}, F)
\cong \left\{ \begin{array}{cc}
\mathbb{C}^{\beta \cdot H}, & i=1, \\
0, & i\neq 1,
\end{array}  \right. 
\end{align*}
it follows that 
\begin{align*}
\Ext_{X}^1(N', F') \cong \mathbb{C}^{\beta \cdot H}.
\end{align*}
Therefore the morphism (\ref{mor:QC})
is a smooth morphism of relative dimension
$\beta \cdot H$. 

Let us consider the Behrend function on 
the RHS of (\ref{mor:QC}). 
It is easy to see that any 
small deformation of 
$N_{H/X} \cong \oO_{H}(H)$ is obtained as a line bundle 
on a divisor in $X$. Hence 
we see that the 
algebraic space 
$\mM$ is smooth of 
dimension $\mathrm{ext}^1(N_{H/X}, N_{H/X})$
at $[N_{H/X}] \in \mM$. 
It follows that we have 
\begin{align*}
\nu_{\mM}(N_{H/X})=(-1)^{\mathrm{ext}^1(N_{H/X}, N_{H/X})}.
\end{align*}
If we denote by 
$\nu_{\cC}$ the Behrend function on 
$\cC oh(X)$, the above arguments imply 
\begin{align}\label{Beh1}
\nu_{\qQ}(p) 
&=\nu_{\mM}(N_{H/X}) \cdot \nu_{\cC}(F) 
\cdot (-1)^{\beta \cdot H}, \\
\label{Beh2}
 &= (-1)^{\beta \cdot H + \mathrm{ext}^1(N_{H/X}, N_{H/X})} \cdot \nu_{\cC}(F).
\end{align}
In (\ref{Beh1}), 
 we have used the property of the
Behrend function under smooth morphisms
and products~\cite[Proposition~1.5]{Beh}.

On the other hand, there is a forgetting 1-morphism, 
\begin{align*}
M_n^{\rm{par}}(U, \beta) \to \cC oh(X), 
\end{align*}
sending $(F, s)$ to $F$. The fiber of the above morphism 
at $[F]$ is an open subset of 
$F \otimes \oO_{H} \cong \mathbb{C}^{\beta \cdot H}$, 
hence it is a smooth morphism of relative 
dimension $\beta \cdot H$. 
Therefore we have 
\begin{align}\label{Behpara}
\nu_{M^{\rm{par}}}(F, s)=(-1)^{\beta \cdot H} \cdot \nu_{\cC}(F). 
\end{align}
Combined with (\ref{Beh2}) and (\ref{Behpara}), we obtain the 
desired equality (\ref{Beh3}). 
\end{proof}

\subsection{Proof of Lemma~\ref{lem:identity}}\label{subsec:proof}
Finally in this section, we give a proof of 
Lemma~\ref{lem:identity}. 
\begin{proof}
Let $\sigma \colon 
\widetilde{U} \to U$ be the $m$-fold 
cyclic cover considered in the statement of Lemma~\ref{lem:identity}. 
Then for $m\gg 0$, we have 
\begin{align*}
M_n^{\rm{par}}(U, \beta)^{\mathbb{Z}/m\mathbb{Z}}
= M_n^{\rm{par}}(U, \beta)^{\mathbb{C}^{\ast}}. 
\end{align*}
Hence by Proposition~\ref{prop:isom:para}, 
for $(F, s) \in M_n^{\rm{par}}(U, \beta)^{\mathbb{C}^{\ast}}$, there is 
unique $(\widetilde{F}, \widetilde{s}) \in M_n^{\rm{par}}(\widetilde{U}, 
\widetilde{\beta})$ such that 
$(F, s)=\sigma_{\ast}(\widetilde{F}, \widetilde{s})$. 
Similarly to (\ref{FEN}), the pair $(\widetilde{F}, \widetilde{s})$
associates the exact sequence of 
sheaves on $\widetilde{U}$, 
\begin{align*}
0 \to \widetilde{F} \to \widetilde{E} \to 
N_{\widetilde{H}/\widetilde{U}} \to 0. 
\end{align*} 
Let $N'$ be a coherent sheaf on $X$, 
which is a
small deformation of $N_{H/X}$. 
Then we can uniquely lift 
$N'|_{U}$ to a sheaf $\widetilde{N}'$
on $\widetilde{U}$
so that $\widetilde{N}'$ is a
small deformation of the 
sheaf $N_{\widetilde{H}/\widetilde{U}}$ on 
$\widetilde{U}$. 
Let $\widetilde{\qQ}$ be the analytic 
local moduli space
parameterizing  
small deformations of surjections 
$\widetilde{E} \twoheadrightarrow N_{\widetilde{H}/\widetilde{U}}$,
\begin{align}\label{E'N'}
\widetilde{E}' \twoheadrightarrow \widetilde{N}', 
\end{align}
where (\ref{E'N'}) is a surjection of coherent
sheaves on $\widetilde{U}$, and 
$\widetilde{N}'$ is 
a lift of a small deformation of 
$N_{H/X}$ restricted to $U$ as above. 
An argument similar to the 
 proof of 
Lemma~\ref{lem:nat:comp}
shows that 
there is a natural morphism, 
\begin{align}\label{sigma:ast:Q}
\sigma_{\ast} \colon 
\widetilde{\qQ}
\to \qQ_{U}^{\mathbb{Z}/m\mathbb{Z}}
\end{align}
satisfying that 
\begin{align*}
\sigma_{\ast}(\widetilde{E} \twoheadrightarrow 
N_{\widetilde{H}/\widetilde{U}}) =(E \twoheadrightarrow N_{H/X}). 
\end{align*}
Also a proof similar to Proposition~\ref{prop:isom:para}
shows that the morphism (\ref{sigma:ast:Q})
is an isomorphism onto 
connected components of $\qQ_{U}^{\mathbb{Z}/m\mathbb{Z}}$. 
Since we have 
$\qQ_{U}^{\mathbb{Z}/m\mathbb{Z}}=\qQ_{U}^{\mathbb{C}^{\ast}}$
for $m\gg 0$, it follows that 
\begin{align}\label{Beh3.5}
\nu_{\widetilde{\qQ}}(\widetilde{E} \twoheadrightarrow 
N_{\widetilde{H}/\widetilde{U}})
=\nu_{\qQ^{\mathbb{C}^{\ast}}}(E\twoheadrightarrow N_{H/X}). 
\end{align}
Also similarly to (\ref{Beh3}), we have the equality, 
\begin{align}\label{Beh4}
\nu_{\widetilde{\qQ}}(\widetilde{E}
\twoheadrightarrow N_{\widetilde{H}/\widetilde{X}})=
(-1)^{\mathrm{ext}^1(N_{H/X}, N_{H/X})}
\cdot \nu_{\widetilde{M}^{\rm{par}}}(\widetilde{F}, \widetilde{s}). 
\end{align} 
By (\ref{Beh3}), (\ref{Beh3.2}),
 (\ref{Beh3.5}) and (\ref{Beh4}), we obtain 
\begin{align}\label{Beh:para}
\nu_{\widetilde{M}^{\rm{par}}}(\widetilde{F}, \widetilde{s})
=(-1)^{\dim T_{\widetilde{p}}\widetilde{\qQ} - \dim T_{p} \qQ}
\cdot 
\nu_{M^{\rm{par}}}(F, s),
\end{align}
where $\widetilde{p}=(\widetilde{E}
\twoheadrightarrow N_{\widetilde{H}/\widetilde{U}}) \in \widetilde{\qQ}$. 
 
Let us evaluate 
$\dim T_{\widetilde{p}}\widetilde{\qQ} - \dim T_{p} \qQ$. 
Since the morphism (\ref{mor:QC}) is a smooth morphism of 
relative dimension $\beta \cdot H$, we have 
\begin{align}\label{TQ1}
\dim T_{p} \qQ = \mathrm{ext}^1(F, F)- \mathrm{hom}(F, F)
+ \mathrm{ext}^1(N_{H/X}, N_{H/X}) + \beta \cdot H. 
\end{align}
Similarly we have 
\begin{align}\label{TQ2}
\dim T_{\widetilde{p}} \widetilde{\qQ}
=\mathrm{ext}^1(\widetilde{F}, \widetilde{F})- 
\mathrm{hom}(\widetilde{F}, \widetilde{F})
+ \mathrm{ext}^1(N_{H/X}, N_{H/X}) + \widetilde{\beta} \cdot \widetilde{H}.
\end{align}
We have 
\begin{align}\notag
&\mathrm{ext}^1(F, F)- \mathrm{hom}(F, F)-
\mathrm{ext}^1(\widetilde{F}, \widetilde{F})+
\mathrm{hom}(\widetilde{F}, \widetilde{F}) \\
\notag
&=\mathrm{ext}^1(\sigma_{\ast}\widetilde{F}, \sigma_{\ast}\widetilde{F})- 
\mathrm{hom}(\sigma_{\ast}\widetilde{F}, \sigma_{\ast}\widetilde{F})-
\mathrm{ext}^1(\widetilde{F}, \widetilde{F})+
\mathrm{hom}(\widetilde{F}, \widetilde{F}) \\
\label{ext:difference}
&=\sum_{0\neq g \in \mathbb{Z}/m\mathbb{Z}}
\{ 
\mathrm{ext}^1(g_{\ast}\widetilde{F}, \widetilde{F})- \mathrm{hom}(g_{\ast}\widetilde{F}, \widetilde{F}) \}
\end{align}
By the Riemann-Roch theorem and the Serre duality, we have 
\begin{align*}
\mathrm{ext}^1(g_{\ast}\widetilde{F}, \widetilde{F})- \mathrm{hom}(g_{\ast}\widetilde{F}, \widetilde{F})
&= \mathrm{ext}^1(\widetilde{F}, g_{\ast}\widetilde{F})- \mathrm{hom}(\widetilde{F}, g_{\ast}\widetilde{F}) \\
&= \mathrm{ext}^1((-g)_{\ast}\widetilde{F}, \widetilde{F})- 
\mathrm{hom}((-g)_{\ast}\widetilde{F}, \widetilde{F}).
\end{align*}
Therefore (\ref{ext:difference}) is an even integer if $m$ is an odd integer. 
By (\ref{TQ1}), (\ref{TQ2}), we have 
\begin{align*}
\dim T_{\widetilde{p}}\widetilde{\qQ} - \dim T_{p} \qQ
\equiv \widetilde{\beta} \cdot \widetilde{H} -
\beta \cdot H, \quad (\mathrm{mod} \ 2). 
\end{align*} 
Combined with (\ref{Beh:para}), we obtain 
(i) of Lemma~\ref{lem:identity}. 
The result of (ii) follows from (i) and (\ref{Behpara}).
\end{proof}

Institute for the Physics and 
Mathematics of the Universe, 

Todai Institute for Advanced Studies (TODIAS), 
University of Tokyo,

5-1-5 Kashiwanoha, Kashiwa, 277-8583, Japan.

\textit{E-mail address}: yukinobu.toda@ipmu.jp

\end{document}